\documentclass[a4paper, 11pt, headings = small, abstract]{scrartcl}
\usepackage[english]{babel}

\usepackage{amscd,amsmath,amsthm,array,hhline, enumerate, booktabs,fancybox,calc,textcomp,xcolor,graphicx,xspace,nicefrac,stmaryrd,arcs,listings,pifont,thmtools, setspace, amssymb, comment,mathtools,tabularx}

\usepackage[semibold]{libertine}
\usepackage[upint,amsthm]{libertinust1math}
\newcommand{\coloneqq}{\coloneq}
\usepackage[scr=boondoxupr, scrscaled = 1.0, cal =stixtwoplain]{mathalpha}

\usepackage[scaled=0.95]{inconsolata}
\usepackage{accents}

\usepackage{enumitem}
\usepackage{etoolbox}
\usepackage[normalem]{ulem}
\usepackage{geometry}
\usepackage{float}

\usepackage{bm}
\usepackage{tikz}
\usepackage[outdir =./]{epstopdf}
\usepackage[citestyle = numeric-comp,bibstyle = numeric, giveninits=true, natbib = true, backend = bibtex,maxbibnames=99, doi=false, isbn = false]{biblatex}
\usepackage{xurl}
\definecolor{myteal}{RGB}{0 123 137}
\usepackage[colorlinks=true, allcolors = myteal]{hyperref}
\usepackage[hyphenbreaks]{breakurl}

\numberwithin{equation}{section}

\swapnumbers

\newtheorem{satz}{Satz}[section]

\newtheorem{theorem}[satz]{Theorem}
\newtheorem{proposition}[satz]{Proposition}
\newtheorem{corollary}[satz]{Corollary}
\newtheorem{lemma}[satz]{Lemma}

\setkomafont{sectioning}{\normalcolor\bfseries}
\setkomafont{descriptionlabel}{\normalcolor\bfseries}
\setkomafont{section}{\large}
\setkomafont{subsection}{\normalsize}
\setkomafont{subsubsection}{\normalsize}
\setkomafont{paragraph}{\normalsize}
\setkomafont{subparagraph}{\normalsize}
\setkomafont{author}{\large}
\setkomafont{title}{\large\bfseries}
\setkomafont{date}{\normalsize}

\DeclareMathOperator{\E}{{\mathbb E}}

\DeclareMathOperator{\R}{{\mathbb R}}

\DeclareMathOperator{\N}{{\mathbb N}}

\DeclareMathOperator{\PP}{{\mathbb P}}

\DeclareMathOperator{\supp}{supp}

 \DeclareMathOperator{\Id}{Id}
\DeclareMathOperator*{\argmin}{arg\,min}\DeclareMathOperator*{\argmax}{arg\,max}
\DeclareMathOperator*{\argminalt}{\overbar{arg}\,min}\DeclareMathOperator*{\argmaxalt}{\overbar{arg}\,max}
\DeclareMathOperator{\Var}{Var} \DeclareMathOperator{\Cov}{Cov}

\providecommand{\eps}{\varepsilon}
\renewcommand{\phi}{\varphi}
\renewcommand{\theta}{\vartheta}
\newcommand{\eqqcolon}{\eqcolon}
\providecommand{\abs}[1]{\lvert #1 \rvert}
\providecommand{\norm}[1]{\lVert #1 \rVert}

\providecommand{\babs}[1]{{\Bigl\lvert #1 \Bigr\rvert}}
\providecommand{\scapro}[2]{\langle #1,#2 \rangle}
\providecommand{\bscapro}[2]{\Bigl\langle #1,#2 \Bigr\rangle}

\providecommand{\ceil}[1]{\lceil #1 \rceil}

\renewcommand{\Re}{\operatorname{Re}}

\newcommand{\modelspace}{(0,1)}

\newcommand{\ko}{k_{\bullet}}
\newcommand{\koo}{k_{\bullet}^0}
\newcommand{\intv}{\mathcal J_{\tau,\delta,\eta}}

\newcommand{\e}{\mathrm{e}}
\newcommand{\one}{\bm{1}}
\newcommand{\lowell}{\underbars{\theta}}
\newcommand{\upell}{\overbar{\theta}}
\newcommand{\difflim}{\theta^\ast}

\renewcommand{\hat}{\widehat}
\renewcommand{\tilde}{\widetilde}
\newcommand{\overbar}[1]{\mkern 1.5mu\overline{\mkern-1.5mu#1\mkern-1.5mu}\mkern 1.5mu}
\renewcommand{\bar}{\overbar}
\newcommand{\underbars}[1]{\mkern 1.5mu\underline{\mkern-1.5mu#1\mkern-1.5mu}\mkern 1.5mu}

\allowdisplaybreaks

\newcommand*\diff{\mathop{}\!\mathrm{d}}

\makeatletter
\newcommand*{\sumc}{%
  \DOTSB
  \mathop{
    \mathchoice
      {\rlap{\kern.25em{$\circleddash\scriptstyle\delta$}}{\sum}}
      {\vcenter{\rlap{\kern.2em{$\scriptstyle{\circleddash}\scriptscriptstyle{\delta}$}}}{\sum}}
      {\sum}{\sum}
  }\slimits@
}
\makeatother

\usepackage[T1]{fontenc}
\makeatletter
\newcommand{\specificthanks}[1]{\@fnsymbol{#1}}
\makeatother

\definecolor{Purple}{RGB}{103 58 183}
\definecolor{Green}{rgb}{.13 .55 .13}

\makeatother

\bibliography{spdes}

\allowdisplaybreaks

\makeatother
\title{\fontsize{16}{19} \selectfont Change point estimation for a stochastic heat equation}
\author{Markus Reiß\thanks{Humboldt-Universität zu Berlin, Institut für Mathematik, Unter den Linden 6, 10099 Berlin, Germany. \newline Email: \href{mailto:mreiss@math.hu-berlin.de}{mreiss@math.hu-berlin.de}} \and Claudia Strauch\thanks{Heidelberg University, Institute for Mathematics, Im Neuenheimer Feld 205, 69120 Heidelberg, Germany. \newline Email: \href{mailto:strauch@math.uni-heidelberg.de}{strauch@math.uni-heidelberg.de}} \and Lukas Trottner\thanks{University of Birmingham, School of Mathematics, Birmingham B15 2TT, UK. \newline Email: \href{mailto:l.trottner@bham.ac.uk}{l.trottner@bham.ac.uk}}}

\begin{document}
\maketitle
	
\begin{abstract}
We study a change point model based on a stochastic partial differential equation (SPDE) corresponding to the heat equation governed by the weighted Laplacian $\Delta_\vartheta = \nabla\vartheta\nabla$, where $\theta=\theta(x)$ is a space dependent diffusivity. As a basic problem, the domain $(0,1)$ is considered with a piecewise constant diffusivity with a jump at an unknown point $\tau$.
Based on local measurements of the solution in space with resolution $\delta$ over a finite time horizon, we construct a simultaneous M-estimator for the diffusivity values and the change point.
The change point estimator converges with rate $\delta$, while the diffusivity constants can be recovered with convergence rate $\delta^{3/2}$.
Furthermore, when the diffusivity parameters are known and the jump height vanishes as the spatial resolution tends to zero, we derive a limit theorem for the change point estimator and identify the limiting distribution.
For the mathematical analysis, a precise understanding of the SPDE with discontinuous $\theta$, tight concentration bounds for quadratic functionals in the solution, and a generalisation of classical M-estimators are developed.
\end{abstract}


\section{Introduction}
{\color{black} For many mathematical models developed for the quantitative analysis of real-world problems, it has proven beneficial to incorporate randomness, thereby accounting for -- to some extent hidden -- random dynamics, measurement errors and other external influences. For a whole class of highly relevant models, namely those based on partial differential equations (PDEs), the inclusion of noise (randomness) in the dynamics naturally results in considering SPDEs. While the corresponding models often describe phenomena such as diffusion or transport more adequately than their deterministic counterparts, the mathematical analysis frequently turns out to be very challenging.
Recently, statistics for SPDEs has received a lot of attention, see \cite{cial18} and the website \cite{stats4spdes} for an overview. Nonparametric estimation for SPDEs based on local measurements was introduced in \cite{altmeyer21} and successfully applied to estimate the diffusivity of actin concentration in a cell-repolarisation model \cite{altmeyer22b}. Generalisations to semilinear equations and multiplicative noise are available \cite{altmeyer23,janak23}. 
To the best of our knowledge, the problem of estimating an interface or change point in diffusivity has not yet been  treated, although this issue is critical in practical applications involving heat transfer through heterogeneous media, where abrupt changes in thermal properties often occur at material interfaces or due to phase transitions. Accurate identification of these points is essential for predicting heat distribution and optimising material design. Since real-world systems often involve random fluctuations, introducing stochastic elements into the heat equation allows for more realistic modelling of uncertainties such as thermal noise or material irregularities, making it natural to study a stochastic heat equation with discontinuous diﬀusion coeﬃcients.

The problem of identifying change points from independent observations has a long history dating back to \cite{wald45} and \cite{page54}. An exhaustive account of classical change point problems is given in \cite{csorgo97}. For an up-to-date exposition of the problem of finding univariate mean change points based on independent observations with piecewise constant mean, the reader may consult \cite{wayuri22}. In recent years, there has been a particular interest in change point analysis in more complex settings, such as high-dimensional data and data in general metric spaces. We refer to \cite{enikeeva19, li24, liu21, wang21, wang18, xu24, wang22} and \cite{dubey20} for some state-of-the-art contributions in these areas. 

We extend this line of research by studying a change point model in an infinite-dimensional setting, where we consider a parabolic SPDE, which has a piecewise constant diffusivity coefficient and is driven by space-time white noise. The methodology developed for this problem gives fundamental insights, useful for various applications, e.g., heat conduction in a medium, consisting of two different materials with unknown interface. 
Assuming that our data are given by observing the solution to the SPDE locally in space at resolution $\delta$ and continuously in time, we introduce a novel M-estimator. Our methodology employs a modified likelihood-based approach that incorporates a nuisance parameter to improve the estimation of the diffusivity parameter and the change point. 
We establish optimal convergence rates for our estimators, showing that the estimates satisfy error bounds of order $\mathcal{O}_{\PP}(\delta^{3/2})$ for the diffusivity and $\mathcal{O}_{\PP}(\delta)$ for the change point. In addition, we derive a limit theorem for the change point estimator, illustrating its behaviour under a faint signal condition given by a vanishing jump height in the diffusivity. These contributions not only provide a better understanding of change point detection in SPDEs but also could serve as a starting point for developing robust tools for practical applications in the analysis of heat conduction through materials with unknown interfaces.

The paper is organised as follows.
The rigorous statement of the basic SPDE model and the change point problem are given in Section \ref{sec:back}. In particular, we detail the observation scheme and discuss both the existence of a weak solution to the SPDE and key probabilistic properties required for the statistical analysis. Section \ref{sec:est} motivates the estimation approach with a discussion of a related problem in a Gaussian signal plus white noise model, which serves as a bridge to the subsequent statistical analysis.
Section \ref{sec:main} derives our estimators and provides a basic insight into their concentration analysis.
Section \ref{sec:anaconsconv} addresses the non-vanishing jump height regime and contains our main result on convergence rates and consistency statements, while Section \ref{sec:clt} provides the change point limit theorem in the vanishing jump height regime, along with the technical tools and its proof. 
Section \ref{sec:disc} contains the discussion and conclusion, summarising our contributions, comparing our results with existing literature, and giving an outline of issues for future research.
While the main steps in the proofs are presented together with the results, all more technical arguments are delegated to the Appendix.
More specifically, Appendix \ref{app:proof_ana} and Appendix \ref{app:proof_conc} collect analytical results and proofs related to the concentration analysis, respectively.
Appendix \ref{app:proof_rem} contains the remaining proofs for the results given in Section \ref{sec:est}, while Appendix \ref{app:proof_clt} collects the proofs for Section \ref{sec:clt}.}

\section{Background}\label{sec:back}
As a basic model,  we consider the weighted Laplace operator $\Delta_\vartheta=\nabla\theta\nabla$ on $(0,1)$ (i.e., $\partial_x\theta\partial_x$) with Dirichlet boundary conditions  and diffusivity
\begin{equation} \label{eq:theta}
	\vartheta(x)  =\vartheta_-{\mathbf 1}_{(0,\tau)}(x)+\vartheta_+{\mathbf 1}_{[\tau,1)}(x),\quad x\in (0,1),
\end{equation}
where $\tau\in (0,1)$,
\begin{equation}\label{eq:unif_ell}
	\theta_- , \theta_+ \in [\lowell, \upell] \text{ for some } \lowell, \upell\in (0,\infty),
\end{equation}
and the operator is understood in the distributional sense at $x = \tau$, see Section \ref{sec:general} below. 
Our interest is in estimating the change point $\tau$ as well as the diffusivity constants $\theta_+,\theta_-$ from observing the corresponding SPDE
\begin{equation} \label{eq:spde}
	\begin{cases} \diff X(t)=\Delta_\vartheta X(t) \diff t+ \diff W(t), & t\in (0,T]\\ X(0)=X_0, &\\ X(t)\vert_{\{0,1\}} = 0,& t \in (0,T]\end{cases}
\end{equation}
with space-time white noise $\dot W$, i.e., $W = (W_t)_{t \in [0,T]}$ is a cylindrical Brownian motion on $L^2(\modelspace)$, and deterministic initial condition $X_0\in C([0,1])$.

\subsection{Observation scheme}
Our data are given by observing the solution to \eqref{eq:spde} locally in space and continuously in time. More precisely, we adopt the local observation scheme introduced in \cite{altmeyer21} and observe at $n$ equidistant points $x_i=(i-1/2)/n$, $i=1,\ldots,n$, the local averages
\begin{equation} \label{eq:locpro}
	X_{\delta,i}(t) \coloneqq \scapro{X(t)}{K_{\delta,i}}\quad \text{ and }\quad X_{\delta,i}^\Delta(t) \coloneqq \scapro{X(t)}{\Delta K_{\delta,i}}\quad \text{ for } t\in[0,T],
\end{equation}
for some localised kernel functions $K_{\delta,i}(x)=\delta^{-1/2}K(\delta^{-1}(x-x_i))$, $i=1,\ldots,n$, with $\delta>0$ and a smooth function $K$, satisfying $\supp(K)\subset[-1/2,1/2]$, $\norm{K}_{L^2}=1$.
An illustration of the observation scheme is given in Figure \ref{fig:obs}. The kernels $K_{\delta,i}$ model for instance the \textit{point-spread function} in microscopy, see  \cite{altmeyer22b} for a concrete application to SPDE activator-inhibitor models in cell motility. For the resolution level, $\delta=n^{-1}$ is assumed so that the observations $X_{\delta,i}(t)$ form non-overlapping local averages of the solution.
We already emphasise that $X_{\delta,i}(t)$ will nevertheless be correlated due to the global SPDE dynamics, which will require a precise analysis in the sequel.
Throughout the paper, we allow the diffusivity parameters to vary with $\delta$, i.e., we consider \eqref{eq:theta} with $\vartheta_\pm=\vartheta_\pm(\delta)$ obeying the bounds \eqref{eq:unif_ell} uniformly in $\delta$.
We shall study convergence rates and asymptotic distributions of our estimators in the asymptotic regime $\delta\to 0$, that is, $n\to\infty$, keeping the observation time $T$ fixed. Quantities possibly depending on $\delta$ will then often be denoted by an index $\delta$.
\begin{figure}[h]
\centering
\includegraphics[scale = 0.9]{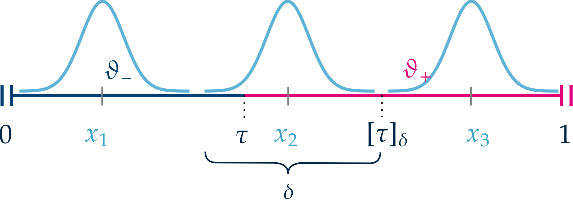}
\caption{Illustration of the observation scheme for three local observations represented by kernel graphs in light blue. $[\tau]_\delta$ denotes the best upper approximation of $\tau$ on the $\delta$-grid.}
\label{fig:obs}
\end{figure}\vspace*{-1cm}

\subsection{General setting}\label{sec:general}
Given observations of the SPDE \eqref{eq:spde}, our interest is in estimating the parameters characterising the diffusivity \eqref{eq:theta}.
This section provides the basis for our subsequent investigations by proving the existence of a weak solution of \eqref{eq:spde} and discussing some basic probabilistic properties that will be needed for the estimation approach.

Recall that an $L^2(\modelspace)$-valued and predictable process $(X_t)_{t \in [0,T]}$ is called \emph{weak solution} of \eqref{eq:spde} if $t \mapsto X(t)$ is $\PP$-a.s.\ Bochner integrable and if, for all $z \in D(\Delta_\theta^\ast)$ and all $t \in [0,T]$,
\begin{equation}\label{eq:weaksol}
	\scapro{X(t)}{z} = \scapro{X_0}{z} + \int_0^t \scapro{X(s)}{\Delta_\theta^\ast z} \diff{s} + \scapro{W(t)}{z}, \quad \PP\text{-a.s.},
\end{equation}
where $\Delta^\ast_\theta$ denotes the adjoint of $\Delta_{\theta}$ and $D(\Delta^\ast_\theta)$ is its domain. Let us start by recalling some essential facts on the divergence-form operator $\Delta_\vartheta$ from the literature that are rooted in the theory of Dirichlet forms (cf.\ \cite{fukushima11}).
Denote by $\lambda$ the Lebesgue measure on $(\modelspace, \mathcal{B}(\modelspace))$, and let $L^2(\modelspace)$ be the Hilbert space of square-integrable functions on $\modelspace$, equipped with the scalar product $\langle u,v \rangle = \int_{\modelspace} uv \diff{\lambda}$.
Denote by $H^k(\modelspace)$ the $L^2$-Sobolev spaces for $k \in \N$, and let $H^1_0(\modelspace)$ be the closure of $C_c^\infty(\modelspace)$ in $H^1(\modelspace)$.
The elliptic divergence-form operator
\[\begin{cases} D(\Delta_\vartheta) = \{u \in H^1_0(\modelspace): \Delta_\vartheta u \in L^2(\modelspace)\},\\
	\Delta_\vartheta u = \nabla \vartheta \nabla u,
\end{cases} \]
is induced by the Dirichlet form
\[\begin{cases} D(\mathcal{E}_\vartheta) =  H^1_0(\modelspace),\\
	\mathcal{E}_\vartheta(u,v) = \int_{\modelspace} \vartheta \nabla u \nabla v \diff{\lambda},
\end{cases} \]
via the relation
\[\mathcal{E}_\vartheta(u,v) = -\langle \Delta_\vartheta u,v \rangle, \quad (u,v) \in D(\Delta_\vartheta) \times D(\mathcal{E}_\vartheta).\]
Because of $\theta(x) \geq \lowell > 0$, $(\Delta_\vartheta, D(\Delta_\vartheta))$ is a negative definite self-adjoint operator on $L^2(\modelspace)$ that generates a strongly continuous symmetric semigroup $(S_\vartheta(t))_{t \in [0,T]}$  on $L^2(\modelspace)$ via $S_\vartheta(t) = \exp(t \Delta_\vartheta )$, cf.\ \cite[Theorem 2.1]{etore06}.
For each $t> 0$, $S_\vartheta(t)$ has an $L^2(\modelspace \times \modelspace)$ density kernel or Green function $p^\vartheta_t$ satisfying the classical off-diagonal Aronson estimate \cite{aronson68},
\[
p^\vartheta_t(x,y) \leq \frac{c_{\lowell,\kappa}}{\sqrt{t}}\exp\Big(-\frac{\lvert x - y \rvert^2}{4(1+\kappa)\upell t} \Big), \quad t > 0,\ x,y \in \modelspace,
\]
where $\kappa > 0$ can be  chosen arbitrarily, and $c_{\lowell,\kappa}$ is a constant that depends only on $\kappa$ and $\lowell$, cf.\ \cite[Corollary 3.2.8]{davies90} for the near optimal constants above.
For any $t > 0$, the kernel $p^\vartheta_t$ is square-integrable and thus $S_\vartheta(t)$ is a Hilbert--Schmidt operator, \cite[Section A.6]{bakry14}.
This implies that $S_{\vartheta}(t)$ has a discrete spectrum $\sigma(S_\vartheta(t))$ and, since $\exp(-t\sigma(-\Delta_\vartheta)) \subset \sigma(S_\vartheta(t))$, see, e.g., \cite[Theorem IV.3.6]{engel00}, it follows that the spectrum of $-\Delta_\vartheta$ is discrete as well.
Therefore, we may choose an orthonormal basis $(e_k)_{k \in \N}$ of $L^2(\modelspace)$ consisting of eigenvectors of $-\Delta_{\theta}$ with corresponding non-negative sequence of eigenvalues $(\lambda_k)_{k \in \N}$.
In fact, $(\lambda_k)_{k \in \N}$ is bounded from below by some $\underbars{\lambda} > 0$, depending only on the parameter $\lowell$.
To see this, note that, for any $u \in D(\Delta_\theta)$, by Poincaré's inequality on the bounded domain $(0,1)$,
\[\langle -\Delta_\theta u, u \rangle \geq \lowell \lVert \nabla u \rVert^2 \geq c\lowell \lVert u \rVert^2,\]
for some constant $c > 0$. We may therefore take $\underbars{\lambda} = c\lowell$ as a universal constant.
In particular, $\Delta_\theta^{-1}$ exists as a bounded linear operator from $L^2((0,1))$ to $D(\Delta_\theta)$, and it holds $\lVert \Delta_\theta^{-1} \rVert \leq \underbars{\lambda}^{-1}$.
We apply functional calculus to $\Delta_\theta$. To this end, we note that, for any measurable function $\psi$ on $\R_+$, the operator $\psi(-\Delta_\vartheta)$ is given by
\begin{equation}
	\begin{split}\label{eq:transform}
		\psi(-\Delta_\theta)z&= \sum_{k\in \N} \psi(\lambda_k)\langle e_k,z \rangle e_k, \\
		& \quad z \in D(\psi(-\Delta_\vartheta)) = \big\{z \in L^2(\modelspace): \sum_{k \in \N} \psi(\lambda_k)^2 \langle e_k,z \rangle^2 < \infty \big\}.
	\end{split}
\end{equation}
The following basic result clarifies the existence of a solution and provides a representation of the measurement process $X_{\delta,i}$ introduced in \eqref{eq:locpro}.
Note that the functions $K_{\delta,i}$ are part of the observation scheme, and observe that $\supp (K_{\delta,i})\cap\supp(K_{\delta,j})=\varnothing$ holds for all $i\not=j$, i.e., the local observation windows do not overlap. 
Let us also define $\ko =\ko(\delta) \coloneqq \lceil \tau/\delta \rceil$, implying that the change point $\tau\in {\supp (K_{\delta,\ko})}$.

\begin{proposition} \label{prop:weaksol}
	The unique weak solution of \eqref{eq:spde} is given by the mild solution process
	\begin{equation}\label{eq:mild}
		X(t) = S_\theta(t)X_0 + \int_0^t S_{\theta}(t-s) \diff{W_s}, \quad t \in [0,T].
	\end{equation}
	Moreover, for any $i =1,\ldots,n$ and $t \in [0,T]$, we have the $\PP$-a.s.\ representations
	\begin{equation}\label{eq:mild2}
		X_{\delta,i}(t)=
		\begin{cases}
			\langle X_0, K_{\delta,i} \rangle + \int_0^t \vartheta_-(\delta) X_{\delta,i}^\Delta(s) \diff{s}+ B_{\delta,i}(t),& \text{ if } x_i+\delta/2\le\tau,\\
			\langle X_0, K_{\delta,i} \rangle + \int_0^t \vartheta_+(\delta) X_{\delta,i}^\Delta(s)\diff{s}+  B_{\delta,i}(t),& \text{ if } x_i-\delta/2\ge\tau,\\
			\langle  S_\theta(t) X_0, K_{\delta,i} \rangle + \int_0^t \int_0^s \scapro{\Delta_\theta S_\vartheta(s-u) K_{\delta,i}}{\diff W(u)} \diff{s} + B_{\delta,\ko}(t),& \text{ if }\abs{x_i-\tau}<\delta/2,\end{cases}
	\end{equation}
	where $(B_{\delta,i})_{i=1,\ldots,n}$ with $B_{\delta,i}(\cdot) = \langle W(\cdot), K_{\delta,i} \rangle$ is a vector of independent scalar Brownian motions. 
\end{proposition}
\begin{proof}
	See Appendix \ref{app:proof_ana}.
\end{proof}
In the following, we turn to our main question of change point estimation from observations \eqref{eq:locpro}.
The analysis in \cite{altmeyer21,altmeyer22} shows that the contribution of the initial condition to the statistics is asymptotically negligible. To avoid lengthy additional calculations without new structural insights, we consider a zero initial condition  $X_0 \equiv 0$ in the sequel.
Furthermore, to shorten notation, we will frequently use the convention $[n]\coloneqq \{1,\ldots,n\}$, $n\in\N$.

\subsection{Motivation: A related model problem}\label{sec:est}
To motivate our statistical approach and as a benchmark for our later results, we briefly discuss the situation in a simpler Gaussian signal plus white noise model that is treated in \cite[Chapter VII, Section 2]{ibra81}. For a related discussion of spatial change point estimation in time-homogeneous SDE models, we refer to \cite[Chapter 3, Section 4]{kuto04}.
Assume that we observe
\[ \diff Y(x)=\vartheta(x)\diff x+\sigma(x)\diff B(x),\quad x\in[0,1],\]
with unknown $\vartheta$ of the form \eqref{eq:theta}, a known space-dependent noise level $\sigma\in L^2((0,1))$, and a scalar Brownian motion $B$.
The log-likelihood with respect to Brownian motion is given by
\begin{align*}
		\ell(\vartheta_-,\vartheta_+,\tau)&=\vartheta_-\int_0^\tau \sigma^{-2}(x)\diff Y(x)-\frac{\vartheta_-^2}2 \int_0^\tau \sigma^{-2}(x)\diff x \\
		&\quad +\vartheta_+\int_\tau^1 \sigma^{-2}(x)\diff Y(x) -\frac{\vartheta_+^2}2 \int_\tau^1 \sigma^{-2}(x)\diff x.
\end{align*}
The MLE $(\hat\vartheta_-,\hat\vartheta_+,\hat\tau)$ exists and yields the change point estimator
\[
	\hat\tau=\argmax_\tau \bigg\{\frac{(\int_0^{\tau} \sigma^{-2}(x)\diff Y(x))^2}{\int_0^{\tau} \sigma^{-2}(x)\diff x} +\frac{(\int_{\tau}^1 \sigma^{-2}(x)\diff Y(x))^2}{\int_{\tau}^1 \sigma^{-2}(x)\diff x}\bigg\}.
\]
In the case where $\vartheta_-,\vartheta_+$ are known and assuming $\eta\coloneqq \vartheta_+-\vartheta_->0$, we can subtract $\ell(\cdot,\cdot,\tau^0)$ from $\ell$, resulting in
\[
\hat\tau=\argmax_\tau \bigg\{(\vartheta_+-\vartheta_-)\int_\tau^{\tau^0}
	\sigma^{-2}(x)\diff Y(x)-\frac{\vartheta_+^2-\vartheta_-^2}{2}\int_\tau^{\tau^0}\sigma^{-2}(x)\diff x\bigg\}.
\]
Analysing this estimator under the true $\tau^0$, we insert the specification of $\diff Y$ and obtain
\[
\hat\tau =
	\argmax_\tau \bigg\{\int_{\tau\wedge\tau^0}^{\tau\vee \tau^0}\frac{\eta}{\sigma(x)}\diff B(x)-
	\frac{1}{2}\int_{\tau\wedge\tau^0}^{\tau\vee\tau^0}\frac{\eta^2}{\sigma(x)^2}\diff x\bigg\}.
	\]
	For the homoskedastic case $\sigma(x)=\delta n^{-1/2}$, scaling properties of Brownian motion show the identity in law
	\[\eta^2\delta^{-2}n(\hat\tau-\tau^0)\stackrel{\mathrm{d}}{=} \argmax_{h\in [-\eta^2\delta^{-2} n \tau^0 , \eta^2\delta^{-2} n(1-\tau^0)]}\bigg(B^\leftrightarrow(h)-\frac{\abs{h}}{2}\bigg) \overset{\mathrm{a.s.}}{\longrightarrow} \argmax_{h\in\R} \bigg(B^\leftrightarrow(h)-\frac{\abs{h}}{2}\bigg),\]
	provided $\delta n^{-1/2} = o(\eta)$, where $B^\leftrightarrow$ is a two-sided Brownian motion. 
	In particular, for $n=\delta^{-1}$, we obtain  convergence with rate $v_{\delta,\eta}=\eta^{-2}\delta^3$ for $\hat\tau$ in $\delta$ and $\eta$, provided $\delta^{3/2}=o(\eta)$.

\section{Proposed method and its convergence analysis}\label{sec:main}
We now turn to the simultaneous estimation of the true model parameters $(\theta_-^0(\delta),\theta_+^0(\delta), \tau^0)$ characterising \eqref{eq:theta} and \eqref{eq:spde}, marked with a superscript $0$ in this section.
In particular, the definition of the index $\ko$ implies that $\koo=\lceil \tau^0/\delta\rceil$.
For $i,k \in \{1,\ldots, \delta^{-1}\}$ and for given $(\theta_-,\theta_+,\theta_\circ) \in \Theta \coloneqq[\lowell,\upell]^3$, define
\begin{equation}\label{eq:deftdi}\theta_{\delta,i}(k) \coloneqq
\begin{cases}\theta_-, &\text{if } i < k, \\ \theta_{\circ}, &\text{if } i =k, \\ \theta_+, &\text{if } i > k,\end{cases} \quad\text{ and } \quad \theta^0_{\delta,i}  \coloneqq  \begin{cases}\theta_-^0(\delta), &\text{if } i < \koo, \\ \theta^0_\circ(\delta), &\text{if } i = \koo,\\ \theta_+^0(\delta), &\text{if } i > \koo, \end{cases}
\end{equation}
where $\theta^0_\circ(\delta)$, specified in Proposition \ref{lem:remainder} below, minimises the error induced by the constant approximation of the discontinuous diffusivity in the modified log-likelihood on the change point interval. As a consequence, the expectation of a remainder term appearing in a convenient representation of our estimator and that originates from this discontinuity is of a smaller order than its $L^1$-norm. This will prove to be decisive for obtaining sharp diffusivity estimation rates.
{\color{black}
The simultaneous estimator for the model parameters is introduced in Section \ref{sec:defbasic}, followed by the analysis of its convergence properties, in which we consider two different regimes: In Section \ref{sec:anaconsconv}, we prove consistency and establish convergence rates for the case of non-vanishing jump height, while Section \ref{sec:clt} investigates weak convergence properties of the change point estimator (adapted accordingly to the simplifying assumption that $\theta^0_\pm$ are known) in the vanishing jump height regime.}

\subsection{Definition of the estimator}\label{sec:defbasic}
Consider the modified $\log$-likelihood $\ell_{\delta,i}(\theta_-,\theta_+,\theta_\circ,k)$, based on the local observation process  $(X_{\delta,i}(t), X^\Delta_{\delta,i}(t))_{t \geq 0}$ introduced in \eqref{eq:locpro} and given by
\begin{equation}\label{eq:loglike}
	\ell_{\delta,i}(\theta_-,\theta_+,\theta_\circ,k) \coloneqq \theta_{\delta,i}(k) \int_0^T X^{\Delta}_{\delta,i}(t) \diff{X_{\delta,i}(t)} - \frac{\theta_{\delta,i}(k)^2}{2} \int_0^T X^\Delta_{\delta,i}(t)^2 \diff{t}.
\end{equation}
A detailed discussion of the motivation behind this modified local log-likelihood is contained in \cite[Section 4.1]{altmeyer21}.
Note here that the stochastic integrals are well-defined by the semimartingale nature of $X_{\delta,i}$, even at the change point, cf.\ Proposition \ref{prop:weaksol}.
Based on these functionals, we follow a modified likelihood approach which yields an M-estimator.
For the vector $(B_{\delta,i})_{i=1,\ldots,n}$ of independent Brownian motions as introduced in Proposition \ref{prop:weaksol}, define
\begin{equation}\label{eq:MI}
	M_{\delta,i} \coloneqq \int_0^T X^{\Delta}_{\delta,i}(t) \diff{B_{\delta,i}(t)}, \qquad I_{\delta,i} \coloneqq \int_0^T X^{\Delta}_{\delta,i}(t)^2 \diff{t}, \quad i = 1,\ldots,\delta^{-1},
\end{equation}
and note that $M_{\delta,i}$ is a continuous martingale in $T$ with quadratic variation $I_{\delta,i}$.
In our subsequent investigation, we want to exploit these structures; hence, as in the model considered in Section \ref{sec:est}, it will be convenient to rewrite \eqref{eq:loglike}.
Using Proposition \ref{prop:weaksol}, we obtain that
\begin{align*}
	\ell_{\delta,i}(\theta_-,\theta_+,\theta_\circ,k) &= (\theta_{\delta,i}(k)\theta_{\delta,i}^0 - \theta_{\delta,i}(k)^2/2) I_{\delta,i} + \theta_{\delta,i}(k) M_{\delta,i}\\
	&\quad + \one_{\{i = \koo\}}  \theta_{\delta,\koo}(k)R_{\delta,\koo}(\theta^0_\circ(\delta)),
\end{align*}
where, for $\theta' \in [\lowell,\upell]$, we define
\begin{equation}\label{eq:remain}
	\begin{split}
		R_{\delta,\koo}(\theta^\prime)
		&\coloneqq \int_0^T X^{\Delta}_{\delta,\koo}(t) \Big(\int_0^t \langle \Delta_{\vartheta^0} S_{\vartheta^0}(t-s)K_{\delta,\koo} - \theta^\prime S_{\theta^0}(t-s)\Delta K_{\delta,\koo},\diff{W_s} \rangle \Big)  \diff{t}.
	\end{split}
\end{equation}
For a continuous function $f\colon \Theta \times [n]\to\R$, let $
\argmaxalt_{\chi \in \Theta \times [n]} f(\chi)$ be a measurable version of a maximiser of $f$, and also define $\argminalt_{\chi \in \Theta \times [n]} f(\chi) \coloneqq \argmaxalt_{\chi \in \Theta \times [n]}( -f(\chi))$. 
We introduce the estimator
\begin{align*}
	(\hat{\theta}_-^\delta,\hat{\theta}_+^\delta, \hat{\theta}^\delta_\circ, \hat{k}^\delta)
	&\coloneqq \argmaxalt_{(\theta_-,\theta_+,\theta_\circ,k) \in \Theta \times [n]} \sum_{i=1}^n \ell_{\delta,i}(\theta_-,\theta_+,\theta_\circ,k) \\
	&=\argmaxalt_{(\theta_-,\theta_+,\theta_\circ,k) \in \Theta \times [n]} \Big\{ \sum_{i=1}^n \Big((\theta_{\delta,i}(k) - \theta_{\delta,i}^0)M_{\delta,i} - \frac{1}{2}(\theta_{\delta,i}(k) - \theta_{\delta,i}^0)^2 I_{\delta,i} \Big)\\
	& \qquad\qquad\qquad\quad + \theta_{\delta,\koo}(k) R_{\delta,\koo}(\theta_\circ^0(\delta))
	+ \sum_{i=1}^n \theta_{\delta,i}^0M_{\delta,i} +  \sum_{i=1}^n \frac{(\theta_{\delta,i}^0)^2}{2} I_{\delta,i}\Big\}\\
	&= \argmaxalt_{(\theta_-,\theta_+,\theta_\circ,k) \in \Theta \times [n]} \Big\{ \sum_{i=1}^n \Big((\theta_{\delta,i}(k) - \theta_{\delta,i}^0)M_{\delta,i} - \frac{1}{2}(\theta_{\delta,i}(k) - \theta_{\delta,i}^0)^2 I_{\delta,i} \Big) \\
	& \qquad\qquad\qquad\quad + \theta_{\delta,\koo}(k) R_{\delta,\koo}(\theta_\circ^0(\delta))\Big\}\\
	&=\argminalt_{(\theta_-,\theta_+,\theta_\circ,k) \in \Theta \times [n]} \big\{Z_\delta(\theta_-,\theta_+,\theta_\circ,k) -\theta_{\delta,\koo}(k) R_{\delta,\koo}(\theta_\circ^0(\delta))\big\},
\end{align*}
where
\[Z_\delta(\theta_-,\theta_+,\theta_\circ,k) \coloneqq \frac{1}{2}\sum_{i=1}^{n} (\theta_{\delta,i}(k) - \theta_{\delta,i}^0)^2 I_{\delta,i} - \sum_{i=1}^n (\theta_{\delta,i}(k) - \theta_{\delta,i}^0)M_{\delta,i}.
\]
The following result summarises important estimates on the remainder term $R_{\delta,\koo}$.

\begin{proposition} \label{lem:remainder}
For any $\theta^\prime \in [\lowell,\upell]$, $R_{\delta,\koo}(\theta^\prime)$ given by \eqref{eq:remain} satisfies
\[\E[\lvert R_{\delta,\koo}(\theta^\prime) \rvert] \lesssim \delta^{-2}  \quad  \text{and} \quad  \mathrm{Var}(R_{\delta,\koo}(\theta^\prime)) \lesssim \delta^{-2}.\]
In case $\theta' = \theta^0_-(\delta)$, we have an explicit bound in terms of the jump size $\eta$, namely
\[\E[\lvert R_{\delta,\koo}(\theta^0_-(\delta)) \rvert] \leq \frac{T}{\sqrt{2}\lowell}\lVert K^\prime \rVert^2_{L^2} \lvert \eta \rvert \delta^{-2}.\]
Moreover, for any $\delta \in 1/\N$, there exists $\theta^0_\circ(\delta) \in [\lowell,\upell]$ such that
\[\lvert \E[R_{\delta,\koo}(\theta^0_\circ(\delta))]\rvert \lesssim \delta^{-1}.\]
\end{proposition}
\begin{proof}
	See Appendix \ref{app:proof_rem}.
\end{proof}

Hence, if we define the estimator $\hat{\chi}_\delta \coloneqq (\hat{\theta}_-^\delta,\hat{\theta}_+^\delta, \hat{\theta}^\delta_\circ,\hat{\tau}^\delta)$ with $\hat{\tau}^\delta \coloneqq \hat{k}^\delta\delta$ and let
\[
\mathcal{Z}_\delta(\theta_-,\theta_+,\theta_\circ,h) \coloneqq Z_\delta(\theta_-,\theta_+, \theta_\circ,\lceil h/\delta \rceil), \quad h \in (0,1],
\]
we obtain
\begin{equation}\label{eq:cons0}
	\begin{split}\mathcal{Z}_\delta(\hat{\chi}_\delta)= Z_\delta(\hat{\theta}_-^\delta,\hat{\theta}_+^\delta,\hat{\theta}^\delta_\circ,\hat{k}^\delta)& = \min_{(\theta_-,\theta_+,\theta_\circ,k) \in \Theta \times [n]} Z_\delta(\theta_-,\theta_+,\theta_\circ,k) + \mathcal{O}_{\PP}(\delta^{-2})\\
		&= \min_{\chi \in \Theta \times (0,1]} \mathcal{Z}_{\delta}(\chi) + \mathcal{O}_{\PP}(\delta^{-2}).
	\end{split}
\end{equation}

\subsection{Simultaneous estimation of the model parameters for non-vanishing jump height}\label{sec:anaconsconv}
We now proceed to derive the following main result on the convergence rate of the estimators $\hat\chi_\delta$.

\begin{theorem}\label{theo:rate}
Suppose that, for some $\theta^\ast_\pm \in [\lowell,\upell]$ with $\theta^\ast_- \neq \theta^\ast_+$, it holds
\begin{equation}\label{eq:condapprox}
(\theta_-^0(\delta),\theta_+^0(\delta),\tau^0) \to (\theta_-^\ast,\theta^\ast_+,\tau^0)\eqqcolon \chi^\ast \quad \text{ as }\delta \to 0.
\end{equation}	
Then, 
\begin{equation}\label{eq:rate}
	\lvert \hat{\theta}_\pm^\delta - \theta^0_{\pm}(\delta) \rvert = {\mathcal O}_{\PP}(\delta^{3/2} )\quad\text{ and }\quad \lvert \hat{\tau}^{\delta} - \tau^0 \rvert = \mathcal{O}_{\PP}(\delta).
\end{equation}
\end{theorem}
The above theorem demonstrates that the change point $\tau^0$ is estimated at the rate $\delta = n^{-1}$, while the diffusivity constants $\theta_\pm^0$ are estimated  at rate $\delta^{3/2}$,  provided that $\eta$ does not vanish. The former is analogous to the classical rate of convergence for scalar change point problems with independent observations, cf.\ \cite[Section 14.5.1]{kosorok08}.  It is, however, much slower than the rate $\delta^3$ obtained in the model problem presented in Section \ref{sec:est} for $\eta$ fixed. This is due to the discrete observations at distance $\delta$. Yet,  the diffusivity parameter estimation rate $\delta^{3/2}$ matches the minimax rate  for a stochastic heat equation without change point based on multiple local measurements, recently determined in \cite{altmeyer22}. This rate is  much faster than the classical rate $\delta^{1/2}$ in i.i.d.\ change point models.
{\color{black}
In the following sections, we provide the technical tools needed to prove Theorem \ref{theo:rate}.}

\subsubsection{Concentration analysis}\label{sec:conc}
For an in-depth analysis of the convergence of $\hat\chi_\delta$, it will be of central importance to understand the concentration properties of sums of the martingales and quadratic variations $M_{\delta,i}$ and $I_{\delta,i}$,  introduced in \eqref{eq:MI}.
Although the Brownian motions $(B_{\delta,i})_{i \in [n]}$ appearing in the representation \eqref{eq:mild2} are independent, the stochastic integrals $(M_{\delta,i})_{i \in [n]}$ are not.
This makes concentration analysis of sums involving $M_{\delta,i}$ a delicate matter, which we resolve by employing a coupling approach based on the Dambis--Dubins--Schwarz theorem.  Furthermore, 
techniques originating from Malliavin calculus allow us to obtain a Bernstein-type inequality.
For notational convenience, we drop the superscript $0$ for the true parameters in this subsection, i.e., $\theta_\pm(\delta) \equiv \theta_\pm^0(\delta), \tau^0 \equiv \tau$.
Recall that the choice $\ko=\lceil \tau / \delta \rceil$ guarantees that the change point $\tau$ belongs to the support of $K_{\delta,\ko}$.

\begin{lemma} \label{lem:basic_est}
	\begin{enumerate}[label = (\roman*), ref = (\roman*)]
		\item
		For any $i \neq \ko$, we have
		\[\E[I_{\delta,i}] = \frac{T}{2\vartheta_{\delta,i}(\ko)}\lVert K^\prime \rVert^2_{L^2} \delta^{-2}  + C(\delta),\]
		where $C(\delta) \in [-1/(4\underline\theta^2),0]$ for any $\delta \in 1/\N$.
		\item \label{lem:basic_est2}
		It holds
		\[\E[I_{\delta,\ko}] \in \Big[\frac{2\underbars{\lambda} T- 1 + \mathrm{e}^{-2\underbars{\lambda}T}}{4\underbars{\lambda}\upell}\lVert K^\prime \rVert^2_{L^2} \delta^{-2}, \frac{T}{2 \lowell} \lVert K^\prime \rVert^2_{L^2} \delta^{-2}\Big]\]
		and also
		\[\E[I_{\delta,\ko}] \in \Big[\frac{T}{2\upell} \lVert K^\prime\rVert^2_{L^2}\delta^{-2} + \mathcal{O}(\delta^{-1}), \frac{T}{2\lowell}\lVert K^\prime\rVert^2_{L^2} \delta^{-2} \Big].\]
		\item
		For any vector $\alpha  \in \R^n$ with the  proviso that $\alpha_{\ko} = 0$, we have
		\begin{equation}\label{eq:varbound}
			\Var\Big(\sum_{i=1}^n \alpha_i I_{\delta,i} \Big) \leq \frac{T}{2 \lowell^3} \delta^{-2} \lVert \alpha \rVert^2_{\ell^2} \lVert K^\prime \rVert^2_{L^2}.
		\end{equation}
		\item \label{lem:basic_est4}
		With a constant only depending on $\underline\theta$, $\overline\theta$ and $K$, it holds
		\[\mathrm{Var}(I_{\delta,\ko}) \lesssim \delta^{-2}.\]
	\end{enumerate}
\end{lemma}
\begin{proof}
	See Appendix \ref{app:proof_conc}.
\end{proof}

The variance bound stated in \eqref{eq:varbound} demonstrates that the linear combination $\sum_{i = 1}^n \alpha_i I_{\delta,i}$ deviates around its mean with order $\delta^{-1}\lVert \alpha \rVert_{\ell^2}$, meaning that the terms $I_{\delta,i}$ are only weakly correlated.
We strengthen this statement by establishing a mixed-tail concentration inequality for such linear combinations.
Noting that $\sum_{i=1}^n \alpha_i(I_{\delta,i} - \E[I_{\delta,i}])$ lies in some second Wiener chaos, our proof relies on Malliavin calculus, based on the results of \cite{nourdin09}.

\begin{proposition} \label{prop:conc}
	Let $\alpha \in \R^n_+ \setminus \{0\}$ with $\alpha_{\ko} = 0$. Then, for any $z > 0$, we have
	\[\PP\Big(\Big\lvert \sum_{i=1}^n \alpha_i(I_{\delta,i} - \E[I_{\delta,i}]) \Big\rvert \geq z\Big) \leq 2\exp\bigg(-\frac{\lowell^2}{4\lVert \alpha \rVert_{\infty}} \frac{z^2}{z + \sum_{i=1}^n \alpha_i \E[I_{\delta,i}]} \bigg).\]
	In particular,   it holds
	\[ \PP\Big(\Big\lvert \sum_{i=1}^n \alpha_i(I_{\delta,i} - \E[I_{\delta,i}]) \Big\rvert \geq z\Big) \leq 2\exp\bigg(-\frac{\lowell^2}{2\lVert \alpha \rVert_{\infty}} \frac{z^2}{2 z + \lVert \alpha \rVert_{\ell^1} T\lowell^{-1}\lVert K^\prime \rVert^2_{L^2}\delta^{-2}} \bigg).\]
\end{proposition}
\begin{proof}
	See Appendix \ref{app:proof_conc}.
\end{proof}

In order to give a fine analysis of deviations of sums of the martingale terms $M_{\delta,i}$, the following coupling construction will be crucial. Introduce the stopping times
\[\sigma_i = \sigma_i(\delta) \coloneqq \inf\Big\{t \geq 0 : \int_0^t X_{\delta,i}^\Delta(s)^2 \diff{s} > \E[I_{\delta,i}]\Big\}, \quad i \in [\delta^{-1}],\]
as well as
\begin{equation}\label{eq:coupling}
	\overbar{M}_{\delta,i} \coloneqq \int_0^{\sigma_i} X^{\Delta}_{\delta,i}(t) \diff{B}_{\delta,i}(t), \quad \overbar{I}_{\delta,i} \coloneqq \E[I_{\delta,i}], \quad i \in [\delta^{-1}].
\end{equation}
Here, we suppose that the cylindrical Brownian motion $W(t)$ and the SPDE solution $X(t)$ are extended to all $t\in[0,\infty)$ so that the $\sigma_i$ are a.s.~finite, noting the linear growth of $\int_0^t\E[X^{\Delta}_{\delta,i}(s)^2]\diff{s}$ in $t>0$ and the strong concentration around the expectation, provided by Lemma \ref{lem:basic_est}.
As the following result demonstrates, contrary to the vector $(M_{\delta,i})_{i \in [n]}$, the vector $(\overbar{M}_{\delta,i})_{i \in [n]}$ is Gaussian with independent components.
The deviation analysis of the sums $\sum_{i=1}^n\alpha_i M_{\delta,i}$ may be broken down into easier to handle deviations of $\sum_{i=1}^n\alpha_i \overbar{M}_{\delta,i}$ and the coupling error which is controlled by $\sum_{i=1}^n \alpha_i^2 \lvert I_{\delta,i} - \overbar{I}_{\delta,i}\rvert$.

\begin{proposition}\label{prop:coupling}
	The family of random variables $\{\overbar{M}_{\delta,i}, i \in [\delta^{-1}]\}$ is independent with $\overbar{M}_{\delta,i} \sim N(0,\overbar{I}_{\delta,i})$.
	Moreover, for any $\alpha \in \R^n$, $z,L > 0$, it holds
	\[\PP\Big(\Big\lvert \sum_{i=1}^{n} \alpha_i (M_{\delta,i} - \bar{M}_{\delta,i})\Big\rvert \geq z, \sum_{i=1}^n \alpha_i^2 \lvert I_{\delta,i} - \overbar{I}_{\delta,i} \rvert \leq L\Big) \leq 2\exp\left(-\frac{z^2}{2L}\right).\]
\end{proposition}
\begin{proof}
	See Appendix \ref{app:proof_conc}.
\end{proof}

\subsubsection{Verification of consistency}
As a first main ingredient, we establish consistency of the estimator $(\hat{\theta}_-^\delta,\hat{\theta}_+^\delta,\hat{\tau}^\delta)$, or, in other words, consistency of $\hat{\chi}_\delta$ with respect to the pseudometric
\begin{equation}\label{eq:pseudo}
	\overline{d}(\chi,\tilde{\chi}) = \lvert \theta_- - \tilde{\theta}_- \rvert +  \lvert \theta_+ - \tilde{\theta}_+ \rvert + \lvert h - \tilde{h} \rvert, \quad \chi,\tilde{\chi} \in \Theta \times (0,1].
\end{equation}
We are not concerned with the convergence of the estimator $\hat{\theta}_\circ^\delta$ of the balancing parameter $\theta^0_\circ(\delta)$  because this is a nuisance parameter, introduced only for the technical reason to achieve optimal simultaneous estimation of the true physical parameters $(\theta_+^0(\delta),\theta_-^0(\delta),\tau^0)$. In fact, an M-estimation strategy without $\theta_\circ(\delta)$ would yield estimators $\hat\theta_{\pm}$ only converging with rate $\delta^{1/2}$ due to the contribution of the change point block to the overall criterion.

The proof of consistency combines an appropriate adaption of classical consistency proofs for M-estimators, cf.\ \cite[Theorem 2.12]{kosorok08}, with uniform convergence of the centered empirical process $\delta^3(\mathcal{Z}_\delta(\cdot) - \E[Z_\delta(\cdot)])$ that is derived based on the estimates from Section \ref{sec:conc}.
Recalling from \eqref{eq:cons0} that
\[
\mathcal{Z}_\delta(\hat{\chi}_\delta)= \min_{\chi \in \Theta \times (0,1]} \mathcal{Z}_{\delta}(\chi) + \mathcal{O}_{\PP}(\delta^{-2}),
\]
the analysis will rely on a convenient decomposition of $\mathcal Z_\delta$.
For $I_{\delta,i}$ and $M_{\delta,i}$ introduced in \eqref{eq:MI}, let
\begin{align*}
	I_{T,\delta}(\theta_-,\theta_+,\theta_\circ,h) &\coloneqq \frac{1}{2}\sum_{i \in [\delta^{-1}] \setminus \{\koo\}} (\theta_{\delta,i}(\lceil h/\delta \rceil) - \theta^0_{\delta,i})^2 I_{\delta,i},\\
	M_{T,\delta}(\theta_-,\theta_+,\theta_\circ,h) &\coloneqq \sum_{i \in [\delta^{-1}] \setminus\{\koo\}} (\theta_{\delta,i}(\lceil h/\delta \rceil) - \theta^0_{\delta,i}) M_{\delta,i}.
\end{align*}
By Lemma \ref{lem:basic_est} and Proposition \ref{lem:remainder}, we can write
\begin{equation}\label{eq:cons1}
	\mathcal{Z}_{\delta}(\theta_-,\theta_+,\theta_\circ,h) = I_{T,\delta}(\theta_-,\theta_+,\theta_\circ,h) - M_{T,\delta}(\theta_-,\theta_+,,\theta_\circ,h) + \mathcal{O}_{\PP}(\delta^{-2}),
\end{equation}
where the $\mathcal{O}_{\PP}(\delta^{-2})$-term is with respect to uniform convergence on $(\Theta \times (0,1], d)$, for $d$ denoting the restriction of the Euclidean metric to $\Theta \times (0,1]$.
For $\chi = (\theta_-,\theta_+,\theta_\circ,h) \in \Theta \times (0,1]$, define the restriction $\chi^\prime \coloneqq (\theta_-,\theta_+,h) \in \Theta^\prime \times (0,1]$, where $\Theta^\prime \coloneqq [\underbars{\theta},\overbar{\theta}]^2$.
Furthermore, let $\chi^0(\delta) = (\theta^0_+(\delta),\theta^0_-(\delta),\theta^0_\circ(\delta),\tau^0)$ and set
\[\theta_{\chi^\prime}(x) \coloneqq \theta_- \one_{(0,h)}(x) + \theta_+\one_{[h,1)}(x), \quad x \in (0,1), \chi^\prime \in \Theta^\prime \times (0,1]. \]
We have the following convergence result in expectation.

\begin{lemma}\label{lem:approxZ}
Suppose  that \eqref{eq:condapprox} holds true.
Then, for
\begin{equation}\label{eq:lim_func}
		\mathcal{Z}(\chi^\prime) \coloneqq \frac{T}{4}\lVert K^\prime \rVert^2_{L^2} \int_0^1 \frac{(\theta_{\chi^\prime}(x) - \theta^\ast(x))^2}{\theta^\ast(x)} \diff{x}, \quad \chi^\prime \in \Theta^\prime \times (0,1],
	\end{equation}
	we have
	\[\lim_{\delta \to 0}\sup_{\chi\in \Theta \times (0,1]} \Big\vert \delta^3 \E[\mathcal{Z}_{\delta}(\chi)] - \mathcal{Z}(\chi^\prime) \Big\vert = 0.\]	
\end{lemma}
\begin{proof}
	See Appendix \ref{app:proof_rem}.
\end{proof}

The key step for proving consistency will consist in verifying that
\[
\sup_{\chi \in \Theta \times (0,1]} \big\lvert\delta^3 \mathcal{Z}_\delta(\chi) - \mathcal{Z}(\chi^\prime)\big\rvert = o_{\PP}(1),
\]
and in view of \eqref{eq:cons1} and \eqref{eq:lim_func},  this task can be broken down into separate investigations of the centered statistics $I_{T,\delta} - \E[I_{T,\delta}]$ and $M_{T,\delta}$.
Exploiting the concentration properties established in Section \ref{sec:conc}, we obtain the following central auxiliary result.

\begin{lemma}\label{lem:aux0}
	It holds
	\begin{align}
		\sup_{\chi \in \Theta \times (0,1]} \big\vert I_{T,\delta}(\chi) - \E[I_{T,\delta}(\chi)] \big\vert &= o_{\PP}\big(\delta^{-3/2} \sqrt{\log(\delta^{-1})}\big),\label{eq:uni_qv} \\
		\sup_{\chi \in \Theta \times (0,1]} \big\vert M_{T,\delta}(\chi) \big\vert &= \mathcal{O}_{\PP}(\delta^{-3/2}).\label{eq:uni_mart}
	\end{align}
\end{lemma}
\begin{proof}
	See Appendix \ref{app:proof_conc}.
\end{proof}

We now prove our first main result.

\begin{theorem} \label{theo:cons}
	Suppose that, for some $\theta^\ast_\pm \in [\lowell,\upell]$ with $\theta^\ast_- \neq \theta^\ast_+$,
	we have
	\eqref{eq:condapprox}.
	Then, $(\hat{\chi}_\delta)^\prime - (\chi^0(\delta))^\prime \overset{\PP}{\longrightarrow} 0$.
	Equivalently, $(\hat{\chi}_\delta)^\prime$ is a consistent estimator of $\chi^\ast$.
\end{theorem}

\begin{proof}
	The assumption $(\chi^0(\delta))^\prime \to \chi^\ast$ as $\delta \to 0$ implies that the statements $(\hat{\chi}_\delta)^\prime - (\chi^0(\delta))^\prime \overset{\PP}{\longrightarrow} 0$ and $(\hat{\chi}_\delta)^\prime \overset{\PP}{\longrightarrow} \chi^\ast$ are equivalent.
	We will prove the latter.
	
	\smallskip
	
	Since $\tau^0 \notin \{0,1\}$ and $\theta^\ast_- \neq \theta^\ast_+$, if $\limsup_{m\to \infty} \lvert \chi_m^\prime - \chi^\ast \rvert > 0$ for some sequence $(\chi_m^\prime)$ in $\Theta^\prime \times (0,1]$, it clearly holds for $\mathcal Z$ defined in \eqref{eq:lim_func} that
	\[\limsup_{m \to \infty} \mathcal{Z}(\chi_m^\prime) > 0 = \mathcal{Z}(\chi^\ast).
	\]
	As in the proof of \cite[Theorem 2.12]{kosorok08}, this identification property guarantees the existence of a non-decreasing càdlàg function $f\colon [0,\infty] \to [0,\infty]$ such that $f(0) = 0$ and
	\begin{equation}\label{eq:bound_cons}
		\lvert \chi^\prime - \chi^\ast \rvert \leq f(\lvert \mathcal{Z}(\chi^\prime) - \mathcal{Z}(\chi^\ast)\rvert), \quad \chi^\prime \in \Theta^\prime \times (0,1].
	\end{equation}
	Using that $\chi^\ast$ minimises $\mathcal{Z}$,
	we obtain 
	\[
	\lvert \mathcal{Z}((\hat{\chi}_\delta)^\prime) - \mathcal{Z}(\chi^\ast)\rvert 
	= \mathcal{Z}((\hat{\chi}_\delta)^\prime) - \delta^3\mathcal{Z}_\delta(\hat{\chi}_\delta) + \delta^3\mathcal{Z}_\delta(\hat{\chi}_\delta)  - \mathcal{Z}(\chi^\ast),
	\]
	 and taking into account \eqref{eq:cons0}, \eqref{eq:cons1} and Lemma \ref{lem:approxZ}, this last expression is upper bounded by
	\begin{equation} \label{eq:bound_expdist}
		\begin{split}
			&\mathcal{Z}((\hat{\chi}_\delta)^\prime) - \delta^3\mathcal{Z}_\delta(\hat{\chi}_\delta) + \delta^3\mathcal{Z}_\delta(\theta_-^\ast,\theta_+^\ast,\underbars{\theta}, \tau^0)  - \mathcal{Z}(\chi^\ast) + \mathcal{O}_{\PP}(\delta)\\
			&\quad\leq 2\sup_{\chi \in \Theta \times (0,1]} \big\lvert \delta^3 \mathcal{Z}_\delta(\chi) - \mathcal{Z}(\chi^\prime) \big\rvert + \mathcal{O}_{\PP}(\delta)\\
			&\quad\leq 2\sup_{\chi \in \Theta \times (0,1]} \big\lvert\delta^3 \E[\mathcal{Z}_\delta(\chi)] - \mathcal{Z}(\chi^\prime)\big\rvert + 2\delta^3 \sup_{\chi \in \Theta \times (0,1]} \big\lvert \mathcal{Z}_\delta(\chi) - \E[\mathcal{Z}_\delta(\chi)]\big\rvert + \mathcal{O}_{\PP}(\delta)\\
			&\quad\leq 2\delta^3\sup_{\chi \in \Theta \times (0,1]} \big\vert I_{T,\delta}(\chi) - \E[I_{T,\delta}(\chi)] \big\vert + 2\delta^3\sup_{\chi \in \Theta \times (0,1]} \big\vert M_{T,\delta}(\chi) \big\vert + \mathcal{O}_{\PP}(\delta) + o(1).
		\end{split}
	\end{equation}
	By \eqref{eq:bound_cons} and the properties of $f$, the assertion already follows since Lemma \ref{lem:aux} implies that
	\[
	\sup_{\chi \in \Theta \times (0,1]} \big\vert I_{T,\delta}(\chi) - \E[I_{T,\delta}(\chi)] \big\vert = o_{\PP}(\delta^{-3}), \qquad
	\sup_{\chi \in \Theta \times (0,1]} \big\vert M_{T,\delta}(\chi) \big\vert = o_{\PP}(\delta^{-3}).
	\]
\end{proof}

The sharp uniform bounds in \eqref{eq:uni_qv} and \eqref{eq:uni_mart} motivate our approach to infer optimal convergence rates  presented next.
Define the empirical process $(\mathcal{L}_\delta(\chi))_{\chi \in \Theta \times (0,1]}$ by setting
\begin{align*}
	\delta^{-3}\mathcal{L}_\delta(\chi)
	&\coloneqq \frac{1}{2}\sum_{i= 1}^{\delta^{-1}} (\theta_{\delta,i}(\lceil h/\delta \rceil)- \theta^0_{\delta,i})^2I_{\delta,i} - \sum_{i= 1}^{\delta^{-1}} (\theta_{\delta,i}(\lceil h/\delta \rceil) - \theta^0_{\delta,i})M_{\delta,i} \\
	&\quad\quad\quad - \theta_{\delta,\koo}(\lceil h \slash \delta \rceil) R_{\delta,\koo}(\theta^0_\circ(\delta))\\
	&= \mathcal{Z}_\delta(\chi)  - \theta_{\delta,\koo}(\lceil h \slash \delta \rceil) R_{\delta,\koo}(\theta^0_\circ(\delta)),
\end{align*}
and notice that $\hat{\chi}_\delta$ is the precise minimiser of $\mathcal{L}_\delta$, i.e.,
\begin{equation}\label{eq:min_precise}
	\mathcal{L}_\delta(\hat{\chi}_\delta) = \min_{\chi \in \Theta \times (0,1]} \mathcal{L}_\delta(\chi).
\end{equation}
Define further
\[\tilde{\mathcal{L}}_\delta(\chi)\coloneqq \delta^3 \E[\mathcal{Z}_\delta(\chi)], \quad \chi \in \Theta \times (0,1],\]
and observe that, according to Proposition \ref{lem:remainder}, it holds
\begin{equation}\label{eq:exp_l}
	\tilde{\mathcal{L}}_\delta(\chi) = \E[\mathcal{L}_\delta(\chi)] + \mathcal{O}(\delta^2).
\end{equation}
Moreover, by Lemma \ref{lem:basic_est} and $\tilde{\mathcal L}_\delta(\chi^0(\delta)) = 0$, we obtain the following lower bound,
\begin{equation}\begin{split}\label{eq:exp_low}
		&\tilde{\mathcal{L}}_\delta(\chi) - \tilde{\mathcal{L}}_\delta(\chi^0_\delta)
		\asymp \delta \sum_{i=1}^{\delta^{-1}} (\theta_{\delta,i}(\lceil nh\rceil) - \theta^0_{\delta,i})^2\\
		&\qquad\gtrsim \delta(\lceil nh \rceil \wedge \koo -1) (\theta_- - \theta^0_-(\delta))^2 + (1 - \delta (\lceil nh\rceil \vee \koo)) (\theta_+ - \theta^0_+(\delta))^2 \\
		&\quad\,\quad +  \delta\big(\big\lvert \koo - \lceil nh \rceil \big\rvert -1\big)^+ \Big(\one_{\{\koo < \lceil nh \rceil\}} (\theta_- - \theta^0_+(\delta))^2 + \one_{\{\koo > \lceil nh \rceil\}} (\theta_+ - \theta^0_-(\delta))^2\Big)\\
		&\quad\, \quad+ \delta (\theta^0_\circ(\delta) - \theta_{\delta,\koo}(\lceil nh\rceil))^2 + \delta(\theta_\circ - \theta^0_{\delta,\lceil nh \rceil})^2.
		\end{split}
	\end{equation}
Suppose \textit{for the moment} that we have additional information on the true parameters $\theta^0_\pm$ and $\tau^0$ in the sense that we know that\vspace{-2mm}
\begin{enumerate}\setlength\itemsep{0mm}
	\item[(i)] $\theta^0_\pm \in \Theta_\pm$ for some disjoint intervals $\Theta_\pm \subset [\lowell,\upell]$ that are separated by a magnitude of $\underbars{\eta} > 0$, i.e., for all $\theta_\pm \in \Theta_\pm$, it holds $\lvert \theta_+ - \theta_-\rvert \geq \underbars{\eta}$, and
	\item[(ii)] the change point $\tau^0$ is separated from the boundary by at least $\kappa \in (0,1)$, i.e., $\tau^0 \in [\kappa,1-\kappa]$.
\end{enumerate}  \vspace{-2mm}
We may then change the optimisation domain of the estimator $\hat{\chi}_\delta$ accordingly, obtaining instead of \eqref{eq:min_precise}
\[\mathcal{L}_\delta(\hat{\chi}_\delta) = \min_{\chi \in \Theta_- \times \Theta_+ \times [\lowell,\upell] \times [\kappa,1-\kappa]} \mathcal{L}_\delta(\chi).\]
Then, for any $\chi \in \Theta_- \times \Theta_+ \times [\lowell,\upell] \times [\kappa,1-\kappa]$ and sufficiently small $\delta$, it follows from \eqref{eq:exp_low}
\begin{equation} \label{eq:low_sep}
	\tilde{\mathcal{L}}_\delta(\chi) - \tilde{\mathcal{L}}_\delta(\chi^0(\delta)) \gtrsim \kappa (\theta_- - \theta^0_-(\delta))^2 + \kappa (\theta_+ - \theta^0_+(\delta))^2 + \underbars{\eta}^2 \big(\lvert [\tau^0]_\delta - [h]_\delta \rvert -\delta \big)^+,
\end{equation}
where we denote for $x \in \R$ by $[x]_\delta = \delta \lceil x/\delta \rceil$ its (upper) $\delta$-approximation.
Similarly to \eqref{eq:bound_expdist}, taking into account \eqref{eq:min_precise} and \eqref{eq:exp_l}, we derive
\begin{align*}
	\tilde{\mathcal{L}}_\delta(\hat{\chi}_\delta) - \tilde{\mathcal{L}}_\delta(\chi^0(\delta)) &\lesssim 2\delta^3\sup_{\chi \in \Theta \times (0,1]} \big\vert I_{T,\delta}(\chi) - \E[I_{T,\delta}(\chi)] \big\vert + 2\delta^3\sup_{\chi \in \Theta \times (0,1]} \big\vert M_{T,\delta}(\chi) \big\vert\\
	&\quad+ \delta^3 \sqrt{\mathrm{Var}(I_{\delta,\koo})} + \delta^3 \sqrt{\mathrm{Var}(M_{\delta,\koo})} + \delta^3 \sqrt{\mathrm{Var}(R_{\delta,\koo}(\theta^0_\circ))} + \mathcal{O}(\delta^2).
\end{align*}
Using the uniform bounds \eqref{eq:uni_qv} and \eqref{eq:uni_mart} as well as the  bounds from Proposition \ref{lem:remainder} and Lemma \ref{lem:basic_est}, it follows
\[\tilde{\mathcal{L}}_\delta(\hat{\chi}_\delta) - \tilde{\mathcal{L}}_\delta(\chi^0(\delta)) = o_{\PP}(\delta^{3/2} \sqrt{\log(\delta^{-1})}).\]
In view of the lower bound \eqref{eq:low_sep}, this translates to the following estimation rates for the parameters,
\[\lvert \theta_\pm - \hat{\theta}^\delta_{\pm} \rvert = o_{\PP}(\delta^{3/4} \log(\delta^{-1})^{1/4}) \quad \text{and} \quad \lvert \tau^0 - \hat{\tau}^\delta \rvert =  \mathcal{O}_{\PP}(\delta).\]
In order to get rid of the separability assumption stated above in (ii) and to sharpen the convergence rate bound for the diffusivity estimators, a more careful analysis of the local fluctuations of the empirical process is necessary. To this end, on the basis of the consistency result stated in the previous theorem, we make use of a \textit{peeling device}.
Since the processes $\mathcal{L}_\delta$ do not have constant expectation, we need a slight generalisation of such a classical convergence result from empirical process theory, given as Theorem 14.4 in \cite{kosorok08}.

\begin{theorem}\label{theo:emp}
	Let $(L_\delta(\chi))_{\delta \in 1/\N}$ be a sequence of stochastic processes, indexed by a pseudometric space $(\mathcal{X},d)$, and, for any $\delta \in 1/\N$, let $\tilde{L}_\delta \colon \mathcal{X} \to \R$ be a deterministic function and $\chi^0_\delta \in \mathcal{X}$.
	Assume that, for $\delta$ small enough, there exist some constants $\kappa, c_1 > 0$ independent of $\delta$ such that, for any $\chi \in B_{d}(\chi^0_\delta, \kappa)$, we have
	\begin{equation}\label{eq:condemp1}
		\tilde{L}_\delta(\chi) - \tilde{L}_\delta(\chi^0_\delta)  \geq c_1 \tilde{d}_\delta^2(\chi,\chi^0_\delta).
	\end{equation}
	Here, $\tilde{d}_\delta\colon \mathcal{X} \times \mathcal{X} \to [0,\infty)$ are such that for any $\varepsilon > 0$ there exists $\varepsilon^\prime > 0$ s.t.\ for all $\delta$ small enough, $d(\chi,\chi^0_\delta) \leq \varepsilon^\prime$ implies $\tilde{d}_\delta(\chi,\chi^0_\delta) \leq \varepsilon$ for any $\chi \in \mathcal{X}$.
	Suppose also that, for all $\delta, \varepsilon > 0$ small enough, we have
\begin{equation}\label{eq:condemp2}
\E^\ast\Big[\sup_{\tilde{d}_\delta(\chi,\chi^0_\delta) < \varepsilon} \big\lvert (L_\delta - \tilde{L}_\delta)(\chi) - (L_\delta - \tilde{L}_\delta)(\chi^0_\delta) \big\rvert \Big] \leq c_2 \psi_\delta(\varepsilon)
\end{equation}
for some $c_2 > 0$ and functions $\psi_\delta$ such that $\varepsilon \mapsto \psi_\delta(\varepsilon)/\varepsilon^\alpha$ is decreasing for some $\alpha < 2$ not depending on $\delta$.
	Let $r_\delta$ be such that, for all $\delta \in 1/\N$ and some $c_3 > 0$, we have
	\begin{equation}\label{eq:condemp3}
		r_\delta^2 \psi_\delta(r_\delta^{-1}) \leq c_3.
	\end{equation}
	If the sequence $(\hat{\chi}_\delta)_{\delta \in 1/\N}$ satisfies  $L_\delta(\hat{\chi}_\delta) \leq \inf_{\chi \in \mathcal{X}} L_\delta(\chi) + \mathcal{O}_{\PP}(r_\delta^{-2})$ and $d(\hat{\chi}_\delta, \chi^0_\delta) \overset{\PP^\ast}{\longrightarrow} 0$, then $ \tilde{d}_\delta(\hat{\chi}_\delta,\chi^0_\delta) = \mathcal{O}_{\PP}(r_\delta^{-1})$.
\end{theorem}

The proof is a straightforward adaptation of the proof given in \cite{kosorok08} and therefore omitted.
We will apply this general result to the sequence of empirical processes $(\mathcal{L}_\delta)_{\delta \in \N}$ and the expectation proxys $(\tilde{\mathcal{L}}_\delta)_{\delta \in 1/\N}$ introduced above, as well as the function
\begin{align*}
&\tilde{d}^2_\delta((x^{(1)}_-,x^{(1)}_+,x^{(1)}_\circ, h^{(1)}), (x^{(2)}_-,x^{(2)}_+,x^{(2)}_\circ,h^{(2)}))\\
	&\,= \lvert x^{(1)}_- - x^{(2)}_- \rvert^2 + \lvert x^{(1)}_+ - x^{(2)}_+ \rvert^2 + \big(\lvert [h^{(1)}]_\delta - [h^{(2)}]_\delta \rvert - \delta\big)^+ + \delta \lvert x^{(1)}_\circ - x^{(2)}_\circ\rvert^2\one_{\{[h^{(1)}]_\delta = [h^{(2)}]_\delta\}}\\
	&\,\quad + \delta \big(\lvert x^{(1)}_+ - x^{(2)}_\circ\rvert^2 + \lvert x^{(1)}_\circ - x^{(2)}_- \rvert^2\big) \one_{\{[h^{(1)}]_\delta < [h^{(2)}]_\delta\}}\\
	&\,\quad + \delta \big(\lvert x^{(2)}_+ - x^{(1)}_\circ\rvert^2 + \lvert x^{(2)}_\circ - x^{(1)}_- \rvert^2\big) \one_{\{[h^{(1)}]_\delta > [h^{(2)}]_\delta\}},
\end{align*}
where $[h]_\delta \coloneqq \delta\lceil\delta^{-1}h\rceil$.

\smallskip

The following result is central to verifying condition \eqref{eq:condemp2} of the consistency theorem in our setting, cf.\ Corollary \ref{coro:aux} below.
Note that the involved supremum is measurable, allowing us to work with $\mathbb E[ \cdot ]$ instead of the outer expectation $\mathbb E^\ast[\cdot]$.

\begin{lemma}\label{lem:aux}
	For sufficiently small $\delta \in 1/\N$, we have, for any $\varepsilon \leq 1$,
	\begin{align*}
		&\E\Big[\sup_{\tilde{d}_\delta(\chi,\chi^0(\delta)) < \varepsilon} \big\lvert \delta^3 (\mathcal{Z}_\delta - \E[\mathcal{Z}_\delta])(\chi) - \delta^3 (\mathcal{Z}_\delta - \E[\mathcal{Z}_\delta])(\chi^0(\delta)) \big\rvert \Big] \\
		&\lesssim  \delta^3 + \delta^{1/2}\varepsilon^3 +  \delta \varepsilon^2 + \delta^{3/2} \varepsilon \eqqcolon \tilde{\psi}_\delta(\varepsilon).
	\end{align*}
\end{lemma}
\begin{proof}
	See Appendix \ref{app:proof_rem}.
\end{proof}

\begin{corollary} \label{coro:aux}
	For sufficiently small $\delta \in 1/\N$, we have, for any $\varepsilon \leq 1$ and $\gamma \in [0,3], \varrho \in [0,2]$,
	\[\E\Big[\sup_{\tilde{d}_\delta(\chi,\chi^0(\delta)) < \varepsilon} \big\lvert (\mathcal{L}_\delta - \tilde{\mathcal{L}}_\delta)(\chi) - ( \mathcal{L}_\delta - \tilde{\mathcal{L}}_{\delta})(\chi^0(\delta)) \big\rvert \Big] \lesssim  \delta^3 + \delta^{1/2} \varepsilon^\gamma + \delta \varepsilon^{\varrho} + \delta^{3/2} \varepsilon \eqqcolon \psi^{\gamma,\varrho}_\delta(\varepsilon).\]
\end{corollary}
\begin{proof}
	See Appendix \ref{app:proof_rem}.
\end{proof}

\subsubsection{Derivation of convergence rates}
Having identified the function $\psi_\delta(\cdot)$ determining the rate $r_\delta$ via \eqref{eq:condemp3} and given our previous findings, we are now in a position to give the proof of Theorem \ref{theo:rate}.

\begin{proof}[Proof of Theorem \ref{theo:rate}]
	We verify the conditions of Theorem \ref{theo:emp}.
	Note first that, for $r_\delta \coloneqq \delta^{-3/2}$, $\varepsilon \leq 1$, $\gamma \coloneqq 5/3, \varrho = 4/3$ and $\overbar{\psi}_\delta \coloneqq \psi_\delta^{\gamma,\varrho}$, we have
	\[\overbar{\psi}_\delta(r_{\delta}^{-1}) = 4\delta^3 = 4r_{\delta}^{-2},\]
	and, by Corollary \ref{coro:aux},
	\[\E\Big[\sup_{\tilde{d}_\delta(\chi,\chi^0(\delta)) < \varepsilon} \big\lvert (\mathcal{L}_\delta - \tilde{\mathcal{L}}_\delta)(\chi) - ( \mathcal{L}_\delta - \tilde{\mathcal{L}}_{\delta})(\chi^0(\delta)) \big\rvert \Big] \lesssim \overbar{\psi}_\delta(\varepsilon).\]
	Note also that $\varepsilon \mapsto \overbar{\psi}_\delta(\varepsilon)/\varepsilon^{\gamma}$ is decreasing, and that we know from \eqref{eq:min_precise} that
	\[
	\mathcal{L}_\delta(\hat{\chi}_\delta) = \min_{\chi \in \Theta \times (0,1]} \mathcal{L}_\delta(\chi).
	\]
	Moreover, our consistency result in Theorem \ref{theo:cons} shows that, for the pseudometric $\overline{d}$ defined in \eqref{eq:pseudo}, we have $\overbar{d}(\hat{\chi}_\delta,\chi^0(\delta)) = o_{\PP}(1)$.
	Given Lemma \ref{lem:aux}, it thus only remains to verify \eqref{eq:condemp1}, i.e., we need to show that, for $\delta$ small enough and some $c_1,\kappa > 0$ independent of $\delta$, we have
	\begin{equation}\label{eq:conc_lim}
		\tilde{\mathcal{L}}_\delta(\chi) - \tilde{\mathcal{L}}_{\delta}(\chi^0(\delta)) \geq c_1 \tilde{d}^2_\delta(\chi,\chi^0(\delta)), \quad \chi \in B_{\overline{d}}(\chi^0(\delta), \kappa).
	\end{equation}
	Note here that, using boundedness of $\Theta$,
	\[\tilde{d}_\delta(\chi,\chi^0(\delta))^2 \lesssim \lvert \theta_- -\theta^0_- \rvert^2 + \lvert \theta_+ - \theta^0_+\rvert^2 + \lvert [h]_\delta - [\tau^0]_\delta \rvert + \delta,\]
	whence clearly, for any given $\varepsilon > 0$ we can find $\varepsilon^\prime > 0$ such that, for all $\delta$ small enough, $\overline{d}(\chi,\chi^0(\delta)) < \varepsilon^\prime$ implies $\tilde{d}_\delta(\chi,\chi^0(\delta)) \leq \varepsilon$. Also observe that because of $\lim_{\delta \to 0} \theta^0_-(\delta) = \theta^\ast_- \neq \theta^\ast_+ = \lim_{\delta \to 0} \theta^0_+(\delta)$, we can find $\underbars{\eta} > 0$ and $M > 0$ such that for $\delta \leq 1/M$ it holds $\lvert \eta^0(\delta) \rvert = \lvert\theta^0_+(\delta) - \theta^0_-(\delta) \rvert \geq \underbars{\eta}$.
	Let now $\kappa \coloneqq \min\{1- \tau^0, \tau^0, \underbars{\eta}\}/2$, which is strictly positive since $\tau^0 \notin \{0,1\}$.
	By Lemma \ref{lem:basic_est} and picking up the calculation in \eqref{eq:exp_low},
	for any $(\theta_-,\theta_+, \theta_\circ,h)=\chi \in B_d(\chi^0(\delta),\kappa)$,  we find for $\delta < \kappa/8 \wedge 1/M$
	\begin{align*}
		\tilde{\mathcal{L}}_\delta(\chi) - \tilde{\mathcal{L}}_\delta(\chi^0(\delta)) &\gtrsim
		\frac{\kappa}{4} (\theta_- -\theta^0_-(\delta))^2 + \frac{\kappa}{4} (\theta_+ -\theta^0_+(\delta))^2 + \frac{\underbars{\eta}^2}{4} (\lvert [\tau^0]_\delta - [h]_\delta\rvert - \delta)^+ \\
		&\quad + \delta (\theta^0_\circ(\delta) - \theta_{\delta,\koo}(\lceil nh\rceil))^2 +\delta(\theta_\circ - \theta^0_{\delta,\lceil nh \rceil})^2\\
		&\gtrsim \tilde{d}_\delta(\chi, \chi^0(\delta)),
	\end{align*}
	where  we used $\lvert \eta^0(\delta) \pm (\theta_\pm - \theta^0_{\pm}(\delta)) \rvert \geq \underbars{\eta}/2$ by our choice of $\kappa$. 
	Thus, \eqref{eq:conc_lim} is satisfied, and Theorem \ref{theo:emp} yields
	\[\delta^{-3/2}\tilde{d}_\delta(\hat{\chi}_\delta,\chi^0(\delta)) = \mathcal{O}_{\PP}(1).\]
	By definition of $\tilde{d}_\delta$, this implies $\lvert \hat{\theta}{}^\delta_\pm - \theta^0_\pm(\delta) \rvert = \mathcal{O}_{\PP}(\delta^{3/2})$ and $(\lvert [\tau^0]_\delta - [\hat{\tau}{}^\delta]_\delta \rvert -\delta)^+ = \mathcal{O}_{\PP}(\delta^{3})$. Since
	\[\lvert \hat{\tau}{}^\delta - \tau^0\rvert \leq \delta + \lvert [\tau^0]_\delta - \hat{\tau}{}^\delta \rvert = \delta + \lvert [\tau^0]_\delta - [\hat{\tau}{}^\delta]_\delta \rvert \leq 2\delta + \big(\lvert [\tau^0]_\delta - [\hat{\tau}^\delta]_\delta \rvert - \delta\big)^+,\]
	the latter conclusion now also yields $\lvert \hat{\tau}{}^\delta - \tau^0\rvert = \mathcal{O}_{\PP}(\delta)$.
\end{proof}

\subsection{Change point limit theorem for vanishing jump height}\label{sec:clt}
Our consistency results from the previous section require that the jump height $\eta$ of the diffusivity does not converge to zero, which simplifies the change point identification task. Under a vanishing jump size asymptotics $\eta=\eta(\delta)\to 0$, we can attain the optimal rate $\eta^{-2}\delta^3$ in the model problem, provided that rate is larger than the observation distance $\delta$, i.e., in the regime $\eta=o(\delta)$.
We obtain precise weak convergence properties of the change point estimator in the vanishing jump height regime
\[
\eta = \eta(\delta) \coloneqq \theta_+(\delta) - \theta_-(\delta) \to 0.
\]
Throughout this section, we work in the setting where the diffusivity constants $\theta_\pm(\delta)$ are known, $\theta_+(\delta)>\theta_-(\delta)$ holds  and where
$\lim_{\delta \to 0} \theta_-(\delta) = \lim_{\delta \to 0} \theta_+(\delta) = \difflim$ for some $\difflim \in [\lowell,\upell]$.
Note that, compared to the previous section, we drop the $0$-superscripts in the notation of the true parameters in the following as no confusion will arise from this.
Based on the modified local $\log$-likelihoods
\[\ell_{\delta,i} \coloneqq \theta_\delta(x_i) \int_0^T X^{\Delta}_{\delta,i}(t) \diff{X_{\delta,i}(t)} - \frac{\theta_\delta(x_i)^2}{2}\int_0^T X^{\Delta}_{\delta,i}(t)^2 \diff{t},\]
we employ the same estimation approach as before, but using the known diffusivity parameters.
Hence, we estimate $\tau$ by $\hat\tau = \hat{\tau}^\delta \coloneqq \hat k  \delta$, where
\begin{equation}\label{eq:cpest}
	\begin{split}
		\hat{k} = \hat{k}(\delta) &\coloneqq \argmaxalt_{k=1,\ldots,n} \Big\{\sum_{i=1}^{k} \Big(\theta_-(\delta)\int_0^T X^{\Delta}_{\delta,i}(t) \diff{X_{\delta,i}(t)} - \frac{\theta_-(\delta)^2}{2} \int_0^T X^{\Delta}_{\delta,i}(t)^2 \diff{t}\Big) \\
		&\quad\quad\quad\quad\quad +  \sum_{i=k+1}^{n} \Big(\theta_+(\delta) \int_0^T X^{\Delta}_{\delta,i}(t) \diff{X_{\delta,i}(t)} - \frac{\theta_+(\delta)^2}{2} \int_0^T X^{\Delta}_{\delta,i}(t)^2 \diff{t}\Big)\Big\}.
	\end{split}
\end{equation}
As in the model problem and in Section \ref{sec:est}, it will be convenient to rewrite the estimator.
Recall that $\ko = \ko(\delta) \coloneqq \ceil{\tau/\delta}$.

\begin{lemma}\label{lem:cp}
	We have
	\begin{equation}\label{eq:cp2}
		\hat{k}  = \argmaxalt_{k = 1,\ldots,\delta^{-1}} Z_k,
	\end{equation}
	where
	\[Z_k = \begin{cases} 0, & k = {\ko},\\ \eta\sum_{i=k + 1}^{\ko} \int_0^T X_{\delta,i}^\Delta(t) \diff{B_{\delta,i}(t)} - \frac{\eta^2}{2}\sum_{i=k+1}^{\ko}  \int_0^T X_{\delta,i}^\Delta(t)^2 \diff{t} + \eta R_{\delta,\ko}, & k < \ko,\\ -\eta\sum_{i=\ko+ 1}^{k} \int_0^T X_{\delta,i}^\Delta(t)\diff{B_{\delta,i}(t)} - \frac{\eta^2}{2}\sum_{i=\ko+1}^{k}  \int_0^T X_{\delta,i}^\Delta(t)^2 \diff{t}, & k > \ko,\end{cases}\]
	with $R_{\delta,\ko}$ from \eqref{eq:remain}.
\end{lemma}
\begin{proof}
	See Appendix \ref{app:proof_clt}.
\end{proof}

Our final main result is the following limit theorem.

\begin{theorem}\label{theo:clt}
	Assume that $\eta= o(\delta)$ and $\delta^{3/2} = o(\eta)$ such that $v_\delta \coloneqq \delta^3/\eta^2\to 0$.
	Then, for a two-sided Brownian motion $(B^{\leftrightarrow}(h),\, h\in \R)$, we have
	\begin{equation}\label{eq:clt}
		v_\delta^{-1}\frac{T  \lVert K^\prime \rVert^2_{L^2}}{2\difflim}(\hat{\tau} - \tau) \overset{\mathrm{d}}{\longrightarrow} \argmin_{h \in \R} \Big\{B^{\leftrightarrow}(h) + \frac{\lvert h \rvert}{2} \Big\}, \quad\text{ as }\delta\to0.
	\end{equation}
\end{theorem}

The proof of Theorem \ref{theo:clt}, which is given at the end of the section, relies on a functional limit theorem, but we need some preparatory technical results first.
The cornerstones of the proof are provided by the argmax theorem presented in Section 3.2.1 in \cite{vdv96}, but considerable effort is needed to fit the underlying idea into our context.
Let us describe the objects which are at the heart of the analysis.
Under the assumptions of Theorem \ref{theo:clt}, $v_\delta \to 0$ and $v_\delta/\delta \to \infty$ as $\delta \to 0$.
In analogy to the piecewise constant processes $M_{T,\delta}(\chi) , I_{T,\delta}(\chi)$ from the previous section, introduce
\[
M^\prime_{T,\delta}(h)\coloneqq\sum_{i=1}^{\ceil{h/\delta}}M_{\delta,i},\qquad
I^\prime_{T,\delta}(h)\coloneqq
\sum_{i=1}^{\ceil{h/\delta}}I_{\delta,i}, \quad h \in [0,1],
\]
with $M_{\delta,i}$ and $I_{\delta,i}$ as defined in \eqref{eq:MI}.
For $h\in\intv\coloneqq [-\tau/v_\delta,(1-\tau)/v_\delta]$, let
\begin{align}\label{eq:Mtau}
	M_{T,\delta}^\tau(h)&\coloneqq  M^\prime_{T,\delta}(\tau+hv_\delta)\\\nonumber
	&\qquad-\Big(M^\prime_{T,\delta}(\tau) - \int_0^T X^{\Delta}_{\delta,\ko}(t) \diff{B_{\delta,\ko}}(t)  \one_{\{\ceil{\tau/\delta} - \ceil{(\tau+hv_{\delta})/\delta} > 0\}}\Big),\\\label{eq:Itau}
	I_{T,\delta}^\tau(h)&\coloneqq \Big\lvert I^\prime_{T,\delta}(\tau+hv_\delta)-\Big(I^\prime_{T,\delta}(\tau) - \int_0^T X^{\Delta}_{\delta,\ko}(t)^2  \diff{t} \one_{\{\ceil{\tau/\delta} - \ceil{(\tau+hv_{\delta})/\delta} > 0\}}\Big)\Big\rvert.
\end{align}
For $h \notin \intv$, we simply set $M^\tau_{T,\delta}(h) = I^{\tau}_{T,\delta}(h) = 1$.
Note that no term involving an observation block around the change point is present in the processes $M^{\tau}_{T,\delta}(h)$ and $I^{\tau}_{T,\delta}(h)$ and that we can write
\begin{align}\nonumber
	\max_{k=1,\ldots, \delta^{-1}} Z_k &= \max_{h \in \intv} \Big\{-\Big(\eta M^{\tau}_{T,\delta}(h) + \frac{\eta^2}{2} I^{\tau}_{T,\delta}(h) \Big) + \big(\eta M_{\delta,\ko} - \frac{\eta^2}{2} I_{\delta,\ko} + \eta R_{\delta,\ko}\big)\one_{\{h < 0\}} \Big\}\\\nonumber
	&=-\min_{h \in \R} \Big\{\eta M^{\tau}_{T,\delta}(h) + \frac{\eta^2}{2} I^{\tau}_{T,\delta}(h) \Big\} + \mathcal{O}_{\PP}(\eta^2\delta^{-2})\\\nonumber
	&=-\min_{h \in \R} \Big\{\eta M^{\tau}_{T,\delta}(h) + \frac{\eta^2}{2} I^{\tau}_{T,\delta}(h) \Big\} + \mathcal{O}_{\PP}(\eta^2\delta^{-2})\\\label{eq:Z_delta}
	&\eqqcolon -\min_{h \in \R} \mathcal{Z}_{\delta}(h) + \mathcal{O}_{\PP}(\eta^2\delta^{-2}),
\end{align}
where the $\mathcal{O}_{\PP}$ term is with respect to uniform convergence and follows from the facts that $I_{\delta,\ko} = \mathcal{O}_{\PP}(\delta^{-2})$ and $\eta R_{\delta,\ko} = \mathcal{O}_{\PP}(\eta^2\delta^{-2})$ according to Lemma \ref{lem:basic_est} and Proposition \ref{lem:remainder}, respectively.
For the second equality, we used that $M^\tau_{T,\delta}(h) = I^{\tau}_{T,\delta}(h) = 1$ for $h \in \R \setminus \intv$ and $M^\tau_{T,\delta}(0) = I^{\tau}_{T,\delta}(0) = 0$, which guarantees that the minimum is attained in $\intv$.

The argmax theorem requires to establish weak convergence of the sequence $(\mathcal Z_\delta(h),h\in\R)$ to a limit process $(\mathcal Z(h),h\in\R)$ with respect to the topology of uniform convergence on compacts.
For doing so, we will exploit the fact that $M_{T,\delta}^\tau(h)$ is a continuous martingale in $T$ with quadratic variation $I_{T,\delta}^\tau(h)$.
Regarding convergence of the finite-dimensional distributions, we might therefore refer to a classical martingale CLT \cite[Theorem 5.5.4]{liptser89} which asserts that, for fixed $h \in \intv$,
\[ \frac{M_{T,\delta}^\tau(h)}{I_{T,\delta}^\tau(h)^{1/2}}\overset{\mathrm{d}}{\longrightarrow} N(0,1)\quad \text{ and } \quad \frac{M_{T,\delta}^\tau(h)}{\E[I_{T,\delta}^\tau(h)]^{1/2}}\overset{\mathrm{d}}{\longrightarrow} N(0,1) \quad \text{ as } \delta\to 0,\]
provided that
\begin{equation}\label{eq:pconv}
	\frac{I_{T,\delta}^\tau(h)}{\E[I_{T,\delta}^\tau(h)]}\overset{\PP}{\longrightarrow} 1, \quad \text{ as }\delta \to 0.
\end{equation}
Applying the moment bounds from Lemma \ref{lem:basic_est}, we employ the analogue of the Kolmogorov--Chentsov criterion in Skorokhod topology to prove tightness of the suitably scaled processes.
Based on these observations, we obtain, as a first ingredient, the following functional limit theorem for the martingale part.

\begin{proposition} \label{prop:martclt}
	Assume that $\eta = o(\delta)$ and $\delta^{3/2} = o(\eta)$.
	Then, for a two-sided Brownian motion $(B^\leftrightarrow(h))_{h \in \R}$ and $v_\delta= \delta^3\eta^{-2}$, we have
	\[ (\delta^{3/2}v_\delta^{-1/2}M_{T,\delta}^\tau(h),\, h\in \R) \overset{\mathrm{d}}{\longrightarrow} \Big(\sqrt{T/(2\difflim)}\norm{K'}_{L^2} B^\leftrightarrow(h),\,h\in\R\Big), \quad \text{ as }\delta \to 0,\]
	in  the Skorokhod space  $\mathbb{D}(\R)$, which in view of the continuous limiting law is equivalent to uniform convergence on compacts.
\end{proposition}
\begin{proof}
	From Lemma \ref{lem:basic_est}, we know that, for fixed $h >0$ and $\delta$ small enough,
	\[\Var(I_{T,\delta}^\tau(h))\lesssim \delta^{-2}\abs{\ceil{(\tau+hv_\delta)/\delta}-\ceil{\tau/\delta}}, \]
	while $\E[I_{T,\delta}^\tau(h)]^2\thicksim (\delta^{-2}(\ceil{(\tau+hv_\delta)/\delta}-\ceil{\tau/\delta}))^2$ is of much larger order since $v_\delta/\delta =\delta^2\eta^{-2}\to \infty$. This implies for all $h \neq 0$ (note that $h=0$ will be trivial in the sequel) that \eqref{eq:pconv} holds true.
	For $\delta$ small enough and fixed $h \neq 0$, we have the representation
	\[\E[\delta^3v_\delta^{-1} I_{T,\delta}^\tau(h)] = \frac T2 \babs{\sum_{i=\ceil{\tau/\delta}}^{\ceil{(\tau+hv_\delta)/\delta}}\frac{\delta}{v_\delta} \frac{\lVert K^\prime \rVert^2_{L^2}}{\theta(x_i)} \one_{\{i \neq \ceil{\tau/\delta}}\}} + o(1),\]
	which, as $\delta \to 0$, converges to $T/(2\difflim) \abs{h}\norm{K'}^2_{L^2}$.
	By the martingale CLT stated as Theorem 5.5.4 in \cite{liptser89}, we therefore obtain
	\[ \delta^{3/2}v_\delta^{-1/2} M_{T,\delta}^\tau(h)\overset{\mathrm{d}}{\longrightarrow} N\Big(0,\frac{T}{2\difflim}\norm{K'}^2_{L^2}\abs{h}\Big).
	\]
	For $-\infty<h_0<h_1<\cdots<h_J<\infty$ and $\beta_j\in\R$, all fixed, it follows similarly
	\[ \E\Big[\delta^3v_\delta^{-1}\sum_{j=1}^J\beta_j(I_{T,\delta}^\tau(h_j)-I_{T,\delta}^\tau(h_{j-1})) \Big]
	\longrightarrow \frac{T}{2\difflim} \norm{K'}^2_{L^2}  \sum_{j=1}^J\beta_j (h_j-h_{j-1}), \quad \text{ as } \delta\to0,
	\]
	as well as
	\begin{align*} \Var\Big(\delta^3v_\delta^{-1}\sum_{j=1}^J\beta_j(I_{T,\delta}^\tau(h_j)-I_{T,\delta}^\tau(h_{j-1})) \Big)&\lesssim \delta^4v_\delta^{-2}\sum_{j=1}^J\beta_j^2(h_jv_\delta/\delta-h_{j-1}v_\delta/\delta)\\
		&\lesssim \delta^3v_\delta^{-2} =\delta^{-3}\eta^4,
	\end{align*}
	which tends to zero.
	By the Cram\'er--Wold device, we thus deduce the convergence of the finite-dimensional distributions to a Gaussian process:
	\[ (\delta^{3/2}v_\delta^{-1/2}M_{T,\delta}^\tau(h),\, h\in \R)\overset{\mathrm{f.d.}}{\longrightarrow} \Big(\sqrt{T/(2\difflim)}\norm{K'}_{L^2}  B^{\leftrightarrow}(h),\,h\in\R\Big), \quad \text{ as } \delta\to 0,\]
	where $B^{\leftrightarrow}$ denotes a two-sided Brownian motion.
	
	The limiting process has continuous trajectories.
	Consequently, it suffices to prove distributional convergence in the Skorokhod topology, that is weak convergence in $\mathbb{D}([-m,m])$ for any $m > 0$, cf.\ \cite[Theorem 16.7]{bill99}. Given the convergence of the finite-dimensional distributions, it remains to prove tightness of the law of \begin{equation}\label{eq:Mrestr}
		(\delta^{3/2}v_\delta^{-1/2}M^\tau_{T,\delta}(h),\, h \in [-m,m])
	\end{equation}
	in $\mathbb{D}([-m,m])$ for any $m > 0$.
	The standard criterion \cite[Theorem 13.5]{bill99} tells us that tightness will follow  from verifying that, for $-m \le x\le y\le z\le m$,
	\begin{equation}\label{eq:momcrit}
		\E[\delta^6v_\delta^{-2}(M_{T,\delta}^\tau(z)-M_{T,\delta}^\tau(y))^2(M_{T,\delta}^\tau(y)-M_{T,\delta}^\tau(x))^2]\lesssim (z-x)^2.
	\end{equation}
	Note first that, for $z-x<\delta v_\delta^{-1}$, the left-hand side is zero because one of the factors in the argument of the expectation vanishes by the piecewise constant nature of $M_{T,\delta}$.
	Generally, the Cauchy--Schwarz inequality yields for the term on the left-hand side the upper bound
	\[\E[\delta^6v_\delta^{-2}(M_{T,\delta}^\tau(z)-M_{T,\delta}^\tau(y))^4]^{1/2}
	\E[\delta^6v_\delta^{-2}(M_{T,\delta}^\tau(y)-M_{T,\delta}^\tau(x))^4]^{1/2}.
	\]
	By the Burkholder--Davis--Gundy inequality and the bounds from Lemma \ref{lem:basic_est}, we obtain for any $\delta$ which is small enough to ensure
	$[-m,m]\subset\intv$ that
	\begin{align*}
		&\E[(M_{T,\delta}^\tau(y)-M_{T,\delta}^\tau(x))^4]\\ 
		&\quad\lesssim \E[(I_{T,\delta}^\tau(y)-I_{T,\delta}^\tau(x))^2]= \E\Big[\Big(\sum_{i=\ceil{(\tau+xv_\delta)/\delta}+1}^{\ceil{(\tau+yv_\delta)/\delta}}I_{\delta,i} \one_{\{i \neq \ceil{\tau/\delta}\}} \Big)^2\Big]\\
		&\quad= \Big(\sum_{i=\ceil{(\tau+xv_\delta)/\delta}+1}^{\ceil{(\tau+yv_\delta)/\delta}}\E[I_{\delta,i}] \one_{\{i \neq \ceil{\tau/\delta}\}}\Big)^2+ \Var\Big(\sum_{i=\ceil{(\tau+xv_\delta)/\delta}+1}^{\ceil{(\tau+yv_\delta)/\delta}}I_{\delta,i} \one_{\{i \neq \ceil{\tau/\delta}\}}\Big)\\
		&\quad\lesssim \big(\delta^{-2}(\ceil{(\tau+yv_\delta)/\delta}-\ceil{(\tau+xv_\delta)/\delta})\big)^2+\delta^{-2}(\ceil{(\tau+yv_\delta)/\delta}-\ceil{(\tau+xv_\delta)/\delta})\\
		&\quad\thicksim \delta^{-4}(\ceil{(\tau+yv_\delta)/\delta}-\ceil{(\tau+xv_\delta)/\delta})^2.
	\end{align*}
	Applying this to both factors and using $x\le y\le z$, we deduce
	\begin{align*}
		&\E[\delta^6v_\delta^{-2}(M_{T,\delta}^\tau(z)-M_{T,\delta}^\tau(y))^2(M_{T,\delta}^\tau(y)-M_{T,\delta}^\tau(x))^2]\\
		&\lesssim
		\big(\ceil{(\tau+zv_\delta)/\delta}\delta v_\delta^{-1}-\ceil{(\tau+xv_\delta)/\delta}\delta v_\delta^{-1}\big)^2.
	\end{align*}
	For $z-x\ge\delta v_\delta^{-1}$, we have
	\[\ceil{(\tau+zv_\delta)/\delta}-\ceil{(\tau+xv_\delta)/\delta}\le (z-x)v_\delta/\delta+1\le 2(z-x)v_\delta/\delta.\]
	This establishes the moment criterion \eqref{eq:momcrit} and hence tightness of the law of the restricted process \eqref{eq:Mrestr} for any $m > 0$, as desired.
\end{proof}

Let $Z_k$ be the random variables from the statement of Lemma \ref{lem:cp}.
Since
\[
\hat{\tau} = \hat{k}\delta = \Big(\argmaxalt_{k=1,\ldots,\delta^{-1}}Z_k\Big) \delta,
\]
we obtain from \eqref{eq:Z_delta} that
\begin{equation}
	\begin{split} \label{eq:min}
		\mathcal{Z}_\delta(v_\delta^{-1}(\hat{\tau} - \tau)) &= \eta M^{\tau}_{T,\delta}(v_\delta^{-1}(\hat{\tau} - \tau)) + \frac{\eta^2}{2} I^{\tau}_{T,\delta}(v_\delta^{-1}(\hat{\tau} - \tau)) = - Z_{{\hat{k}}} + \mathcal{O}_{\PP}(\eta^2\delta^{-2})\\
		&= \min_{h \in \R} \mathcal{Z}_{\delta}(h) + o_{\PP}(1).
	\end{split}
\end{equation}
As a second main ingredient to the proof of Theorem \ref{theo:clt}, we must verify the tightness criterion stated in \cite[Theorem 3.2.2]{vdv96} for the change point estimator, which, in view of the above representation, boils down to studying the sequence
\begin{equation}\label{eq:tighttau}
	(v_{\delta}^{-1}(\hat{\tau} - \tau), \, \delta \in 1/\N)
\end{equation}
For this purpose, the coupling technique for the martingale terms presented in Section \ref{sec:conc} plays a central role.

\begin{proposition}\label{prop:tight}
	Suppose that $\eta = o(\delta)$ and $\delta^{3/2} = o(\eta)$. 
	Then, the sequence \eqref{eq:tighttau} is tight.
\end{proposition}
\begin{proof}
	See Section \ref{app:proof_clt}.
\end{proof}

We are  finally ready to prove Theorem \ref{theo:clt}.
\begin{proof}[Proof of Theorem \ref{theo:clt}]
	Note first that, for any $m > 0$ and $\delta > 0$ small enough to ensure that $[-m,m] \subset \intv$, an application of Proposition \ref{prop:conc} and a union bound using
	\[
	\lvert\{\ceil{\tau + hv_\delta/\delta}: h \in [-m,m]\}\rvert \leq (2m+1)v_{\delta}/\delta
	\]
	show that
	\begin{align*}
	&	\PP\Big(\max_{h \in [-m,m]} \lvert I^\tau_{T,\delta}(h) - \E[I^\tau_{T,\delta}(h)]\rvert  \geq z \Big) \\
	&\quad\leq (2m+1)\frac{v_{\delta}}{\delta} \exp\Big(-\frac{\lowell^2}{2} \frac{z^2}{2z + (2m+1)v_\delta\delta^{-3}T\lowell^{-1}\lVert K^\prime \rVert^2_{L^2}} \Big),
	\end{align*}
	implying that
	\[\max_{h \in [-m,m]} \lvert I^\tau_{T,\delta}(h) - \E[I^\tau_{T,\delta}(h)]\rvert = \mathcal{O}_{\PP}(\delta^{-3/2}(v_{\delta}\log(v_\delta/\delta))^{1/2}) = \mathcal{O}_{\PP}(\eta^{-1}(\log(\delta/\eta))^{1/2}). \]
	Thus, for $\mathcal Z_\delta$ defined in \eqref{eq:Z_delta},
	\begin{align*}
		\big(\mathcal{Z}_{\delta}(h),\, h \in \R \big) &= \big(\eta M^{\tau}_{T,\delta}(h) + \frac{\eta^2}{2} \E[I^{\tau}_{T,\delta}(h)], \, h \in \R \big) + \mathcal{O}_{\PP}(\eta\log(\delta/\eta)^{1/2} )\\
		&= \big(\eta M^{\tau}_{T,\delta}(h) + \frac{\eta^2}{2} \E[I^{\tau}_{T,\delta}(h)], \, h \in \R \big) + o_{\PP}(1),
	\end{align*}
	where $\mathcal{O}_{\PP}$ and $o_{\PP}$ are with respect to uniform convergence on compacts.
	Using Lemma \ref{lem:basic_est}, one obtains
	\[ \eta^2 \E[I^\tau_{T,\delta}(h)] = \eta^2 \delta^{-3} v_\delta T\lvert h \rvert \frac{\lVert K^\prime \rVert^2_{L^2}}{2\difflim} +o(1) = T\lvert h \rvert \frac{\lVert K^\prime \rVert^2_{L^2}}{2\difflim} + o(1),\]
	where $o(\cdot)$ is with respect to uniform convergence on compacts. Moreover, Proposition \ref{prop:martclt} yields that, for some two-sided Brownian motion $\tilde{B}^{\leftrightarrow}$,
	\[(\eta M^{\tau}_{T,\delta}(h),\, h \in \R) \overset{\mathrm{d}}{\longrightarrow} \Big(\sqrt{\frac{T \norm{K'}_{L^2}^2}{2\difflim}}  \tilde{B}^{\leftrightarrow}(h), h \in \R\Big), \quad \text{ as }\delta \to 0,\]
	uniformly on compacts. Thus,  it follows
	\[(\mathcal{Z}_\delta(h),\, h \in \R) \overset{\mathrm{d}}{\longrightarrow} \Big(\sqrt{\frac{T\norm{K'}_{L^2}^2}{2\difflim}} \tilde{B}^{\leftrightarrow}(h) + \frac{T \lVert K^\prime \rVert^2_{L^2}}{4 \difflim}\lvert h \rvert,\, h \in \R\Big),\]
	uniformly on compacts.
	In addition, Proposition \ref{prop:tight} demonstrates that $(v_{\delta}^{-1}(\hat{\tau} - \tau), \delta \in 1/\N)$ is tight and, by \cite[Theorem 2.1]{williams74} (see also \cite[Remark 1.2]{bhatt76}), the limiting process of $\mathcal{Z}_{\delta}$ almost surely has a unique (random) minimiser $\hat{h} \in \R$.
	Therefore, taking into account \eqref{eq:min}, the argmin continuous mapping theorem (\cite[Theorem 3.2.2]{vdv96}) implies that
	\[v_{\delta}^{-1}(\hat{\tau} - \tau)  \overset{\mathrm{d}}{\longrightarrow} \argmin_{h \in \R} \Big\{\sqrt{\frac{T\norm{K'}_{L^2}^2}{2\difflim}}  \tilde{B}^{\leftrightarrow}(h) + \frac{T \lVert K^\prime \rVert^2_{L^2}}{4 \difflim}\lvert h \rvert \Big\}, \quad \text{ as }\delta \to 0.\]
	Substituting $\bar h=T\norm{K'}^2h/(2\theta_\ast)$, we arrive with a new two-sided Brownian motion $B^{\leftrightarrow}(\bar h)\coloneqq\sqrt{\bar h/h} \tilde{B}^{\leftrightarrow}(h)$ at the asserted limit result, i.e.,
	\[\delta^{-3}\eta^2\frac{T\norm{K'}^2}{2\theta_\ast } (\hat\tau-\tau) \overset{\mathrm{d}}{\longrightarrow} \argmin_{\bar h \in \R}\Big\{B^{\leftrightarrow}(\bar h)+\frac{1}{2} \abs{\bar h}\Big\}, \quad \text{ as }\delta \to 0.
	\]
\end{proof}

\section{Discussion}\label{sec:disc}
This paper presents a change-point model in a framework not previously considered, where the analysis relies on a  combination of a wide range of statistical principles and methods from operator and probability theory. 
\paragraph*{Consistency and rates of convergences}
The proof of consistency (see Theorem \ref{theo:cons} for the concrete statement) is based on ideas underlying the consistency theorem for M-estimators.
In the case of change point analysis based on independent observations (as carried out in Section 14.5.1 of \cite{kosorok08}), the verification of the required conditions can draw on well-known results concerning uniform convergence and on established concepts such as Glivenko--Cantelli or Donsker classes.
In contrast, in our framework, we first need to perform a careful study of the concentration properties of the components that appear in the (modified) $\log$-likelihood function determining our estimators.
In particular, it is crucial to understand the concentration behaviour of linear combinations of certain martingales $M_{\delta,i}$ and quadratic variation terms $I_{\delta,i}$, $i=1,\ldots,\delta^{-1}$, and in both cases the investigation requires new ideas:
\begin{enumerate}\vspace*{-.5em}
    \item[(a)] The individual martingale terms $M_{\delta,i}$ are not independent, which is why we use a coupling approach based on the Dambis--Dubins--Schwarz theorem to prove the deviation inequality stated in Proposition \ref{prop:coupling}.
    \item[(b)] The random variables $I_{\delta,i}$ are not independent either and, moreover, determined by nonlinear unbounded quadratic functionals, whose joint concentration properties are established by methods from Malliavin calculus.
This allows us to prove a Bernstein inequality (cf.\ Proposition \ref{prop:conc}), which is essential for verifying optimal convergence rates of our estimators.
\end{enumerate}\vspace*{-.5em}
Indeed, if the jump height $\eta(\delta) = \lvert \theta^0_+(\delta) - \theta^0_-(\delta) \rvert$ between the true parameters does not vanish in the limit $\delta \to 0$, we prove that our approach yields estimators that achieve the optimal error rates 
\[
	\lvert \hat{\theta}_\pm^\delta - \theta^0_{\pm}(\delta) \rvert = \mathcal{O}_{\PP}(\delta^{3/2}) \quad \text{and} \quad \lvert \hat{\tau}^\delta - \tau^0 \rvert = \mathcal{O}_{\PP}(\delta).
\]
For the change point estimator $\hat{\tau}^\delta$, this is the best precision we can hope for because our observations are at distance $x_i-x_{i-1}=\delta$ in space. The diffusivity parameter estimators $\hat{\theta}^\delta_{\pm}$ converge at rate $\delta^{3/2}$, which is the optimal rate for estimation of the constant diffusivity without change point based on multiple local measurements, as shown in \cite{altmeyer22}  in a general context. The introduction of the nuisance parameter $\theta_\circ$ in the estimation procedure is fundamentally important to achieve this optimal rate since it allows us to sufficiently reduce the bias originating from the constant approximation of the diffusivity on a change point interval in the modified log-likelihood.

\paragraph*{On the change point limit theorem}
A natural question is what size $\eta$ the jump must have, so that we can detect the change point at a given resolution level $\delta$. Our second main result, a limit theorem for the change point estimator for vanishing jump height $\eta$, provides an answer to this.
For ease of exposition, we assume here that the diffusivity parameters $\theta^0_\pm(\delta)$ are known. Our analysis  shows that $\hat\tau$ detects the change point consistently as soon as $\eta/\delta^{3/2}\to\infty$ and $\eta/\delta \to 0$, in which case 
\begin{equation}\label{eq:disc_limit}
	\frac{\eta^2}{\delta^3} \frac{T \lVert K^\prime \rVert^2_{L^2}}{2\difflim} (\hat{\tau}^\delta - \tau^0) \overset{\mathrm{d}}{\longrightarrow} \argmin_{h \in \mathbb{R}} \left\{ B^{\leftrightarrow}(h) + \frac{\lvert h \rvert}{2} \right\}.
\end{equation}
This enables us to construct approximate confidence intervals for the change point $\tau^0$ in situations when the signal is weak (i.e., $\eta$ is very small), which is typical in challenging applied scenarios that necessitate a statistical approach, such as detection of minor material contamination. For this purpose, it is particularly important to note that the limiting distribution which we identify in \eqref{eq:disc_limit} is known explicitly (cf.\ \cite[Lemma 1.6.3]{csorgo97}), allowing for the simple construction of asymptotic confidence intervals.  The condition $\eta = o(\delta)$ for our limit theorem is necessary since the discretisation error $\delta$ must be of lower order than the convergence rate $\eta^{-2}\delta^3$, that is $\delta/(\eta^{-2}\delta^3)\to 0$ or $\eta=o(\delta)$. For asymptotically larger or even non-vanishing jump height, the asymptotic distribution at error rate $\delta$ will degenerate because of the deterministic discretisation error, which is unknown to the statistician. 

\paragraph*{Perspectives}
The analysis performed for the basic model \eqref{eq:spde} can be extended to more general SPDE models. A notable example is the \emph{multivariate extension} investigated in the recent preprint \cite{tietrott24}, where the parametric change point estimation problem in the non-vanishing signal case is translated into the nonparametric task of estimating a \emph{change domain}. The multivariate analysis substantially draws on the concentration results of Section \ref{sec:conc}, which can be straightforwardly extended to the multivariate model since the spectral calculus techniques employed are dimension independent. 

The techniques developed in this paper also provide a foundation for more flexible models with tractable asymptotic behaviour by relaxing the assumptions on the diffusivity.
The change point estimation scheme described in Section \ref{sec:defbasic}, whose behaviour under rescaling can then be analysed in more detail, serves as a starting point for further exploration of these extensions.

In the vanishing jump height regime, we work under the simplifying assumption that the diffusivity constants $\theta_\pm$ are known. A natural extension is to investigate the asymptotic behaviour of an appropriately rescaled version of our simultaneous estimator $(\hat{\theta}_-,\hat{\theta}_+,\hat{\tau})$. \cite[p.37]{bhatt94} suggests that this approach works well in related one-dimensional iid settings, with weak convergence of rescaled simultaneous estimators yielding normal limits for $\theta_\pm$ and the normalised change point converging to the minimiser of a two-sided Brownian motion with drift. In the absence of a change point, asymptotic normality for the diffusivity parameter holds when rescaled at the convergence rate $\delta^{-3/2}$, as shown by \cite{altmeyer22}.
Finally, in practical applications, the continuous-time processes studied in this paper are, of course, observed on discrete grids. Regarding the implementation of our proposed procedure, we refer to the discussion of practical aspects in Section 5.3 in \cite{altmeyer22}.

\begin{appendix}
\section{Analytical results}\label{app:proof_ana}
Table \ref{table:notation} summarises important notation used throughout the paper. 
{\color{black}\begin{table}
	\begin{minipage}[!htbp]{0.45\textwidth}
	\begin{tabularx}{\textwidth}{@{}p{0.25\textwidth}X@{}}
	\toprule
	$B^\leftrightarrow$& two-sided Brownian motion\\
	$\delta$ & resolution level \\
	$\Delta_\vartheta$ & weighted Laplace operator\\
	$\eta$ & jump height\\
	$\theta$ & diffusivity\\
	$\theta_-,\theta_+$ & diffusivity constants\\ 
	$\underline\theta,\overline\theta$ & lower/upper bound on $\theta_\pm$\\
	$\theta_{\delta,i}(k),\theta_{\delta,i}^0$ & cf.~\eqref{eq:deftdi}\\
	$I_{\delta,i}$ & quadratic variation (cf.~\eqref{eq:MI})\\
	$k_\bullet$ & $=k_\bullet(\delta)=\lceil \tau/\delta\rceil$\\
	$[x]_\delta$ & $=\delta\lceil x/\delta\rceil$\\
	\bottomrule
	\end{tabularx}
	\end{minipage}\hfill
	\begin{minipage}[!htbp]{0.45\textwidth}	
	\begin{tabularx}{\textwidth}{@{}p{0.15\textwidth}X@{}}
	\toprule
	$K_{\delta,i}(\cdot)$ & localised kernel functions \\
	$\ell_{\delta,i}(\cdot)$& modified $\log$-likelihood (cf.~\eqref{eq:loglike})\\
	$M_{\delta,i}$ & martingale term (cf.~\eqref{eq:MI})\\
	$n$ & number of observation points \\
	$R_{\delta,\koo}(\cdot)$ & remainder (cf.~\eqref{eq:remain})\\
	$T$   & observation time \\
	$\tau$ & change point\\
	$\dot W$ & space-time white noise\\
	$X_{\delta,i}(t)$ & local observations of $X(t)$ wrt $K_{\delta,i}$\\
	$X_{\delta,i}^\Delta(t)$ & local obs.~of $X(t)$ wrt $\Delta K_{\delta,i}$\\
	\bottomrule
	\end{tabularx}
	\end{minipage}
	\vspace{5pt}
	\caption{Notation}
	\label{table:notation}
	\end{table}}
The rest of this section comprises analytical results that are essential for the subsequent statistical analysis.
	\begin{proof}[Proof of Proposition \ref{prop:weaksol}]
		Selfadjointness and the semigroup property imply
		\[\lVert S_\vartheta(t) \rVert_{\mathrm{HS}(L^2(\modelspace))}^2 = \int_{(0,1)^2} p^\vartheta_t(x,y)^2 \diff{x} \diff{y} = \int_{\modelspace} p^\vartheta_{2t}(x,x) \diff{x} \leq \frac{C_1(\vartheta)}{\sqrt{2 \pi t}}.\]
		Consequently,
		\[\int_0^T \lVert S_\vartheta(t) \rVert^2_{\mathrm{HS}(L^2(\modelspace))} \diff{t} < \infty,\]
		such that Theorem 5.4 in \cite{daprato14} yields that \eqref{eq:mild}
		determines the unique weak solution to \eqref{eq:spde}.
		Thus, by self-adjointness of $\Delta_\theta$, for any $z \in D(\Delta^\ast_\vartheta) = D(\Delta_\vartheta)$ and all $t \in [0,T]$, it holds that
		\begin{equation} \label{weak}
			\langle X(t), z \rangle = \langle X_0, z \rangle + \int_0^t \langle X(s), \Delta_\vartheta z \rangle \diff{s} + \langle W(t), z \rangle, \quad \PP\text{-a.s.}
		\end{equation}
		The normalised kernel function $K_{\delta,i}$ is twice continuously differentiable and supported on $[x_i - \delta/2, x_i + \delta/2]$.
		It thus follows for $i \in [n]$  with $\lvert x_i - \tau \rvert \geq \delta/2$ that $K_{\delta,i} \in D(\Delta_\theta)$ and $\Delta_\theta K_{\delta,i} = \theta_{\delta,i}(\ko) \Delta K_{\delta,i}$.
		Therefore, \eqref{eq:weaksol} yields
		\[ X_{\delta,i}(t)=
		\begin{cases}
			\langle X_0, K_{\delta,i} \rangle + \int_0^t \vartheta_-(\delta) X_{\delta,i}^\Delta(t) \diff{t}+ B_{\delta,i}(t),& \text{ if } x_i+\delta/2\le\tau,\\
			\langle X_0, K_{\delta,i} \rangle + \int_0^t \vartheta_+(\delta) X_{\delta,i}^\Delta(t)\diff{t}+  B_{\delta,i}(t),& \text{ if } x_i-\delta/2\ge\tau,
		\end{cases}\]
		where $B_{\delta,i}(\cdot) = \scapro{W(\cdot)}{K_{\delta,i}}$, $i= 1,\ldots,n$, are independent Brownian motions because of $\langle K_{\delta,i}, K_{\delta,j} \rangle = 0$ for $i \neq j$ and $\lVert K_{\delta,i} \rVert = \lVert K \rVert_{L^2} = 1$. Finally, the expression for $X_{\delta,\ko}$ follows from Lemma \ref{lem:mildweak} below.
	\end{proof}

	\begin{lemma} \label{lem:mildweak}
		Let $X$ be the mild solution process \eqref{eq:mild}.
		For any $z \in H_0^1(\modelspace)$, it $\PP$-a.s.\ holds
		\[\scapro{X(t)}{z} = \scapro{X(0)}{z} + \int_0^t \int_0^s \scapro{\Delta_\theta S_\theta(s-u)z}{\diff{W(u)}} \diff{s} +  \scapro{W(t)}{z}, \quad t \in [0,T]. \]
	\end{lemma}
	\begin{proof}
		Let us first show that $\int_0^s \scapro{\Delta_\theta S_\theta(s-u)z}{\diff{W(u)}}$ is well-defined. Since the eigenvalues $(\lambda_k)$ of $-\Delta_\vartheta$ are nonnegative, it follows from \eqref{eq:transform} that, for any $s > 0$ and $z \in L^2(\modelspace)$, we have $S_\theta(s)z \in D(-\Delta_\theta)$.
		Therefore, $\Delta_\theta S_{\theta}(s) z$ is well-defined, and we have
		\[\int_0^t \lVert \Delta_\theta S_\theta(s) z \rVert^2 \diff{s} = \int_0^t \sum_{k \in \N} \lambda_k^2 \mathrm{e}^{-2\lambda_k s} \scapro{e_k}{z}^2 \diff{s} \leq \frac{1}{2}\sum_{k \in \N} \lambda_k \scapro{e_k}{z}^2=\frac12\norm{(-\Delta_\theta)^{1/2}z}^2.\]
		Thus, $\int_0^t \lVert \Delta_\theta S_\theta(s) z \rVert^2 \diff{s}<\infty$ if $z \in D((-\Delta_\theta)^{1/2}) = D(\mathcal{E}_\theta) = H_0^1(\modelspace)$.
		Consequently, for any $z \in H^1_0(\modelspace)$, $\int_0^s \scapro{\Delta_\theta S_\theta(s-u)z}{\diff{W(u)}}$ is well-defined.
		Let now, for fixed $z \in H_0^1$,
		\[\Phi_{s,u}(x) = \langle \Delta_\theta S_{\theta}(s-u)z,x\rangle \one_{[0,s]}(u), \quad u,s \in [0,t], x \in L^2(\modelspace).\]
		Then, using Fubini's theorem,
		\begin{align*}
			\int_0^t \int_0^t \lVert \Phi_{s,u} \rVert^2_{\mathrm{HS}(L^2(\modelspace),\R)} \diff{u} \diff{s}  &= \int_0^t \int_0^s \norm{\Delta_\theta S_{\theta}(s-u)z}^2 \diff{u}\diff{s} 
			\leq 
			\frac{t}{2} \lVert (-\Delta_\theta)^{1/2} z \rVert^2 < \infty.
		\end{align*}
		This allows us to apply the stochastic Fubini theorem (\cite[Theorem 4.33]{daprato14}), and we obtain
		\begin{align*}
			\int_0^t \int_0^s \scapro{\Delta_\theta S_\theta(s-u)z}{\diff{W(u)}} \diff{s}  &= \int_0^t \Big(\int_0^t \Phi_{s,u} \diff{W(u)} \Big) \diff{s}\\
			&= \int_0^t \Big(\int_0^t \Phi_{s,u} \diff{s} \Big) \diff{W(u)}\\
			&= \int_0^t \Big\langle\int_{u}^t \Delta_\theta S_\theta(s-u)z \diff{s},\diff{W}(u)\Big\rangle\\
			&= \int_0^t \scapro{(S_\theta(t-u) - I)z}{\diff{W}_u}\\
			&= \langle X(t), z \rangle - \langle W(t), z \rangle - \langle S_\theta(t)X(0), z \rangle
		\end{align*}
		almost surely.
	\end{proof}

	\begin{lemma} \label{lem:divop}
		If $u \in C^2(\modelspace)$, we have $\Delta_\theta u = \theta \Delta u + \eta u^\prime(\tau) \delta_\tau$ in the sense of distributions, where $\delta_\tau$ is the $\delta$-distribution at $\tau$. 
	\end{lemma}
	\begin{proof}
		Let $z \in C_c^\infty(\modelspace)$. Integration by parts shows that
		\begin{align*}
			\langle \Delta_\theta u, z\rangle &= -\Big(\theta_-(\delta)\int_0^\tau u^\prime(x) z^\prime(x) \diff{x} + \theta_+(\delta)\int_\tau^1 u^\prime(x) z^\prime(x) \diff{x} \Big)\\
			&= \eta u^\prime(\tau)z(\tau) + \int_0^1 \theta(x) u^{\prime \prime}(x) z(x) \diff{x}\\
			&= \langle \theta \Delta u + \eta u^\prime(\tau) \delta_\tau, z\rangle.
		\end{align*}
	\end{proof}
	
	\begin{lemma}\label{lem:sol_c2}
		For $f \in C^2_0((0,1))$ and $x\in(0,1)$, we have
		\begin{equation}\label{eq:sol_c2}
			\begin{split}
				(-\Delta_\theta)^{-1}(f^{\prime \prime})(x) &= \frac{1}{\theta_-} \Big(f(x) - \frac{\theta_+ - \theta_-}{\tau \theta_+ + (1-\tau) \theta_-} f(\tau) x \Big) \one_{(0,\tau)}(x)\\
				&\quad+ \frac{1}{\theta_+} \Big(f(x) + \frac{\theta_+ - \theta_-}{\tau \theta_+ + (1-\tau) \theta_-} f(\tau) (1-x) \Big) \one_{[\tau,1)}(x).
			\end{split}
		\end{equation}
		In particular,
		\begin{equation}\label{eq:sol_c21}
			\lVert (-\Delta_{\theta})^{-1} \Delta K_{\delta,\lceil \tau/\delta \rceil } \rVert \lesssim \delta^{-1/2}.
		\end{equation}
	\end{lemma}
	\begin{proof}
		Let $g$ be the unique function in $D(\Delta_\theta)$ such that $\Delta_\theta g = f^{\prime \prime}$. Such $g$ must satisfy
		\[\theta g' = f^\prime + c,\]
		for some constant $c$, that is,
		\[\nabla g(x) = \begin{cases} \frac{1}{\theta_-}(f^\prime(x) + c), &x \in (0,\tau),\\ \frac{1}{\theta_+}(f^\prime(x) + c), & x \in [\tau,1). \end{cases} \]
		Setting $g(x) = \int_0^x g^\prime(y) \diff{y}$, the constant $c$ is identified via $\int_0^1 g^\prime(y) \diff{y} = 0$ due to the Dirichlet boundary conditions of the operator $\Delta_\theta$, and we obtain $c = \frac{\theta_- - \theta_+}{\tau \theta_+ + (1-\tau) \theta_-} f(\tau)$.
		Straightforward calculations then establish that $g = (-\Delta_\theta)^{-1}(f^{\prime\prime})$ is given by \eqref{eq:sol_c2}.
		For $f = K_{\delta,\ko}$, we have $\lvert f(\tau) \rvert \lesssim \delta^{-1/2}$ and $\lVert f\rVert = 1$, such that \eqref{eq:sol_c21} follows.
	\end{proof}
	
	\begin{lemma}\label{lem:sol_interpol}
		For $f \in C_0^2((0,1))$, we have
		\[\norm{(-\Delta_\theta)^{-3/4}\Delta f}\lesssim \lVert f \rVert^{1/2} \lVert f \rVert^{1/2}_{H^1((0,1))} + \lVert f \rVert_\infty,\]
		for a constant only depending on $\lowell,\upell$. In particular,
		with a constant only depending on $K$ and $\lowell,\upell$,
		\[ \norm{(-\Delta_\theta)^{-3/4}\Delta K_{\delta,\ko}} \lesssim \delta^{-1/2}.\]
	\end{lemma}
	\begin{proof}
		By $(-\Delta_\theta)^{1/2}\le \bar\theta^{1/2}(-\Delta)^{1/2}$, we get
		\begin{align*}
			\norm{(-\Delta_\theta)^{-3/4}\Delta f}_{L^2}^2 &=  \norm{(-\Delta_\theta)^{1/4}(-\Delta_\theta)^{-1}\Delta f}^2\\
			&\le \bar\theta^{1/2}\norm{(-\Delta)^{1/4}(-\Delta_\theta)^{-1}\Delta f}^2\\
			&\le \bar\theta^{1/2}\norm{(-\Delta_\theta)^{-1}\Delta f}_{H^{1/2}((0,1))}^2.
		\end{align*}
		The exact representation of Lemma \ref{lem:sol_c2} yields an $H^{1}$-function $(-\Delta_\theta)^{-1}\Delta f$ for $f \in C_0^2((0,1))$. By the inner description of the $H^{1/2}((0,1))$-norm in terms of first order differences, cf.\ \cite[Section 3.4.2]{triebel10}, we have
		\[ \forall g\in H^{1/2}((0,1)):\; \norm{g}_{H^{1/2}((0,1))}^2\thicksim \norm{g|_{(0,\tau)}}_{H^{1/2}((0,\tau))}^2+ \norm{g|_{(\tau,1)}}_{H^{1/2}((\tau,1))}^2.
		\]
		We obtain the straightforward bounds for the representation in Lemma \ref{lem:sol_c2},
		\begin{align*}
			&\norm{((-\Delta_\theta)^{-1}\Delta f)|_{(0,\tau)}}_{H^{1/2}((0,\tau))}\\
			&\le \theta_-^{-1}\norm{f|_{(0,\tau)}}_{H^{1/2}((0,\tau))}+\frac{\theta_+-\theta_-}{\tau\theta_++(1-\tau)\theta_-}
			\abs{f(\tau)}\norm{x|_{(0,\tau)}}_{H^{1/2}((0,\tau))}\\
			&\lesssim \norm{f|_{(0,\tau)}}_{H^{1/2}((0,\tau))}+\norm{f}_\infty
		\end{align*}
		as well as
		\begin{align*}
			&\norm{((-\Delta_\theta)^{-1}\Delta f|_{(\tau,1)}}_{H^{1/2}((\tau,1))}\\
			&\le \theta_+^{-1}\norm{f|_{(\tau,1)}}_{H^{1/2}((\tau,1))}+\frac{\theta_+-\theta_-}{\tau\theta_++(1-\tau)\theta_-}
			\abs{f(\tau)}\norm{(1-x)|_{(\tau,1)}}_{H^{1/2}((\tau,1))}\\
			&\lesssim \norm{f|_{(\tau,1)}}_{H^{1/2}((\tau,1))}+\norm{f}_\infty
		\end{align*}
		with uniform constants, depending on $\lowell$, $\upell$ only. We thus conclude
		\[ \norm{(-\Delta_\theta)^{-3/4}\Delta f} \lesssim \norm{f}_{H^{1/2}((0,1))}+\norm{f}_\infty \leq \norm{f}^{1/2}\norm{f}_{H^{1}((0,1))}^{1/2} + \norm{f}_\infty.
		\]
		In view of $\norm{K_{\delta,\ko}}_\infty=\delta^{-1/2}\norm{K}_\infty$ and $\norm{K_{\delta,\ko}}_{H^{1}((0,1))} = \delta^{-1} \norm{K}_{H^1}$, this yields
		\[ \norm{(-\Delta_\theta)^{-3/4}\Delta K_{\delta,\ko}} \lesssim \delta^{-1/2}.\]
	\end{proof}

	\section{Proofs for the concentration analysis}\label{app:proof_conc}
	\begin{proof}[Proof of Lemma \ref{lem:basic_est}]
		\begin{enumerate}[label = (\roman*), ref = (\roman*)]
			\item
			It follows from the variation of constants formula that
			\[ X_{\delta,i}^\Delta(t)=\int_0^t \scapro{S_\theta(t-s)\Delta K_{\delta,i}}{\diff{W(s)}}.
			\]
			Let $i,j \in [n]$ with $i,j \neq \ko$.
			Then, $\vartheta \equiv \vartheta_{\delta,i}(\ko)$ on $\mathrm{supp}(K_{\delta,i})$ and, using the selfadjointness of $(S_\theta(t))_{t \geq 0}$, we obtain for $s,t>0$ the covariance
			\begin{equation} \label{eq:cov}
				\begin{split}
					c_{ij}(t,s) &\coloneqq \Cov(X_{\delta,i}^\Delta(t),X_{\delta,j}^\Delta(s))\\
					&= \frac{1}{\theta_{\delta,i}(\ko)\theta_{\delta,j}(\ko)}\int_0^{t\wedge s} \scapro{S_\theta(t-u)\Delta_\theta K_{\delta,i}}{S_\theta(s-u)\Delta_\theta K_{\delta,j}} \diff{u}\\
					&= \frac{1}{\theta_{\delta,i}(\ko)\theta_{\delta,j}(\ko)}\int_0^{t\wedge s}\scapro{S_\theta(t+s-2u)\Delta_\theta K_{\delta,i}}{\Delta_\theta K_{\delta,j}}\diff{u}\\
					&= \frac{1}{2\theta_{\delta,i}(\ko)\theta_{\delta,j}(\ko)}\scapro{(S_\theta(\abs{t-s})-S_\theta(t+s)) K_{\delta,i}}{(-\Delta_\theta)K_{\delta,j}}.
				\end{split}
			\end{equation}
			Therefore, for $i \neq \ko$,
			\begin{equation} \label{eq:exp_quad}
				\begin{split}
					\E[I_{\delta,i}] &= \int_0^T c_{ii}(t,t)\diff{t}\\
					&= \frac{1}{2\vartheta_i^2}\int_0^T\scapro{(\Id-S_\theta(2t)) K_{\delta,i}}{(-\Delta_\theta)K_{\delta,i}} \diff{t}\\
					&= \frac{T}{2\vartheta_i^2}\int_0^1 \theta(x)K_{\delta,i}'(x)^2 \diff{x}-\frac{1}{4\theta_{\delta,i}(\ko)^2}\scapro{(\Id-S_\theta(2T)) K_{\delta,i}}{K_{\delta,i}}\\
					&= \frac{T}{2\vartheta_i} \delta^{-2} \lVert K^\prime \rVert^2_{L^2}+{\mathcal O}(1),
				\end{split}
			\end{equation}
			where, as can be seen from the last but one line, $\mathcal{O}(1) \leq 0$ and $\abs{{\mathcal O}(1)}\le 1/(4\lowell^2)$.
			\item
			Let $\mathcal{E}$ be the Dirichlet form associated to the Laplacian $\Delta$ with Dirichlet boundary conditions, i.e., $D(\mathcal{E}) = H_0^1((0,1))$ and $\mathcal{E}(u,v) = \int_{(0,1)} \nabla u \nabla v \diff{\lambda}$ for $u,v \in H_0^1((0,1))$. Then, $D(\mathcal{E}) = D(\mathcal{E}_\theta)$ and $\upell (-\Delta) \geq -\Delta_\theta \geq \lowell (-\Delta)$ in the sense that
			\[
			\upell \mathcal{E}(u,u) \geq \mathcal{E}_\theta(u,u) \geq \lowell \mathcal{E}(u,u)
			\quad\text{ for any } u \in  H_0^1((0,1)).
			\]
			Thus, by the argument given in Theorem VI.2.21 of \cite{kato95}, see also p.\ 333 of the same reference, it follows that $\upell^{-1}(-\Delta)^{-1} \leq (-\Delta_\theta)^{-1} \leq \lowell^{-1} (-\Delta)^{-1}$. Consequently, recalling that $\lambda_k \geq \underbars{\lambda} > 0$ for any $k \in \N$, we obtain
			\begin{align*}
				\E[I_{\delta,\ko}] &= \int_0^T \E[X^\Delta_{\delta,\ko}(t)^2] \diff{t} = \int_0^T \int_0^t \lVert S_\theta(u) \Delta K_{\delta,\ko} \rVert^2 \diff{u} \diff{t}\\
				&= \frac{1}{2}\int_0^T \big\langle (\mathrm{Id} - S_\theta(2t))(-\Delta_\theta)^{-1}\Delta K_{\delta,\ko} , \Delta K_{\delta,\ko}\big\rangle \diff{t},
			\end{align*}
		and this quantity lies in the interval
				\begin{align*}& \frac{1}{2}\Big[\Big(T- \frac{1-\mathrm{e}^{-2\underbars{\lambda}T}}{2\underbars{\lambda}}\Big)\big\langle (-\Delta_\theta)^{-1} \Delta K_{\delta,\ko}, \Delta K_{\delta,\ko} \big\rangle, T \big\langle (-\Delta_\theta)^{-1} \Delta K_{\delta,\ko}, \Delta K_{\delta,\ko} \big\rangle\Big]  \\
				&\quad\in \Big[\frac{2\underbars{\lambda} T- 1 +\mathrm{e}^{-2\underbars{\lambda}T}}{4\underbars{\lambda}\upell}\big\lVert \nabla K_{\delta,\ko} \big\rVert^2, \frac{T}{2 \lowell} \big\lVert \nabla K_{\delta,\ko} \big\rVert^2\Big] \\
				&\quad= \Big[\frac{2\underbars{\lambda} T- 1 + \mathrm{e}^{-2\underbars{\lambda}T}}{4\underbars{\lambda}\upell}\lVert K^\prime \rVert^2_{L^2} \delta^{-2}, \frac{T}{2 \lowell} \lVert K^\prime \rVert^2_{L^2} \delta^{-2}\Big].
			\end{align*}
			An alternative bound is obtained using Lemma \ref{lem:sol_c2}, namely
			\begin{align*}
				0 \leq \int_0^T \big\langle S_\theta(2t)(-\Delta_\theta)^{-1} \Delta K_{\delta,\ko}, \Delta K_{\delta,\ko} \big\rangle &\leq \frac{1}{2}\big\langle (-\Delta_\theta)^{-2} \Delta K_{\delta,\ko}, \Delta K_{\delta,\ko} \big\rangle\\
				&= \frac12\big\lVert (-\Delta_\theta)^{-1} \Delta K_{\delta,\ko} \big\rVert^2 \lesssim \delta^{-1}.
			\end{align*}
			This yields
			\[\E[I_{\delta,\ko}] \in \Big[\frac{T}{2\upell} \lVert K^\prime\rVert^2_{L^2}\delta^{-2}, \frac{T}{2\lowell}\lVert K^\prime\rVert^2_{L^2} \delta^{-2} \Big] + \mathcal{O}(\delta^{-1}).\]
			\item
			Since $(X^\Delta_{\delta,i}(t))_{t \geq 0, i= 1,\ldots,n}$ is a centered Gaussian process, it follows from Wick's formula (\cite[Theorem 1.28]{janson97}) that
			\[ \Var\Big(\sum_{i=1}^n\alpha_i I_{\delta,i}\Big) = \sum_{i,j=1}^n\alpha_i\alpha_j\int_0^T\int_0^T 2c_{ij}(t,s)^2 \diff{s}\diff{t}.\]
			In the sequel, we employ tensor products $f^{\otimes 2}(x,y)\coloneqq f(x)f(y)$ and $A^{\otimes 2}f^{\otimes 2}\coloneqq (Af)^{\otimes 2}$ for $f\in L^2((0,1))$ and $A\colon L^2((0,1))\to L^2((0,1))$.
			Using \eqref{eq:cov} and the assumption $\alpha_{\ko} = 0$, we obtain by Bochner integration and spectral calculus
			\begin{align*}
				&\Var\Big(\sum_{i=1}^n\alpha_i I_{\delta,i}\Big)\\
				&\,= \frac12\int_0^T\int_0^T\sum_{i,j=1}^n\frac{\alpha_i\alpha_j}{\theta_{\delta,i}(\ko)^2 \theta_{\delta,j}(\ko)^2}\scapro{((-\Delta_\theta)(S_\theta(\abs{t-s})-S_\theta(t+s)))^{\otimes 2} K_{\delta,i}^{\otimes 2}}{K_{\delta,j}^{\otimes 2}}_{L^2([0,1]^2)}\diff t\diff s\\
				&\,= \frac12\bscapro{\int_0^T\int_0^T((-\Delta_\theta)(S_\theta(\abs{t-s})-S_\theta(t+s)))^{\otimes 2}\diff t \diff s \sum_{i=1}^n\tfrac{\alpha_i}{\theta_{\delta,i}(\ko)^2}K_{\delta,i}^{\otimes 2}}{\sum_{i=1}^n\tfrac{\alpha_i}{\theta_{\delta,i}(\ko)^2}K_{\delta,i}^{\otimes 2}}_{L^2([0,1]^2)}\\
				&\,= \frac12\sum_{k,l\ge 1} \int_0^T\int_0^T \lambda_k\lambda_l(\e^{-\lambda_k\abs{t-s}}-\e^{-\lambda_k(t+s)})(\e^{-\lambda_l\abs{t-s}}-\e^{-\lambda_l(t+s)})\diff t \diff s\\
				&\qquad\qquad \times \Big(\sum_{i=1}^n \tfrac{\alpha_i}{\theta_{\delta,i}(\ko)^2}\scapro{K_{\delta,i}}{e_k}
				\scapro{K_{\delta,i}}{e_l}\Big)^2\\
				&\,\le \frac12\sum_{k,l\ge 1} \int_0^T\int_0^T \lambda_k\lambda_l\e^{-\lambda_k\abs{t-s}}(\e^{-\lambda_l\abs{t-s}}-\e^{-\lambda_l(t+s)})\diff t \diff s \Big(\sum_{i=1}^n\tfrac{\alpha_i}{\theta_{\delta,i}(\ko)^2}\scapro{K_{\delta,i}}{e_k}
				\scapro{K_{\delta,i}}{e_l}\Big)^2\\
				&\,\le \frac12\sum_{k,l\ge 1} \int_0^T\int_0^T \lambda_k\lambda_l\e^{-(\lambda_k+\lambda_l)\abs{t-s}}\diff t \diff s \Big(\sum_{i=1}^n\tfrac{\alpha_i}{\theta_{\delta,i}(\ko)^2}\scapro{K_{\delta,i}}{e_k}
				\scapro{K_{\delta,i}}{e_l}\Big)^2\\
				&\,\le  T \sum_{k,l\ge 1}\frac{\lambda_k\lambda_l}{\lambda_k+\lambda_l}\Big(\sum_{i=1}^n\tfrac{\alpha_i}{\theta_{\delta,i}(\ko)^2}\scapro{K_{\delta,i}}{e_k}
				\scapro{K_{\delta,i}}{e_l}\Big)^2\\
				&\,\le  \frac T2 \sum_{k,l\ge 1}\lambda_k^{1/2}\lambda_l^{1/2}\Big(\sum_{i=1}^n\tfrac{\alpha_i}{\theta_{\delta,i}(\ko)^2}\scapro{K_{\delta,i}}{e_k}
				\scapro{K_{\delta,i}}{e_l}\Big)^2\\
				&\,\le  \frac T2 \Big( \sum_{k,l\ge 1}\lambda_k\lambda_l\Big(\sum_{i=1}^n\tfrac{\alpha_i}{\theta_{\delta,i}(\ko)^2}\scapro{K_{\delta,i}}{e_k}
				\scapro{K_{\delta,i}}{e_l}\Big)^2\Big)^{\frac12} \Big( \sum_{k,l\ge 1}\Big(\sum_{i=1}^n\tfrac{\alpha_i}{\theta_{\delta,i}(\ko)^2}\scapro{K_{\delta,i}}{e_k}
				\scapro{K_{\delta,i}}{e_l}\Big)^2\Big)^{\frac12}\\
				&\,= \frac T2 \Big(\sum_{i,j=1}^n\tfrac{\alpha_i\alpha_j}{\theta_{\delta,i}(\ko)^2 \theta_{\delta,j}(\ko)^2}\scapro{(-\Delta_\theta)^{\frac12}K_{\delta,i}}{(-\Delta_\theta)^{\frac12}K_{\delta,j}}^2
				\Big)^{\frac12} \\
				&\qquad \times\Big(\sum_{i,j=1}^n\frac{\alpha_i\alpha_j}{\theta_{\delta,i}(\ko)^2 \theta_{\delta,j}(\ko)^2}\scapro{K_{\delta,i}}{K_{\delta,j}}^2\Big)^{\frac12}\\
				&\,=  \frac T2 \delta^{-2} \Big(\sum_{i=1}^n\tfrac{\alpha_i^2}{\theta_{\delta,i}(\ko)^4}\Big(\int\theta(x_i+\delta y)K'(y)^2\diff{y}\Big)^2\Big)^{\frac12}\Big(\sum_{i=1}^n\tfrac{\alpha_i^2}{\theta_{\delta,i}(\ko)^4}\Big)^{\frac12}\\
				&\,\leq \frac{T}{2\lowell^3} \delta^{-2} \lVert \alpha \rVert_{\ell^2}^2 \lVert K^\prime \rVert^2_{L^2}.
			\end{align*}
			\item  We have
			\begin{align*}
				c_{\ko,\ko}(t,s) &= \int_0^{t \wedge s} \langle S_\theta(t+s-2u) \Delta K_{\delta,\ko}, \Delta K_{\delta,\ko} \rangle \diff{u} \\
				&= \frac{1}{2}\big\langle (S_\theta(\lvert t -s \rvert) - S_\theta(t+s)) (-\Delta_\theta)^{-1} \Delta K_{\delta,\ko}, \Delta K_{\delta,\ko} \big\rangle\\
				&\in \frac{1}{2}\Big[0 , \big\langle S_\theta(\lvert t -s \rvert)  (-\Delta_\theta)^{-1} \Delta K_{\delta,\ko}, \Delta K_{\delta,\ko} \big\rangle \Big].
			\end{align*}
			Thus,
			\begin{align*}
				\mathrm{Var}(I_{\delta,\ko}) &= \int_0^T \int_0^T 2c_{\ko,\ko}(t,s)^2 \diff{s} \diff{t}\\
				&\leq \frac{1}{2}\int_0^T \int_0^T \big\langle S_\theta(\lvert t -s \rvert)  (-\Delta_\theta)^{-1} \Delta K_{\delta,\ko}, \Delta K_{\delta,\ko} \big\rangle^2 \diff{s} \diff{t} \\
				&= \frac{1}{2}\sum_{k,l \geq 1} \frac{1}{\lambda_k \lambda_l} \langle \Delta K_{\delta,\ko}, e_k \rangle^2 \langle \Delta K_{\delta,\ko}, e_l \rangle^2  \int_0^T \int_0^T \mathrm{e}^{-(\lambda_k + \lambda_l)\lvert t -s \rvert} \diff{s}\diff{t} \\
				&\leq \frac{T}{2} \sum_{k,l \geq 1} \frac{1}{\lambda_k\lambda_l(\lambda_k + \lambda_l)} \langle \Delta K_{\delta,\ko}, e_k \rangle^2 \langle \Delta K_{\delta,\ko}, e_l \rangle^2 \\
				&\leq \frac{T}{4} \Big(\sum_{k \in \N} \lambda_k^{-3/2} \langle \Delta K_{\delta,\ko}, e_k\rangle^2\Big)^2 = \frac{T}{4} \big\lVert (-\Delta_\theta)^{-3/4} \Delta K_{\delta,\ko} \big\rVert^4.
			\end{align*}
			By Lemma \ref{lem:sol_interpol}, $\lVert (-\Delta_\theta)^{-3/4} \Delta K_{\delta,\ko} \rVert \lesssim \delta^{-1/2}$ holds, implying the result.
		\end{enumerate}
	\end{proof}
	
	We now turn to the proof of the concentration inequality for linear combinations of $(I_{\delta,i})_{i=1,\ldots,\delta^{-1}}$ stated in Proposition \ref{prop:conc}.
	It relies on the following result given in \cite{nourdin09}.
	For details on Malliavin calculus, we refer to the standard references \cite{nourdin12,nualart06}.
	
	\begin{theorem}[Theorem 4.1 in \cite{nourdin09}]	\label{theo:nourdin}
		Let $X = (X(h))_{h \in \mathfrak{H}}$ be an isonormal Gaussian process on a real separable Hilbert space $\mathfrak{H}$, and let  $\mathbb{D}^{1,2}$ be the domain of the Malliavin derivative operator $D$ associated to $X$.
		Let $Z \in \mathbb{D}^{1,2}$ have zero mean, and define
		\[g_Z(z) \coloneqq \E\big[\scapro{DZ}{-DL^{-1}Z}_{\mathfrak{H}} \mid Z = z \big], \quad z \in \R,\]
		where $L^{-1}$ is the pseudoinverse of the Ornstein--Uhlenbeck generator $L \coloneqq \sum_{m=0}^\infty -mJ_m$, $J_m$ being the projection onto the $m$-th Wiener chaos.
		Assume that, for some $\alpha \geq 0, \beta > 0$,
		\begin{enumerate}[label = (\roman*), ref = (\roman*)]
			\item $g_Z(Z) \leq \alpha Z + \beta$, $\PP$-a.s.;
			\item the law of $Z$ has a density $\rho$.
		\end{enumerate}
		Then, for all $z > 0$, we have
		\[\PP(Z \geq z) \leq \exp\Big(-\frac{z^2}{2\alpha z + 2\beta} \Big) \quad \text{and} \quad \PP(Z \leq -z) \leq \exp\Big(-\frac{z^2}{2\beta} \Big).\]
	\end{theorem}
	
	In order to apply this result, we must first construct an appropriate Hilbert space for our specific setting. Let $\mathcal{E}$ be the set of $\R$-valued stepfunctions on $[0,T]^n\setminus \{0\}^n$ that can be expressed as linear combinations of indicator functions $\mathbf{1}_{[0,t_1]\times \cdots \times [0,t_n]}$, $t_i \in [0,T]$ and $t_j \neq 0$ for some $j \in \{1,\ldots,n\}$, and let $\mathfrak{H}$ be the separable Hilbert space obtained by closing $\mathcal{E}$ with respect to the inner product determined by
	\[\scapro{f}{g}_{\mathfrak{H}} = \sum_{k=1}^K \sum_{l=1}^L a_kb_l \sum_{i,j =1}^n\E[X^{\Delta}_{\delta,i}(t_{k,i})X^{\Delta}_{\delta,j}(s_{l,j})] = \sum_{k=1}^K \sum_{l=1}^L a_kb_lc_{i,j}(t_{k,i},s_{l,j})\]
	for $f = \sum_{k=1}^K a_k \mathbf{1}_{[0,t_{k,1}] \times \cdots \times [0,t_{k,n}]} \in \mathcal{E}$ and $g = \sum_{l=1}^Lb_l \mathbf{1}_{[0,s_{l,1}] \times \cdots \times [0,s_{l,n}]} \in \mathcal{E}$.
	For
	\[
	h=\sum_{k=1}^K c_k \mathbf{1}_{[0,t_{k,1}]\times \cdots \times [0,t_{k,n}]} \in \mathcal{E},
	\]
	we set $X(h) = \sum_{k=1}^K c_k \sum_{i=1}^n X^{\Delta}_{\delta,i}(t_{k,i})$ and, for any $h \in \mathfrak{H}$, let $X(h)$ be the $L^2$ limit of $(X(h_n))_{n \in \N}$ for some sequence $(h_n)_{n \in \N} \subset \mathcal{E}$ converging to $h$ in $(\mathfrak{H}, \langle\cdot,\cdot \rangle_{\mathfrak{H}})$, then $(X(h))_{h \in \mathfrak{H}}$ is an isonormal  Gaussian process (cf.\ \cite[Proposition 2.1]{altmeyer21}) over $\mathfrak{H}$.
	In particular, since $X_0 \equiv 0$ by assumption, this implies that, for any $t>0$, $X^{\Delta}_{\delta,i}(t) = X(\mathbf{1}^{(i)}_{[0,t]})$, where $\mathbf{1}^{(i)}_{[0,t]} = \mathbf{1}_{\prod_{j=1}^n A_j}$ for $A_i = [0,t]$ and $A_j = \{0\}$ for $i \neq j$.

	\begin{proof}[Proof of Proposition \ref{prop:conc}]
		The Malliavin derivative of $\varphi(X(\mathbf{1}^{(i)}_{[0,t]})) = X^{\Delta}_{\delta,i}(t)^2$ is  $D\varphi(X(\mathbf{1}^{(i)}_{[0,t]})) = 2X^{\Delta}_{\delta,i}(t) \mathbf{1}^{(i)}_{[0,t]}$. Introduce
		\[Z = \sum_{i=1}^n \alpha_i(I_{\delta,i} - \E[I_{\delta,i}]) = \sum_{i=1}^n \alpha_i \int_0^T (X^{\Delta}_{\delta,i}(t)^2 - \E[X^{\Delta}_{\delta,i}(t)^2]) \diff{t},\]
		which is an element of the second Wiener chaos associated with $(X(h))_{h \in \mathfrak{H}}$ since for the second Hermite polynomial $H_2(x) = x^2 -1$ and $h^N_{k,i} = \mathbf{1}^{(i)}_{[0,kT/N]}$ we can write $Z$ as the $L^2$ limit
		\[Z = \lim_{N \to \infty} \sum_{k=1}^N \sum_{i=1}^n \alpha_i \frac{T\lVert h^N_{k,i} \rVert^2_{\mathfrak{H}}}{N} H_2(X(h^N_{k,i}/\Vert h^N_{k,i} \rVert_{\mathfrak{H}})).\]
		This implies that the law of $Z$ has a Lebesgue density, $Z \in \mathbb{D}^{1,2}$ and its Malliavin derivative is given by
		\[DZ = 2\sum_{i=1}^n \alpha_i \int_0^T X^{\Delta}_{\delta,i}(t) \mathbf{1}^{(i)}_{[0,t]} \diff{t}.\]
		Then, $L^{-1}Z = -\tfrac{1}{2} Z$ and, therefore, for the orthonormal eigensystem $(e_k,\lambda_k)$ of $-\Delta_\theta$, it follows for $(\alpha_i)_{i=1}^n \in \R_+^n \setminus \{0\}$ with $\alpha_{\ko} = 0$ by similar reasoning as in the proof of Lemma \ref{lem:basic_est},
		\begin{align*}
			&\scapro{DZ}{-DL^{-1}Z}_{\mathfrak{H}}\\
			&\quad= \tfrac{1}{2} \lVert DZ \rVert^2_{\mathfrak{H}}\\
			&\quad= 2 \sum_{i,j = 1}^n \alpha_i \alpha_j\int_0^T \int_0^T  X^{\Delta}_{\delta,i}(t)X^{\Delta}_{\delta,j}(s) \scapro{\mathbf{1}^{(i)}_{[0,t]}}{\mathbf{1}^{(j)}_{[0,s]}}_{\mathfrak{H}} \diff{t}\diff{s}\\
			&\quad= 2\sum_{i,j=1}^n \alpha_i \alpha_j \int_0^T \int_0^T  X^{\Delta}_{\delta,i}(t)X^{\Delta}_{\delta,j}(s) c_{i,j}(t,s) \diff{t}\diff{s}\\
			&\quad= \sum_{k \in \N} \int_0^T \int_0^T \lambda_k (\mathrm{e}^{-\lambda_k \lvert t-s \rvert} - \mathrm{e}^{-\lambda_k (t+s)}) \sum_{i = 1}^n \langle \tfrac{\alpha_i}{\theta_{\delta,i}(\ko)} X^{\Delta}_{\delta,i}(t)K_{\delta,i}, e_k \rangle \sum_{j = 1}^n \langle \tfrac{\alpha_j}{\theta_{\delta,j}(\ko)} X^{\Delta}_{\delta,j}(s)K_{\delta,j}, e_k \rangle \diff{t}\diff{s} \\
			&\quad\leq \sum_{k \in \N} \int_0^T \int_0^T \lambda_k (\mathrm{e}^{-\lambda_k \lvert t-s \rvert} - \mathrm{e}^{-\lambda_k (t+s)}) \Big(\sum_{i = 1}^n \langle \tfrac{\alpha_i}{\theta_{\delta,i}(\ko)} X^{\Delta}_{\delta,i}(t)K_{\delta,i}, e_k \rangle \Big)^2 \diff{t} \diff{s}\\
			&\quad\leq \sum_{k \in \N} \int_0^T \Big(\sum_{i = 1}^n \langle \tfrac{\alpha_i}{\theta_{\delta,i}(\ko)} X^{\Delta}_{\delta,i}(t)K_{\delta,i}, e_k \rangle \Big)^2 \int_0^T \lambda_k \mathrm{e}^{-\lambda_k\lvert t - s \rvert} \diff{s} \diff{t}\\
			&\quad\leq 2\sum_{k \in \N} \int_0^T \Big(\sum_{i = 1}^n \langle \tfrac{\alpha_i}{\theta_{\delta,i}(\ko)} X^{\Delta}_{\delta,i}(t)K_{\delta,i}, e_k \rangle \Big)^2 \diff{t}\\
			&\quad= 2\sum_{i,j=1}^n \int_0^T \tfrac{\alpha_i \alpha_j}{\theta_{\delta,i}(\ko)\theta_{\delta,j}(\ko)} X_{\delta,i}^\Delta(t) X_{\delta,j}^\Delta(t) \langle K_{\delta,i}, K_{\delta,j} \rangle \diff{t}\\
			&\quad = 2\sum_{i=1}^n \tfrac{\alpha_i^2}{\theta_{\delta,i}(\ko)^2} \int_0^T X^\Delta_{\delta,i}(t)^2 \diff{t}\\
			&\quad\leq 2 \frac{\lVert \alpha \rVert_{\infty}}{\lowell^2}\Big(Z + \sum_{i=1}^n \alpha_i \E[I_{\delta,i}]\Big),
		\end{align*}
		whence,
		\begin{align*}
			\PP(\lvert Z \rvert \geq z) &= \PP\Big(\Big\lvert \sum_{i=1}^n \alpha_i(I_{\delta,i} - \E[I_{\delta,i}]) \Big\rvert \geq z\Big)\\
			&\leq 2\exp\bigg(-\frac{\lowell^2}{4\lVert \alpha \rVert_{\infty}} \frac{z^2}{z + \sum_{i=1}^n \alpha_i \E[I_{\delta,i}]} \bigg),
		\end{align*}
		follows from Theorem \ref{theo:nourdin}. In particular, since by Lemma \ref{lem:basic_est} we have
		\[\sum_{i=1}^n \alpha_i \E[I_{\delta,i}] \leq \lVert \alpha \rVert_{\ell^1} \frac{T}{2\lowell} \lVert K^\prime \rVert^2_{L^2} \delta^{-2},\]
		we obtain
		\begin{align*}
			\PP\Big(\Big\lvert \sum_{i=1}^n \alpha_i(I_{\delta,i} - \E[I_{\delta,i}]) \Big\rvert \geq z\Big) \leq 2\exp\bigg(-\frac{\lowell^2}{2\lVert \alpha \rVert_{\infty}} \frac{z^2}{2 z + \lVert \alpha \rVert_{\ell^1} T\lowell^{-1}\lVert K^\prime \rVert^2_{L^2}\delta^{-2}} \bigg).
		\end{align*}
	\end{proof}
	
	\begin{proof}[Proof of Proposition \ref{prop:coupling}]
		For $i \neq j$, $B_{\delta,i}$ and $B_{\delta,j}$ are independent Brownian motions with respect to the same filtration.
		Hence, the quadratic covariations satisfy $\langle M_{\delta,i}, M_{\delta,j} \rangle \equiv 0$ a.s., and by Knight's multivariate extension of the  Dambis--Dubins--Schwarz construction \cite[Theorem V.1.9]{revyor99}, $(\bar M_{\delta,i})_i$ has the law of a vector of independent Brownian motions at times $\bar I_{\delta,i}$ in coordinate $i$, that is, $\bar M_{\delta,i}\sim N(0,\bar I_{\delta,i})$ and $(\bar M_{\delta,i})_{i=1,\ldots, \delta^{-1}}$ are independent.
		Define  the continuous martingale $(\mathcal{M}_{ \delta,i}(t))_{t\geq 0}$ by setting
		\begin{align*}
			\mathcal{M}_{\delta,i}(t) \coloneqq \int_0^t \big(\one_{\{s\le T\}}- \one_{\{s\le\sigma_i\}}\big)X_{\delta,i}^\Delta(s) \diff{B_{\delta,i}(s)}, \quad t \geq 0.
		\end{align*}
		Fix $\alpha\in\R^n$.
		Then, the process $\bar{\mathcal{M}}_t \coloneqq \sum_{i = 1}^n \alpha_i \mathcal{M}_{\delta,i}(t)$, $t \geq 0$, is a continuous martingale with
		\begin{align*}
			\sum_{i = 1}^n \alpha_i (M_{\delta,i}-\bar M_{\delta,i}) &= \sum_{i = 1}^n \alpha_i \int_0^\infty \big(\one_{\{t\le T\}}- \one_{\{t\le\sigma_i\}}\big) X_{\delta,i}^\Delta(t) \diff{B_{\delta,i}(t)} = \bar{\mathcal{M}}_\infty,\\
			\sum_{i = 1}^n \alpha_i^2 \abs{I_{\delta,i}-\bar I_{\delta,i}} &= \sum_{i = 1}^n \alpha_i^2 \int_0^\infty \big(\one_{\{t\le T\}}- \one_{\{t\le\sigma_i\}}\big)^2 X_{\delta,i}^\Delta(t)^2 \diff{t} = \langle \bar{\mathcal{M}} \rangle_\infty,
		\end{align*}
		where the representation of the second sum as the limit of the quadratic variations of $\overbar{\mathcal{M}}_t$ follows from the independence of  the Brownian motions $B_{\delta,i}$.
		Thus, for any $z, L > 0$, we infer by the continuous martingale Bernstein inequality (see, e.g., \cite[Exercise 3.16]{revyor99})
		\[\PP\Big( \sum_{i=1}^n \alpha_i (M_{\delta,i}-\bar M_{\delta,i})\ge z,\, \sum_{i=1}^n \alpha_i^2\abs{I_{\delta,i}-\bar I_{\delta,i}}\leq L\Big) = \PP(\bar{\mathcal{M}}_\infty \geq z, \langle \bar{\mathcal{M}} \rangle_\infty \le L)\le \e^{-\frac{z^2}{2L}}.\]
		Repeating the argument for the continuous martingale $-\overbar{\mathcal{M}}$ then yields the second claim.
	\end{proof}
	
	Next, we provide the proof for the uniform orders of the empirical processes $I_{T,\delta}(\cdot) - \E[I_{T,\delta}(\cdot)]$ and $M_{T,\delta(\cdot)}$.
	\begin{proof}[Proof of Lemma \ref{lem:aux0}]
		We start with verifying \eqref{eq:uni_qv}.
		Note that
		\begin{equation}\label{eq:err5}
			\begin{split}
				&\big\lvert I_{T,\delta}(\theta_-,\theta_+,\theta_\circ,h) - \E[I_{T,\delta}(\theta_-,\theta_+,\theta_\circ,h)] \big\rvert\\
				&\,= \frac{1}{2}\Big\vert\Big(\sum_{i =1, i \neq \koo}^{\lceil h / \delta \rceil \wedge \koo} + \one_{\{\lceil h/\delta \rceil  \neq \koo\}}\sum_{i = \lceil h/\delta \rceil \wedge \koo + 1, i \neq \koo}^{\lceil h / \delta \rceil \vee  \koo} + \sum_{i= \lceil h/\delta \rceil \vee \koo+1, i \neq \koo}^n \Big) (\theta_{\delta,i}(\lceil h/\delta\rceil) - \theta^0_{\delta,i})^2(I_{\delta,i} - \E[I_{\delta,i}]) \Big\rvert \\
				&\, \lesssim (\theta_- -\theta^0_-(\delta))^2 \Big \lvert  \sum_{i =1, i \neq \koo}^{\lceil h / \delta \rceil \wedge \koo -1} (I_{\delta,i}- \E[I_{\delta,i}]) \Big\rvert\  \\
				&\,\quad+  \max\big\{(\theta_\circ - \theta^0_\pm(\delta))^2\big\} \lvert I_{\delta,\lceil h/\delta \rceil} - \E[I_{\delta,\lceil h/\delta \rceil}] \rvert\one_{\{\lceil h/\delta \rceil \neq \koo\}} \\
				&\,\quad+ \max\big\{(\theta_\pm -\theta^0_\mp(\delta))^2\big\} \Big \lvert  \sum_{i=\lceil h / \delta \rceil \wedge \koo + 1,i \notin \{\koo,\lceil h/\delta\rceil\}}^{\lceil h / \delta \rceil \vee \koo} (I_{\delta,i}- \E[I_{\delta,i}])\Big\rvert \one_{\{\lceil h/\delta \rceil \neq \koo\}}\\
				&\,\quad + (\theta_+ - \theta^0_+(\delta))^2 \Big \lvert  \sum_{i = \lceil h / \delta \rceil \vee \koo + 1 , i \neq \koo}^{n} (I_{\delta,i}- \E[I_{\delta,i}]) \Big\rvert.
			\end{split}
		\end{equation}
		It follows
		\begin{align*}
			&\sup_{\chi \in \Theta \times (0,1]} \big\vert I_{T,\delta}(\chi) - \E[I_{T,\delta}(\chi)] \big\vert \\
			&\,\lesssim_{\lowell,\upell}\max_{k=1, \ldots,\delta^{-1}} \Big \lvert \sum_{i \in [k] \setminus \{\koo\}} (I_{\delta,i} - \E[I_{\delta,i}]) \Big\rvert + \max_{k=1, \ldots,\delta^{-1}-1} \Big \lvert \sum_{i = k+1, i \neq \koo}^n (I_{\delta,i} - \E[I_{\delta,i}]) \Big\rvert\\
			&\,\quad + \max_{k_1,k_2 \in \{1, \ldots,\delta^{-1} - 1\}, k_1 \leq k_2} \Big \lvert \sum_{i = k_1+1, i\neq \koo}^{k_2} (I_{\delta,i} - \E[I_{\delta,i}]) \Big\rvert  + \max_{k \in [\delta]^{-1}} \lvert I_{\delta,k} - \E[I_{\delta,k}] \rvert.
		\end{align*}
		Hence, for some constant $C > 0$ depending only on $\lowell$ and $\upell$, first using a union bound and then Proposition \ref{prop:conc}, we obtain, for any $z > 0$,
		\begin{align*}
			&\PP\Big(\sup_{\chi \in \Theta \times (0,1]} \big\vert I_{T,\delta}(\chi) - \E[I_{T,\delta}(\chi)] \big\vert \geq z\Big)\\
			&\,\leq \sum_{k=1}^{n} \PP\Big(\Big \lvert \sum_{i \in [k] \setminus \{\koo\}} (I_{\delta,i} - \E[I_{\delta,i}]) \Big\rvert \geq z/(4C) \Big) + \sum_{k=1}^{n-1} \PP\Big(\Big \lvert \sum_{i = k+1, i \neq \koo }^n (I_{\delta,i} - \E[I_{\delta,i}]) \Big\rvert \geq z/(4C)\Big) \\
			&\,\quad + \sum_{k_1,k_2 \in \{1, \ldots,\delta^{-1} - 1\}, k_1 \leq k_2} \PP\Big(\Big \lvert \sum_{i = k_1+1, i \neq \koo}^{k_2} (I_{\delta,i} - \E[I_{\delta,i}]) \Big\rvert \geq z/(4C) \Big)\\
			&\,\quad + \sum_{k \in [n] \setminus \{\koo\}} \PP\big(\lvert I_{\delta,k}- \E[I_{\delta,k}] \rvert \geq z/(4C)\big)\\
			&\, \leq (3\delta^{-1} + \delta^{-2}/4) \exp\Big(-\frac{\lowell^2}{32} \frac{z^2}{Cz/4 +  C^2T\lowell^{-1}\lVert K^\prime \rVert^2_{L^2}\delta^{-3}} \Big),
		\end{align*}
		which establishes \eqref{eq:uni_qv}.
		It remains to verify \eqref{eq:uni_mart}.
		Let
		\[\overbar{M}_{T,\delta}^{(1)}(\chi) \coloneqq \sum_{i \in [n] \setminus \{\koo\}} \theta_{\delta,i}(\lceil h/\delta \rceil) \overbar{M}_{\delta,i}, \quad \overbar{M}_{T,\delta}^{(2)}(\chi) \coloneqq \sum_{i \in [n] \setminus \{\koo\}} \theta_{\delta,i}^0 \overbar{M}_{\delta,i}.\]
		Then, $\overbar{M}_{T,\delta}^{(1)}(\theta_-,\theta_+,h)$ can be written as
		\[ \theta_- \sum_{i \in [\lceil h/\delta \rceil -1] \setminus \{\koo\}} \overbar{M}_{\delta,i} + \theta_+ \sum_{i \ge \lceil h/\delta \rceil + 1, i \neq \koo} \overbar{M}_{\delta,i} + \theta_\circ \overbar{M}_{\delta,\lceil h/\delta \rceil} \one_{\{\lceil h/\delta \rceil \ne \koo\}},
		\]
		implying that $\sup_{\chi \in \Theta \times (0,1]} \big\vert \overbar{M}_{T,\delta}^{(1)}(\chi) \big\vert$ is upper bounded by
		\begin{align*}
		&\upell \Big(\max_{k =1, \ldots, \delta^{-1}} \Big\vert \sum_{i \in [k]\setminus\{\koo\}} \overbar{M}_{\delta,i}\Big\vert + \max_{k =1, \ldots, \delta^{-1}} \Big\vert \sum_{i \in [k]\setminus\{n-\koo+1\}} \overbar{M}_{\delta,n - (i-1)}\Big\vert + \max_{k \in [\delta^{-1}] \setminus \{\koo\}} \lvert \overbar{M}_{\delta,k} \rvert\Big)\\
			&\quad\eqqcolon \upell \Big(\max_{k=1,\ldots,\delta^{-1}} \lvert Y_k \rvert + \max_{k=1,\ldots,\delta^{-1}} \lvert \tilde{Y}_k \rvert + \max_{k \in [\delta^{-1}] \setminus \{\koo\}} \lvert \overbar{M}_{\delta,k} \rvert\Big),
		\end{align*}
		where $(Y_k)$ and $(\tilde{Y}_k)$ are martingales in $k$, due to the independence of the zero mean summands provided by Proposition \ref{prop:coupling}. Since $\overbar{M}_{\delta,k} \sim N(0,\overbar{I}_{\delta,k})$ with $\overbar{I}_{\delta,k} = \E[I_{\delta,k}] \lesssim \delta^{-2}$ for $k \neq \koo$, a union bound immediately yields
		\[\max_{k \in [\delta^{-1}] \setminus\{\koo\}} \lvert \overbar{M}_{\delta,k} \rvert = \mathcal{O}_{\PP}(\delta^{-1}\sqrt{\log(\delta^{-1})}).\]
		Using moreover that Lemma \ref{lem:basic_est} implies
		\[\E[Y_n^2] = \E[\tilde{Y}_n^2] = \sum_{i \in [n]\setminus \koo} \E[I_{\delta,i}] \lesssim \delta^{-3},\]
		we obtain from Doob's (sub)martingale inequality that
		\begin{equation*}\label{eq:eq:err1}
			\sup_{\chi \in \Theta \times (0,1]} \big\vert \overbar{M}_{T,\delta}^{(1)}(\chi) \big\vert = \mathcal{O}_{\PP}(\delta^{-3/2}).
		\end{equation*}
		The same arguments give
		\begin{equation*}\label{eq:err2}
			\sup_{\chi \in \Theta \times (0,1]} \big\vert \overbar{M}_{T,\delta}^{(2)}(\chi) \big\vert = \mathcal{O}_{\PP}(\delta^{-3/2}),
		\end{equation*}
		such that \eqref{eq:uni_mart} will follow once we show that
		\begin{equation} \label{eq:uni_mart2}
			\sup_{\chi \in \Theta \times (0,1]} \big\vert M_{T,\delta}(\chi) - (\overbar{M}_{T,\delta}^{(1)}(\chi) - \overbar{M}_{T,\delta}^{(2)}(\chi)) \big\vert = \mathcal{O}_{\PP}(\delta^{-3/2}).
		\end{equation}
		Similarly to the calculations above, we can write
		\begin{align*}
			&\sup_{\chi \in \Theta \times (0,1]} \big\lvert M_{T,\delta}(\chi) - (\overbar{M}_{T,\delta}^{(1)}(\chi) - \overbar{M}_{T,\delta}^{(2)}(\chi)) \big\rvert \\
			&\, \lesssim \upell \Big(\max_{k =1, \ldots, \delta^{-1}} \Big\vert \sum_{i \in [k]\setminus\{\koo\}} (M_{\delta,i} - \overbar{M}_{\delta,i})\Big\vert \\
			&\qquad+ \max_{k =1, \ldots, \delta^{-1}} \Big\vert \sum_{i \in [k]\setminus\{n-\koo+1\}} (M_{\delta,n-(i-1)}-\overbar{M}_{\delta,n - (i-1)})\Big\vert + \max_{k \in [\delta^{-1}]\setminus\{\koo\}} \lvert M_{\delta,k} - \overbar{M}_{\delta,k} \rvert \Big).
		\end{align*}
		For any $L,z > 0$, we have
		\begin{align*}
			&\PP\Big(\max_{k =1, \ldots, \delta^{-1}} \Big\vert \sum_{i \in [k]\setminus\{\koo\}} (M_{\delta,i} - \overbar{M}_{\delta,i})\Big\vert \geq z \Big)\\
			&\,\leq \PP\Big(\max_{k =1, \ldots, \delta^{-1}} \Big\vert \sum_{i \in [k]\setminus\{\koo\}} (M_{\delta,i} - \overbar{M}_{\delta,i})\Big\vert \geq z , \sum_{i \in [n] \setminus \{\koo\}} \lvert I_{\delta,i} - \E[I_{\delta,i}] \rvert \leq L\Big)\\
			&\,\quad+ \PP\Big(\sum_{i \in [n] \setminus \{\koo\}} \lvert I_{\delta,i} - \E[I_{\delta,i}] \rvert > L \Big)\\
			&\,\leq \sum_{k=1}^{\delta^{-1}} \PP\Big( \Big\vert \sum_{i \in [k]\setminus\{\koo\}} (M_{\delta,i} - \overbar{M}_{\delta,i})\Big\vert \geq z , \sum_{i \in [k] \setminus \{\koo\}} \lvert I_{\delta,i} - \E[I_{\delta,i}] \rvert \leq L\Big)\\
			&\,\quad+ \PP\Big(\delta^{-1} \sum_{i \in [n] \setminus \{\koo\}} (I_{\delta,i} - \E[I_{\delta,i}])^2 > L^2 \Big)\\
			&\,\leq 2 \delta^{-1} \mathrm{e}^{-\frac{z^2}{L}} + \frac{\delta^{-4} T \lVert K^\prime \rVert^2_{L^2}}{2\lowell^3 L^2},
		\end{align*}
		where in the second inequality the Cauchy--Schwarz inequality was used for the sum in the second probability, while the final inequality is a consequence of Proposition \ref{prop:coupling} and Lemma \ref{lem:basic_est}, combined with Markov's inequality.
		Thus, for the choice $L = R_L \delta^{-2}$ with $R_L \to \infty$ and $z_L = R_L \delta^{-1} \sqrt{\log(\delta^{-1})}$, we obtain
		\[\PP\Big(\max_{k =1, \ldots, \delta^{-1}} \Big\vert \sum_{i \in [k]\setminus\{\koo\}} (M_{\delta,i} - \overbar{M}_{\delta,i})\Big\vert \geq z_L \Big) \longrightarrow 0, \quad \text{ as } R_L \to \infty,\]
		whence
		\begin{equation*}\label{eq:err3}
			\max_{k =1, \ldots, \delta^{-1}} \Big\vert \sum_{i \in [k]\setminus\{\koo\}} (M_{\delta,i} - \overbar{M}_{\delta,i})\Big\vert = \mathcal{O}_{\PP}\big(\delta^{-1} \sqrt{\log(\delta^{-1})}\big) = o_{\PP}(\delta^{-3/2}).
		\end{equation*}
		The same arguments also yield
		\begin{align*}
			\max_{k =1, \ldots, \delta^{-1}} \Big\vert \sum_{i \in [k]\setminus\{n-\koo+1\}} (M_{\delta,n- (i-1)} - \overbar{M}_{\delta,n-(i-1)})\Big\vert &=  o_{\PP}(\delta^{-3/2}),\\
			\max_{k \in [\delta^{-1}]\setminus\{\koo\}} \lvert M_{\delta,k} - \overbar{M}_{\delta,k} \rvert &=  o_{\PP}(\delta^{-3/2}),
		\end{align*}
		which establishes \eqref{eq:uni_mart2} and therefore proves \eqref{eq:uni_mart}.
	\end{proof}
	
	\section{Remaining proofs for Section \ref{sec:est}}\label{app:proof_rem}
	We start by proving the upper bound for the error term $R_{\delta,\koo}$ arising in the representation of our simultaneous M-estimator $(\hat\theta_-^\delta,\hat\theta_+^\delta, \hat{\theta}^\delta_\circ,\hat{\tau}^\delta)$.
	
	\begin{proof}[Proof of Proposition \ref{lem:remainder}]
		Let $\theta = \theta^0(\delta)$, $\theta_- = \theta_-^0(\delta)$ and $\tau = \tau^0$. Let us first treat the case \ref{err1} concerning the $L^1$-bound of $R_{\delta,\koo}$, then \ref{err2} the variance bound on $R_{\delta,\koo}$ and finally prove \ref{err3} the existence of $\theta^0_\circ \in [\lowell,\upell]$ such that $\lvert \E[R_{\delta,\koo}(\theta^0_\circ)] \rvert \lesssim \delta^{-1}$.
		\begin{enumerate}[label = (\roman*), ref = (\roman*),wide, labelwidth=!, labelindent=0pt]
			\item \label{err1}
			By the Cauchy--Schwarz inequality, $\lvert R_{\delta,\koo}(\theta^\prime) \rvert $ is upper bounded by
			\[\Big(\int_0^T X^\Delta_{\delta,\koo}(t)^2 \diff{t} \Big)^{1/2} \Big(\int_0^T \Big(\int_0^t \langle \Delta_{\vartheta} S_{\vartheta}(t-s)K_{\delta,\koo} - \theta^\prime S_\theta(t-s) \Delta K_{\delta,\koo},\diff{W_s} \rangle\Big)^2  \diff{t} \Big)^{1/2}.\]
			Hence, the Cauchy--Schwarz inequality and Fubini's theorem show
			\begin{equation}
				\begin{split}\label{eq:est1}
					\E[ \lvert R_{\delta,\koo}(\theta^\prime) \rvert] &\leq \sqrt{\E[I_{\delta,\koo}]} \sqrt{\int_0^T \int_0^t \lVert \Delta_{\vartheta} S_{\vartheta}(t-s)K_{\delta,\koo} - \theta^\prime S_\theta(t-s) \Delta K_{\delta,\koo} \rVert^2 \diff{s} \diff{t}}\\ &\eqqcolon \sqrt{\mathcal{I}_1} \sqrt{\mathcal{I}_2(\theta^\prime)}.
				\end{split}
			\end{equation}
			The crux of the problem is now that in general $K_{\delta,\koo}$ does not belong to the domain $D(\Delta_\theta)$, which prevents us from swapping the order of application of $\Delta_\theta$ and $S_\theta(t-s)$. To resolve this problem, we construct an appropriate approximation of $K_{\delta,\koo}$ within the domain $D(\Delta_\theta)$. Let $\varphi \in C^\infty(\R)$ such that $\mathrm{supp}(\varphi) \subset [-1/2,1/2]$ and $\varphi^\prime(0) = 1$.
			For $\varepsilon > 0$, let
			\[\varphi_{\varepsilon}(x) \coloneqq \varepsilon \varphi((x-\tau)/\varepsilon), \quad x \in \modelspace.\]
			It is easily seen that $\varphi_\varepsilon \to 0$ in $H^1(\modelspace)$ as $\varepsilon \to 0$.
			Thus, $K^{\varepsilon}_{\delta,\koo} \coloneqq K_{\delta,\koo} - \delta^{-3/2}K^{\prime}((\tau-x_{\koo})/\delta)\varphi_{\varepsilon}$ converges to $K_{\delta,\koo}$ in $H^1(\modelspace)$ as $\varepsilon \to 0$.
			Moreover, for $\varepsilon > 0$ small enough, $\varphi_\varepsilon \in C_c^\infty(\modelspace)$ since $\tau \notin \{0,1\}$, and since $\varphi^\prime(0) = 1$ it follows that $\nabla K^{\varepsilon}_{\delta,\koo}(\tau) = 0$.
			Therefore, for $\varepsilon > 0$ small enough, Lemma \ref{lem:divop} implies that $K^{\varepsilon}_{\delta,\koo} \in D(\Delta_\theta)$ and $\Delta_{\theta} K^{\varepsilon}_{\delta,\koo} = \theta \Delta K^{\varepsilon}_{\delta,\koo}$. In particular, $\Delta_\theta S_\theta(s) K^{\varepsilon}_{\delta,\koo} = S_\theta(s) \Delta_\theta K^\varepsilon_{\delta,\koo}$. Let now $\mathcal{E}$ be the Dirichlet form given by $\mathcal{E}(u,v) = \int_0^1 \nabla u \nabla v \diff{\lambda}$ for $u,v$ belonging to the  domain $D(\mathcal{E}) = \{u \in H^1((0,1)): u(0) = 0\}$. The associated self-adjoint operator on $L^2((0,1))$ is the Laplacian subject to mixed homogeneous Dirichlet--Neumann boundary conditions, i.e., $u(0) = 0$ and $\nabla u(1) = 0$ for any $u \in D(\Delta)$. Moreover, it is well-known that the spectrum of the positive self-adjoint operator $-\Delta$ is discrete and bounded from below by some strictly positive constant, whence $(-\Delta)^{-1}$ exists as a bounded operator from $L^2((0,1))$ to $D(-\Delta)$. Since for any
			\[u \in D((-\Delta_\theta)^{1/2}) = D(\mathcal{E}_\theta) \subset D(\mathcal{E}) = D((-\Delta)^{1/2}),\]
			we have $\mathcal{E}_\theta(u,u) \geq \lowell \mathcal{E}(u,u)$, it now follows from the argument in the proof of Theorem VI.2.21 in \cite{kato95} that $(-\Delta_{\theta})^{-1} \leq \lowell^{-1} (-\Delta)^{-1}$.
			Letting  $\tilde{K}^{\varepsilon}_{\delta,\koo} \coloneqq K^{\varepsilon}_{\delta,\koo} - K^{\varepsilon}_{\delta,\koo}(\tau)$ and $\varphi \in C_c((0,1))$, integration by parts shows
			\begin{align*}
				\langle \one_{[\tau,1)}\tilde{K}^{\varepsilon}_{\delta,\koo}, \varphi^{\prime\prime} \rangle &= \int_\tau^1 \tilde{K}^\varepsilon_{\delta,\koo}(x) \varphi^{\prime\prime}(x) \diff{x}\\
				&= -\int_\tau^1 \nabla K^\varepsilon_{\delta,\koo}(x) \varphi^{\prime}(x) \diff{x}\\
				&= \int_\tau^1 K^\varepsilon_{\delta,\koo}(x) \varphi(x) \diff{x}\\
				&= \langle \one_{[\tau,1)} K^\varepsilon_{\delta,\koo}, \varphi \rangle,
			\end{align*}
			where for the second line we used $\tilde{K}^\varepsilon_{\delta,\koo}(\tau) = 0 = \varphi^\prime(1)$ and the third line follows from $\nabla K^\varepsilon_{\delta,\koo}(\tau) = 0 = \varphi(1)$. Thus, $\one_{[\tau,1)} \tilde{K}^{\varepsilon}_{\delta,\koo}$ solves the Poisson equation $\Delta f = \one_{[\tau,1)} \Delta K^{\varepsilon}_{\delta,\koo}$ subject to the mixed Dirichlet--Neumann boundary conditions (note here that $f \coloneqq \one_{[\tau,1)} \tilde{K}^\varepsilon_{\delta,\koo} \in D(\Delta)$ since $f(0) = 0$ and $f^\prime(1) = 0$). From above, we know that the solution is unique, whence  $\Delta^{-1} (\one_{[\tau,1)} \Delta K^{\varepsilon}_{\delta,\koo}) = \one_{[\tau,1)}\tilde{K}^{\varepsilon}_{\delta,\koo}$.
			Thus, in case $\theta^\prime = \theta_-$, for any $0 < t \leq T$ we can calculate as follows for small $\varepsilon > 0$:
			\begin{align*}
				&\int_0^t \lVert \Delta_{\vartheta} S_{\vartheta}(s)K^{\varepsilon}_{\delta, \koo} - \theta_- S_\theta(s) \Delta K^{\varepsilon}_{\delta,\koo} \rVert^2 \diff{s}\\
				 &\quad= \int_0^t \lVert  S_\theta(s) (\Delta_\theta - \theta_- \Delta) K^{\varepsilon}_{\delta,\koo} \rVert^2 \diff{s} \\
				&\quad= \frac{1}{2}\langle (-\Delta_\theta)^{-1} (\mathrm{Id} - S_\theta(2t)) (\theta-\theta_-)\Delta K^{\varepsilon}_{\delta,\koo}, (\theta-\theta_-)\Delta K^{\varepsilon}_{\delta,\koo}\rangle \\
				&\quad\leq \frac{1}{2}\langle (-\Delta_\theta)^{-1} (\theta - \theta_-)\Delta K^{\varepsilon}_{\delta,\koo}, (\theta - \theta_-)\Delta K^{\varepsilon}_{\delta,\koo}\rangle \\
				&\quad= \frac{\eta^2}{2} \langle (-\Delta_\theta)^{-1} (\one_{[\tau,1)} \Delta K^{\varepsilon}_{\delta,\koo}), \one_{[\tau,1)}\Delta K^{\varepsilon}_{\delta,\koo}\rangle\\
				&\quad\leq \frac{\eta^2}{2\lowell} \langle (-\Delta)^{-1} (\one_{[\tau,1)} \Delta K^{\varepsilon}_{\delta,\koo}), \one_{[\tau,1)}\Delta K^{\varepsilon}_{\delta,\koo}\rangle= -\frac{\eta^2}{2\lowell} \langle \one_{[\tau,1)}\tilde{K}^{\varepsilon}_{\delta,\koo}, \one_{[\tau,1)}\Delta K^{\varepsilon}_{\delta,\koo}\rangle \\
				&\quad\leq \frac{\eta^2}{2\lowell}\lVert \nabla K^{\varepsilon}_{\delta,\koo}\rVert^2,
			\end{align*}
			where the last line follows from an integration by parts using $\nabla K^{\varepsilon}_{\delta,\koo}(\tau) = \nabla K^{\varepsilon}_{\delta,\koo}(1) = 0$.
			Due to $\lVert \nabla K_{\delta,\koo} - \nabla K^{\varepsilon}_{\delta,\koo} \rVert \to 0$ as $\varepsilon \to 0$,
			for
			\[\mathcal{I}_2^{\varepsilon}(\theta^\prime) \coloneqq  \int_0^T \int_0^t \lVert \Delta_{\vartheta} S_{\vartheta}(t-s)K^{\varepsilon}_{\delta, \koo} - \theta^\prime S_\theta(t-s) \Delta K^{\varepsilon}_{\delta,\koo} \rVert^2 \diff{s} \diff{t}, \]
			we now arrive at
			\[\lim_{\varepsilon \to 0}\mathcal{I}_2^{\varepsilon}(\theta_-)  \leq \frac{T}{2\lowell} \eta^2 \lVert \nabla K_{\delta,\koo} \rVert^2 =\frac{T}{2\lowell}\lVert K^\prime \rVert^2_{L^2} \eta^2 \delta^{-2}.\]
			For general $\theta^\prime \in [\lowell,\upell]$, we have
			\begin{align*}
				&\int_0^t \lVert \Delta_{\vartheta} S_{\vartheta}(s)K^{\varepsilon}_{\delta, \koo} - \theta^\prime S_\theta(s) \Delta K^{\varepsilon}_{\delta,\koo} \rVert^2 \diff{s}\\
				 &\quad\leq \frac{1}{2}\langle (-\Delta_\theta)^{-1} (\theta - \theta^\prime)\Delta K^{\varepsilon}_{\delta,\koo}, (\theta - \theta^\prime)\Delta K^{\varepsilon}_{\delta,\koo}\rangle \\
				&\quad= \frac{1}{2}\langle (-\Delta_\theta)^{-1} \Delta_\theta K^{\varepsilon}_{\delta,\koo}, \Delta_\theta K^{\varepsilon}_{\delta,\koo}\rangle -\theta^\prime \langle (-\Delta_\theta)^{-1} \Delta_\theta K^{\varepsilon}_{\delta,\koo}, \Delta K^{\varepsilon}_{\delta,\koo}\rangle\\
				&\quad \quad+ \frac{(\theta^\prime)^2}{2}\langle (-\Delta_\theta)^{-1} \Delta K^{\varepsilon}_{\delta,\koo}, \Delta K^{\varepsilon}_{\delta,\koo}\rangle\\
				&\quad\leq \frac{1}{2}\lVert \theta \nabla K^{\varepsilon}_{\delta, \koo} \rVert^2 + \frac{(\theta^\prime)^2}{2\lowell} \lVert \nabla K^{\varepsilon}_{\delta,\koo} \rVert^2\\
				&\quad\leq \frac{\upell^2(1 + \lowell^{-1})}{2} \lVert \nabla K^\varepsilon_{\delta,\koo} \rVert^2 \underset{\varepsilon \to 0}{\longrightarrow} \frac{\upell^2(1 + \lowell^{-1})}{2} \delta^{-2}.
			\end{align*}
			It remains to relate $\mathcal{I}_2^\varepsilon(\theta^\prime)$ to $\mathcal{I}_2$. Let
			\[\mathcal{I}_3^{\varepsilon} \coloneqq \int_0^T \lVert \Delta_{\theta} S_{\theta}(t)(K_{\delta,\koo} - K^{\varepsilon}_{\delta,\koo})\rVert^2 \diff{t} , \quad \mathcal{I}_4^{\varepsilon} \coloneqq \int_0^T \lVert S_{\theta}(t) \Delta (K_{\delta,\koo} - K_{\delta,\koo}^{\varepsilon})\rVert^2 \diff{t},\]
			and recall that $(\lambda_k,e_k)_{k \in \N}$ denotes an eigenbasis of $-\Delta_\theta$. It holds
			\begin{align*}
				\mathcal{I}_3^{\varepsilon} & \leq \frac{1}{2} \sum_{k \in \N} \lambda_k \langle K_{\delta,\koo} - K^{\varepsilon}_{\delta,\koo}, e_k \rangle^2 =\frac{1}{2}\lVert (-\Delta_\theta)^{1/2} (K_{\delta,\koo} - K^{\varepsilon}_{\delta,\koo})\rVert^2\\
				&\leq \frac{\upell^2}{2} \lVert \nabla K_{\delta,\koo} - \nabla K^{\varepsilon}_{\delta,\koo}\rVert^2 \underset{\varepsilon \to 0}{\longrightarrow} 0.
			\end{align*}
			Moreover, since
			\[K_{\delta,\koo} - K^{\varepsilon}_{\delta,\koo} = \delta^{-3} K^\prime((\tau-x_{\koo})/\delta) \varphi_\varepsilon,\]
			following the steps from the proof of Lemma \ref{lem:basic_est}.\ref{lem:basic_est2}, we obtain
			\[ \mathcal{I}_4^{\varepsilon} \lesssim_\delta \frac{T}{2\lowell} \lVert \nabla \varphi_\epsilon \rVert^2 = \frac{T}{2\lowell} \varepsilon \lVert \varphi^\prime \rVert^2_{L^2} \underset{\varepsilon \to 0}{\longrightarrow} 0.\]
			Since $\mathcal{I}_2(\theta^\prime) \leq 2 \mathcal{I}_2^\varepsilon(\theta^\prime) +  4 (1 \vee \upell^2)T(\mathcal{I}^{\varepsilon}_3 + \mathcal{I}^{\varepsilon}_4)$, it now follows from the above estimates by taking $\varepsilon \to 0$ that $\mathcal{I}_2(\theta_-) \leq \tfrac{T}{\lowell} \lVert K^\prime \rVert^2_{L^2} \eta^2 \delta^{-2}$ and, in general, $\mathcal{I}_2(\theta^\prime) \lesssim \delta^{-2}$.
			Combining these estimates with Lemma \ref{lem:basic_est}, the assertion follows from \eqref{eq:est1}.
			
			\item \label{err2} We now proceed with the variance bound. Let $Y = (Y(t))_{t \in [0,T]}$ be the Gaussian process defined by $Y(t) \coloneqq \int_0^t \langle \Delta_{\vartheta} S_{\vartheta}(t-u)K_{\delta,\koo} - \theta^\prime S_\theta(t-u) \Delta K_{\delta,\koo},\diff{W_u} \rangle$. Since $R_{\delta,\koo}(\theta^\prime) = \langle X^\Delta_{\delta,\koo}, Y \rangle_{L^2[0,T]}$ and both $X^\Delta_{\delta,\koo}$ and $Y$ are centered $L^2([0,T])$-valued jointly Gaussian processes we use Lemma \ref{lem:GaussCov} below and Lemma \ref{lem:basic_est} to obtain
			\begin{equation}\label{eq:var_rem1}
				\begin{split}
					\mathrm{Var}(R_{\delta,\koo}) &\leq \sqrt{\mathrm{Var}\big(\lVert X^\Delta_{\delta,\koo} \rVert_{L^2[0,T]}^2\big)} \sqrt{\mathrm{Var}\big(\lVert Y \rVert_{L^2([0,T])}^2\big)}\\
					 &= \sqrt{\mathrm{Var}(I_{\delta,\koo})} \sqrt{\mathrm{Var}\Big(\int_0^T Y(t)^2 \diff{t}\Big)}\\
					&\lesssim \delta^{-1} \sqrt{\mathrm{Var}\Big(\int_0^T Y(t)^2 \diff{t}\Big)},
				\end{split}
			\end{equation}
			By Wick's formula and It\^{o}-isometry, we have
			\begin{align*}
				\mathrm{Var}\Big(\int_0^T Y(t)^2 \diff{t}\Big) &= \int_0^T \int_0^T \mathrm{Cov}(Y(t),Y(s))^2 \diff{s} \diff{t}\\
				&= 2\int_0^T \int_0^T \Big(\int_0^{t\wedge s} \langle F_{t-u}(K_{\delta,\koo}), F_{s-u}(K_{\delta,\koo}) \rangle \diff{u}\Big)^2 \diff{s} \diff{t},
			\end{align*}
			where we denoted $F_r(g) \coloneqq \Delta_{\vartheta} S_{\vartheta}(r)g - \theta^\prime S_\theta(r) \Delta g$. As above, we first bound the $\varepsilon$-approximation
			\[\mathcal{I}^\varepsilon \coloneqq \int_0^T \int_0^T \Big(\int_0^{t\wedge s} \langle F_{t-u}(K^{\varepsilon}_{\delta,\koo}), F_{s-u}(K^{\varepsilon}_{\delta,\koo})\rangle \diff{u} \Big)^2\diff{s} \diff{t}.\]
			Since $K^{\varepsilon}_{\delta,\koo} \in D(\Delta_\theta)$ with $\Delta_\theta K^{\varepsilon}_{\delta,\koo} = \theta \Delta K^{\varepsilon}_{\delta,\koo}$, we can write
			\begin{align*}
				&\int_0^{t\wedge s} \langle F_{t-u}(K^\varepsilon_{\delta,\koo}), F_{s-u}(K^\varepsilon_{\delta,\koo}) \rangle \diff{u} \\
				&\quad= \int_0^{t \wedge s} \big\langle S_{\theta}(t-u) (\theta - \theta^\prime) \Delta K^{\varepsilon}_{\delta,\koo},  S_{\theta}(s-u) (\theta - \theta^\prime) \Delta K^{\varepsilon}_{\delta,\koo}\big\rangle \diff{u}\\
				&\quad= \int_0^{t \wedge s} \big\langle S_{\theta}(t+s-2u) (\theta - \theta_-) \Delta K^{\varepsilon}_{\delta,\koo}, (\theta - \theta^\prime) \Delta K^{\varepsilon}_{\delta,\koo}\big\rangle \diff{u} \\
				&\quad\leq \frac{1}{2} \big\langle (-\Delta_\theta)^{-1} S_\theta(\lvert t -s \rvert)(\theta - \theta^\prime) \Delta K^{\varepsilon}_{\delta,\koo}, (\theta - \theta^\prime) \Delta K^{\varepsilon}_{\delta,\koo}\big\rangle.
			\end{align*}
			Calculating as in Lemma \ref{lem:basic_est}.\ref{lem:basic_est4} and part \ref{err1}, we therefore obtain
			\begin{align*}
				\mathcal{I}^\varepsilon &\lesssim \int_0^T \int_0^T \big\langle (-\Delta_\theta)^{-1} S_\theta(\lvert t -s \rvert)(\theta - \theta^\prime) \Delta K^{\varepsilon}_{\delta,\koo}, (\theta - \theta^\prime) \Delta K^{\varepsilon}_{\delta,\koo}\big\rangle^2 \diff{s} \diff{t} \\
				&\leq \frac{T}{2} \big\lVert (-\Delta)_{\theta}^{-3/4}(\theta - \theta^\prime) \Delta K^{\varepsilon}_{\delta,\koo} \big\rVert^4 \\
				&\lesssim \big\lVert (-\Delta)_{\theta}^{-3/4}\Delta_\theta K^{\varepsilon}_{\delta,\koo} \big\rVert^4 + \big\lVert (-\Delta)^{-3/4}\Delta K^{\varepsilon}_{\delta,\koo} \big\rVert^4\\
				&= \big\lVert (-\Delta_\theta)^{1/4} K^{\varepsilon}_{\delta,\koo} \big\rVert^4 + \big\lVert (-\Delta_\theta)^{-3/4}\Delta K^{\varepsilon}_{\delta,\koo} \big\rVert^4.
			\end{align*}
			By Cauchy--Schwarz inequality,
			\begin{align*}
				\big\lVert (-\Delta_\theta)^{1/4} K^{\varepsilon}_{\delta,\koo} \big\rVert^4 &\leq \lVert (-\Delta_\theta)^{1/2} K^{\varepsilon}_{\delta,\koo} \rVert^2 \lVert K^{\varepsilon}_{\delta,\koo} \rVert^2\\ 
				&\leq \upell \lVert \nabla K^{\varepsilon}_{\delta,\koo} \rVert^2 \lVert K^{\varepsilon}_{\delta,\koo} \rVert^2 \underset{\varepsilon \to 0}{\longrightarrow} \lVert K^\prime \rVert^2_{L^2} \delta^{-2},
			\end{align*}
			and Lemma \ref{lem:sol_interpol} yields
			\begin{align*}
				\norm{(-\Delta_\theta)^{-3/4}\Delta K^{\varepsilon}_{\delta,\koo}} &\lesssim \lVert K^{\varepsilon}_{\delta,\koo} \rVert^{1/2} \lVert K^{\varepsilon}_{\delta,\koo} \rVert^{1/2}_{H^1((0,1))} + \lVert  K^{\varepsilon}_{\delta,\koo} \rVert_\infty\\
				&\underset{\varepsilon \to 0}{\longrightarrow} \lVert K_{\delta,\koo} \rVert^{1/2} \lVert K_{\delta,\koo} \rVert^{1/2}_{H^1((0,1))} + \lVert  K_{\delta,\koo} \rVert_\infty \sim \delta^{-1/2}.
			\end{align*}
			It follows that $\lim_{\varepsilon \to 0} \mathcal{I}^\varepsilon \lesssim \delta^{-2}$, and analogously to part \ref{err1}, we therefore obtain
			\begin{align*}
				\mathrm{Var}\big(\lVert Y \rVert^2_{L^2([0,T])}\big)&= \int_0^T \int_0^T \Big(\int_0^{t\wedge s} \langle F_{t-u}(K_{\delta,\koo}), F_{s-u}(K_{\delta,\koo}) \rangle \diff{u}\Big)^2 \diff{s} \diff{t} \\
				&\lesssim \lim_{\varepsilon \to 0} \mathcal{I}^\varepsilon \lesssim \delta^{-2}.
			\end{align*}
			Thus, \eqref{eq:var_rem1} implies $\mathrm{Var}(R_{\delta,\koo}(\theta^\prime)) \lesssim \delta^{-2}$.
			
			\item\label{err3} We have
			\begin{equation}\label{eq:rep_rem}
				R_{\delta,\koo}(\theta^\prime) = \int_0^T \int_0^t \langle S_{\theta}(t-s)\Delta K_{\delta,\koo}, \diff{W_s} \rangle \int_0^t \langle \Delta_\theta S_\theta(t-s) K_{\delta,\koo}, \diff{W_s} \rangle \diff{t} - \theta^\prime I_{\delta,\koo}.
			\end{equation}
			By now familiar calculations give
			\begin{align*}
				&\E\Big[\int_0^T \int_0^t \langle S_{\theta}(t-s)\Delta K_{\delta,\koo}, \diff{W_s} \rangle \int_0^t \langle \Delta_\theta S_\theta(t-s) K_{\delta,\koo}, \diff{W_s} \rangle \diff{t}\Big] \\
				&\quad =\int_0^T \int_0^t \langle S_{\theta}(t-s)\Delta K_{\delta,\koo}, \Delta_\theta S_{\theta}(t-s) K_{\delta,\koo} \rangle \diff{s} \diff{t}\\
				&\quad= \int_0^T \int_0^t \langle \Delta_\theta S_\theta(2s) K_{\delta,\koo} , \Delta K_{\delta,\koo} \rangle \diff{s} \diff{t}\\
				&\quad= \frac{1}{2} \sum_{k \in \N} \int_0^T (\mathrm{e}^{-2\lambda_k t} -1)\langle \Delta K_{\delta,\koo}, e_k\rangle \langle K_{\delta,\koo}, e_k\rangle \diff{t} \\
				&\quad= \frac{1}{2} \langle -\Delta K_{\delta,\koo}, K_{\delta,\koo} \rangle + \frac{1}{4}\langle (-\Delta_\theta)^{-1} (\operatorname{Id} - S_\theta(2T)) \Delta K_{\delta,\koo}, K_{\delta,\koo} \rangle \\
				&\quad= \frac{T}{2}\lVert K^\prime \rVert^2_{L^2}\delta^{-2} + \mathcal{O}(\delta^{-1/2}),
			\end{align*}
			where for the last line we used that by self-adjointness, Cauchy--Schwarz inequality and Lemma \ref{lem:sol_c2},
			\begin{align*}\lvert\langle (-\Delta_\theta)^{-1} (\operatorname{Id} - S_\theta(2T)) \Delta K_{\delta,\koo}, K_{\delta,\koo} \rangle \rvert&\leq \lVert (-\Delta_\theta)^{-1} \Delta K_{\delta,\koo} \rVert \lVert (\operatorname{Id} -S_\theta(2T)) K_{\delta,\koo} \rVert\\
				&\lesssim  \delta^{-1/2}.
			\end{align*}
			Since by Lemma \ref{lem:basic_est} it holds
			\[\frac{T}{2\upell} \lVert K^\prime \rVert^2_{L^2} \delta^{-2} + \mathcal{O}(\delta^{-1}) \leq \E[I_{\delta,\koo}] \leq \frac{T}{2\lowell} \lVert K^\prime \rVert^2_{L^2} \delta^{-2} + \mathcal{O}(\delta^{-1}),\]
			it therefore follows from \eqref{eq:rep_rem} that there exists $\theta^0_\circ \in [\lowell,\upell]$ such that
			\[\E[R_{\delta,\koo}(\theta^0_\circ)] = \mathcal{O}(\delta^{-1}).\]
		\end{enumerate}
	\end{proof}
	
	\begin{lemma}\label{lem:GaussCov}
		For two centred, jointly Gaussian random variables $X,Y$ with values in a real separable Hilbert space we have
		\[ \Var(\scapro{X}{Y})\le \Var(\norm{X}^2)^{1/2}\Var(\norm{Y}^2)^{1/2}.\]
	\end{lemma}
	\begin{proof}
		Let us decompose $Y=LX+Z$ with a bounded linear operator $L$ such that $LX \coloneq\E[Y\,|\,X]$, $Z\coloneq Y-\E[Y\,|\,X]$ and $Z$ is independent of $X$. We denote by $Q_X$, $Q_Z$ the trace-class covariance operators of $X$ and $Z$. Then using an orthonormal eigensystem $(e_i,\lambda_i)$ of $Q_X$ we find
		\[ \Var(\norm{X}^2)=\sum_{i\ge 1} \Var(\scapro{X}{e_i}^2)=\sum_{i\ge 1}2\lambda_i^2=2\norm{Q_X}_{\text{HS}}^2.\]
		The same argument with orthonormal systems of $Q_Z$ and $Q_X^{1/2}L^\ast LQ_X^{1/2}$, respectively, yields
		\[ \Var(\norm{Z}^2)=2\norm{Q_Z}_{\text{HS}}^2,\quad \Var(\norm{LX}^2)=2\norm{Q_X^{1/2}L^\ast LQ_X^{1/2}}_{\text{HS}}^2.\]
		If $L_\sigma=\frac12(L+L^\ast)$ denotes the symmetrisation of $L$, this argument also yields
		\[ \Var(\scapro{X}{LX})=\Var(\scapro{X}{L_\sigma X})=2\norm{Q_X^{1/2}L_\sigma Q_X^{1/2}}_{\text{HS}}^2.\]
		Due to the independence of $X$ and $Z$, we may disintegrate to obtain
		\[ \Var(\scapro{X}{Z})=\E[\scapro{Q_ZX}{X}]=\E[\scapro{Q_Z}{XX^\ast}_{\text{HS}}]=\scapro{Q_X}{Q_Z}_{\text{HS}}\]
		as well as
		\[ \Var(\scapro{LX}{Z})=\Var(\scapro{L_\sigma X}{Z})=\E[\scapro{Q_Z}{L_\sigma XX^\ast L_\sigma}_{\text{HS}}]=\scapro{Q_X}{L_\sigma Q_ZL_\sigma}_{\text{HS}}.\]
		Using these identities, we arrive at
		\begin{align*}
			\Var(\scapro{X}{Y})&=\Var(\scapro{X}{LX}+\scapro{X}{Z})\\
			&=\Var(\scapro{X}{LX})+\Var(\scapro{X}{Z})+2\Cov(\scapro{X}{LX},\scapro{X}{Z}) \\
			&= 2\norm{Q_X^{1/2}L_\sigma Q_X^{1/2}}_{\text{HS}}^2+\scapro{Q_X}{Q_Z}_{\text{HS}}+0\\
			&=\scapro{Q_X}{2L_\sigma Q_XL_\sigma+Q_Z}_{\text{HS}},
		\end{align*}
		where the covariance vanishes since the product of the arguments is linear in $Z$, which is centered and independent of $X$. By the Cauchy--Schwarz inequality and the identity $\norm{T^\ast T}_{\text{HS}}=\norm{TT^\ast}_{\text{HS}}$ for $T=Q_X^{1/2}L_\sigma$, we obtain further
		\begin{align*}
			\Var(\scapro{X}{Y})^2& \le \norm{Q_X}_{\text{HS}}^2\Big(\norm{2L_\sigma Q_XL_\sigma}_{\text{HS}}^2+\norm{Q_Z}_{\text{HS}}^2+2\scapro{2L_\sigma Q_XL_\sigma}{Q_Z}_{\text{HS}}\Big)\\
			& =\tfrac12\Var(\norm{X}^2)\Big(4\norm{Q_X^{1/2}L_\sigma^2 Q_X^{1/2}}_{\text{HS}}^2+\norm{Q_Z}_{\text{HS}}^2+4\scapro{L_\sigma Q_XL_\sigma}{Q_Z}_{\text{HS}}\Big)\\
			& \le\Var(\norm{X}^2)\Big(\Var(\norm{LX}^2)+\tfrac14 \Var(\norm{Z}^2)+2\Var(\scapro{LX}{Z})\Big),
		\end{align*}
		where the last line follows  from the partial ordering $L_\sigma^2\eqcolon\Re(L)^2\le \abs{L}^2 \coloneq L^\ast L$ and
		from $\scapro{L_\sigma Q_XL_\sigma}{Q_Z}_{\text{HS}}=\scapro{ Q_X}{L_\sigma Q_ZL_\sigma}_{\text{HS}}$. Finally, note
		\[ \Var(\norm{Y}^2)=\Var(\norm{LX+Z}^2)=\Var(\norm{LX}^2)+\Var(\norm{Z}^2)+\Var(2\scapro{LX}{Z}),\]
		since all covariances between $\norm{LX}^2$, $\norm{Z}^2$ and $\scapro{LX}{Z}$ vanish due to independence or symmetry in $Z$ (use $Z\stackrel{\text{d}}{=}-Z$), such that the asserted inequality follows.
	\end{proof}
	
	We now provide the proof of the expectation result needed for the application of the consistency theorem for M-estimators.
	
	\begin{proof}[Proof of Lemma \ref{lem:approxZ}]
		It is enough to show
		\[\lim_{\delta \to 0}\sup_{\chi\in \Theta \times (0,1]} \Big\vert \delta^3 \E[I_{T,\delta}(\chi)] - \mathcal{Z}(\chi^\prime) \Big\vert = 0,\]	
		since $\E[I_{\delta,\koo}] \lesssim \delta^{-2}$.
		By Lemma \ref{lem:basic_est}, it holds
		\begin{align*}
			\delta^3 \E[I_{T,\delta}(\theta_-,\theta_+,h)] &= \frac{T}{4}\lVert K^\prime \rVert^2_{L^2} \delta\sum_{i=1}^{\delta^{-1}} \frac{(\theta_{\delta,i}(\lceil nh \rceil) - \theta^0_{\delta,i})^2}{\theta^0_{\delta,i}} + \mathcal{O}(\delta^2),
		\end{align*}
		and we always have
		\[\frac{\lvert \theta_{\delta,i}(k) - \theta^0_{\delta,i} \rvert^2}{\theta^0_{\delta,i}} \leq \frac{(\upell - \lowell)^2}{\lowell}, \quad i,k \in [n]\]
		It is thus easily verified that
		\begin{align*}
			&\Big\vert \delta^3 \E[I_{T,\delta}(\chi)] - \frac{T}{4}\lVert K^\prime \rVert^2_{L^2} \int_0^1 \frac{(\theta_{\chi^\prime}(x) - \theta^0_\delta(x))^2}{\theta^0_\delta(x)} \diff{x} \Big\vert \notag\\
			&\,\lesssim \frac{(\theta_- - \theta^0_-(\delta))^2}{\theta^0_-(\delta)} \big\lvert\delta(\lceil \tau^0/\delta \rceil \wedge \lceil h/\delta \rceil -1) - \tau^0 \wedge h \big\rvert + \frac{(\theta_+ - \theta^0_+(\delta))^2}{\theta^0_+(\delta)} \big\lvert\delta(\lceil \tau^0/\delta \rceil \vee \lceil h/\delta \rceil) - \tau^0 \vee h \big\rvert \notag\\
			&\quad + \big\lvert\delta(\lvert \lceil \tau^0/\delta \rceil - \lceil h/\delta \rceil\rvert -1)^+ - \lvert \tau^0 - h\rvert \big\rvert \frac{(\theta_+ - \theta_-^0(\delta))^2 + (\theta_- - \theta_+^0(\delta))^2}{\theta^0_+(\delta) \wedge \theta^0_-(\delta)} + \delta \frac{(\upell - \lowell)^2}{\lowell} + \delta^2\notag\\
			&\lesssim \delta \frac{(\upell-\lowell)^2}{\lowell}.\notag
		\end{align*}
		We therefore obtain the uniform convergence
		\begin{equation}\label{eq:uni_exp1}
			\lim_{\delta \to 0}\sup_{\chi \in \Theta \times (0,1]} \Big\vert \delta^3 \E[I_{T,\delta}(\theta_-,\theta_+,h)] - \frac{T}{4}\lVert K^\prime \rVert^2_{L^2} \int_0^1 \frac{(\theta_{\chi^\prime}(x) - \theta^0_\delta(x))^2}{\theta^0_\delta(x)} \diff{x} \Big\vert = 0.
		\end{equation}
		Moreover, for $\theta^\ast \coloneqq \theta_-^\ast \one_{(0,\tau^0)} + \theta_+^\ast \one_{[\tau^0,1)}$, we have
		\begin{equation*}\label{eq:disc_err2}
			\begin{split}
				&\frac{\lvert \theta^\ast_{\pm} (\theta_{\pm}-\theta^0_{\pm}(\delta))^2 - \theta^0_{\pm}(\delta)(\theta_\pm - \theta^\ast_{\pm})^2 \rvert}{\theta^0_{\pm}(\delta)\theta^\ast_{\pm}}\\
				&\quad\leq \frac{\theta_{\pm}^\ast \lvert (\theta_{\pm} - \theta_{\pm}^0(\delta))^2 - (\theta_{\pm} - \theta_{\pm}^\ast)^2 \rvert + (\theta_{\pm}-\theta^\ast_{\pm})^2 \lvert \theta_{\pm}^0(\delta) - \theta^\ast_{\pm}\rvert}{\lowell^2}\\
				&\quad\leq \frac{\upell + (\upell - \lowell)^2}{\lowell^2}\big((\theta_{\pm}^0(\delta) - \theta^\ast_\pm)^2 + \lvert \theta_{\pm}^0(\delta) - \theta^\ast_\pm \rvert \big) \underset{\delta \to 0}{\longrightarrow} 0,
			\end{split}
		\end{equation*}
		and similarly
		\begin{equation*}\label{eq:disc_err3}
			\begin{split}
				\frac{\lvert \theta^\ast_{\pm} (\theta_{\mp}-\theta^0_{\pm}(\delta))^2 - \theta^0_{\pm}(\delta)(\theta_\mp - \theta^\ast_{\pm})^2 \rvert}{\theta^0_{\pm}(\delta)\theta^\ast_{\pm}} \leq \frac{\upell + (\upell - \lowell)^2}{\lowell^2}\big((\theta_{\pm}^0(\delta) - \theta^\ast_\pm)^2 + \lvert \theta_{\pm}^0(\delta) - \theta^\ast_\pm \rvert \big)
			\end{split}
		\end{equation*}
	vanishes as $\delta\to0$.
		By the piecewiese constant nature of $\theta_{\chi^\prime},\theta^0_\delta,\theta^\ast$, it is therefore straightforward to show that
		\[
		\lim_{\delta \to 0}\sup_{\chi \in \Theta \times (0,1]} \Big\vert \int_0^1 \frac{(\theta_{\chi^\prime}(x) - \theta^0_\delta(x))^2}{\theta^0_\delta(x)} \diff{x} -   \int_0^1 \frac{(\theta_{\chi^\prime}(x) - \theta^\ast(x))^2}{\theta^\ast(x)} \diff{x} \Big\vert = 0.
		\]
		The claim then follows by using \eqref{eq:uni_exp1} and the triangle inequality.
	\end{proof}
	
	Finally, we give the proofs for the local fluctuation bounds on the centered empirical processes $\mathcal{Z}_\delta(\cdot) - \E[\mathcal{Z}_\delta(\cdot)]$ and $\mathcal{L}_\delta(\cdot) - \E[\mathcal{L}_\delta(\cdot)]$ around the true parameter $\chi^0(\delta)$.
	\begin{proof}[Proof of Lemma \ref{lem:aux}]
		Due to $\mathcal{Z}_\delta(\chi^0(\delta)) = 0$, the assertion is equivalent to the claim that
		\[\E\Big[\sup_{\tilde{d}_\delta(\chi,\chi^0(\delta)) < \varepsilon} \big\lvert \delta^3 (\mathcal{Z}_\delta(\chi) - \E[\mathcal{Z}_\delta(\chi)]) \big\rvert \Big] \lesssim  \delta^3 + \delta^{1/2}\varepsilon^3 + \delta\varepsilon^2 + \delta^{3/2} \varepsilon.\]
		Let $\varepsilon \leq 1$. With the notation from the proof of Theorem \ref{theo:cons}, we have the bound
		\begin{equation}\label{eq:decomp}
			\begin{split}
				\E\Big[\sup_{\tilde{d}_\delta(\chi,\chi^0(\delta)) < \varepsilon} \big\lvert  \delta^3(\mathcal{Z}_\delta(\chi) - \E[\mathcal{Z}_\delta(\chi)]) \big\rvert \Big] &\lesssim
				\E\Big[\sup_{\tilde{d}_\delta(\chi,\chi^0(\delta)) < \varepsilon}  \delta^3 \big\lvert  I_{T,\delta}(\chi) - \E[I_{T,\delta}(\chi)]  \big\rvert \Big] \\
				&\quad + \delta^3 \E\Big[\sup_{\tilde{d}_\delta(\chi,\chi^0(\delta)) < \varepsilon}  \big\lvert  M_{T,\delta}(\chi)  \big\rvert \Big] \\
				&\quad + \delta^2(\delta \wedge \varepsilon^2) \E[\lvert I_{\delta,\koo} - \E[I_{\delta,\koo}] \rvert]\\
				&\quad + \delta^{5/2}(\delta^{1/2} \wedge \varepsilon)\E[\lvert M_{\delta,\koo} \rvert]),
			\end{split}
		\end{equation}
		where for the last two summands we used that $\tilde{d}_\delta(\chi,\chi^0(\delta)) \leq \varepsilon$ implies
		\[\lvert \theta_{\delta,\koo}(\lceil h/\delta\rceil) - \theta^0_{\delta,\koo} \rvert^2 \lesssim 1 \wedge \frac{\varepsilon^2}{\delta}.\]
		We also observe that $\tilde{d}(\chi,\chi^0(\delta)) < \varepsilon$ implies $(\lvert  [\tau^0]_\delta - [h]_\delta  \rvert - \delta)^+ < \varepsilon^2$, giving
		\begin{equation}\label{eq:eps_size}
			\lvert \lceil \tau^0/\delta \rceil - \lceil h/\delta \rceil \rvert \leq  1+ \varepsilon^2 \delta^{-1}.
		\end{equation}
		Consequently, $\lvert \{\lceil h/\delta \rceil : (\lvert [\tau^0]_\delta - [h]_\delta \rvert - \delta)^+ < \varepsilon^2\} \rvert \lesssim 1 + \varepsilon^2\delta^{-1}$.  Moreover,  if $\varepsilon^2 < \delta/2$, then $\tilde{d}_\delta(\chi,\chi^0(\delta)) <\varepsilon$ implies $\lceil h/\delta\rceil = \lceil \tau^0/\delta \rceil \pm 1 = \koo \pm 1$, i.e.,
		\begin{equation}\label{eq:eps_size1}
			\lceil h/\delta \rceil \wedge \koo +1 < \lceil h/\delta \rceil \vee \koo \implies \delta \leq 2\varepsilon^2.
		\end{equation}
		Thus, for the first term, using \eqref{eq:err5} and Proposition \ref{prop:conc}, it follows with a union bound that, for any $z > 0$,
		\begin{align*}
			&\PP\Big(\delta^3 \sup_{\tilde{d}_\delta(\chi,\chi^0(\delta)) < \varepsilon}  \big\lvert  I_{T,\delta}(\chi) - \E[I_{T,\delta}(\chi)]  \big\rvert \geq z \Big)\\
			&\,\leq \PP\Big(\delta^3 \varepsilon^2 \sup_{\tilde{d}_\delta(\chi,\chi^0(\delta)) < \varepsilon}  \Big\lvert  \sum_{i=1,i\neq\koo}^{\lceil h/\delta \rceil \wedge \koo -1} (I_{\delta,i} - \E(I_{\delta,i}))  \Big\rvert \geq z/4 \Big)\\
			&\,\quad + \PP\Big(\delta^3 \varepsilon^2 \sup_{\tilde{d}_\delta(\chi,\chi^0(\delta)) < \varepsilon}  \Big\lvert  \sum_{i=\lceil h/\delta \rceil \vee \koo +1 ,i\neq \koo}^{n} (I_{\delta,i} - \E(I_{\delta,i}))  \Big\rvert \geq z4 \Big) \\
			&\,\quad + \PP\Big(\delta^3 ((\upell - \lowell)^2 + 2 \varepsilon^2) \sup_{\tilde{d}_\delta(\chi,\chi^0(\delta)) < \varepsilon}  \Big\lvert  \sum_{i=\lceil h/\delta \rceil \wedge \koo + 1,i\neq \koo}^{\lceil h/\delta \rceil \vee \koo - 1} (I_{\delta,i} - \E(I_{\delta,i}))  \Big\rvert \one_{\big\{\frac{\lceil h/\delta \rceil \wedge \koo +1}{\lceil h/\delta \rceil \vee \koo}< 1  \big\}} \geq z/4 \Big)\one_{\{\delta \leq 2\varepsilon^2\}}\\
			&\,\quad + \PP\Big(\delta^2(\delta \wedge \varepsilon^2) \sup_{\tilde{d}(\chi,\chi^0(\delta) < \varepsilon }\lvert I_{\delta,\lceil h/\delta \rceil} - \E[I_{\delta,\lceil h/\delta \rceil}] \rvert \geq z/4\Big)\\
			&\, \lesssim (1 + \varepsilon^2 \delta^{-1}) \Big(\exp\Big(-\frac{\delta^{-3} \varepsilon^{-4} z^2}{C(1 + \varepsilon^{-2}z)} \Big) + \exp\Big(-\frac{\delta^{-3} z^2}{C(\varepsilon^2 +   z)} \Big) \one_{\{\delta \leq 2\varepsilon^2\}}\\
			&\qquad\qquad\qquad+ \exp\Big(-\frac{\delta^{-2}(\delta^{-1}\vee \varepsilon^{-2})z^2}{C(\delta \wedge \varepsilon^2 + z)} \Big)\Big)\\
			&\,\lesssim (1+ \varepsilon^2 \delta^{-1}) \exp\Big(-\frac{\delta^{-3} z^2}{C(\varepsilon^2 +   z)} \Big).
		\end{align*}
		Since
		\begin{align*}
			&(1+ \varepsilon^2 \delta^{-1}) \int_0^\infty \exp\Big(-\frac{\delta^{-3} z^2}{C( \varepsilon^2 + z)} \Big) \diff{z}\\
			&\,\leq (1+ \varepsilon^2 \delta^{-1}) \int_0^\infty \Big(\exp\Big(-\frac{\delta^{-3} \varepsilon^{-2} z^2}{2C}\Big) + \exp\Big(-\frac{\delta^{-3} z}{2C} \Big) \Big) \diff{z} \lesssim \delta^{3/2} \varepsilon + \delta^{1/2} \varepsilon^3 + \delta^3,
		\end{align*}
		we obtain for $\varepsilon \leq 1$,
		\begin{equation}\label{eq:err_qv}
			\begin{split}
				&\E\Big[\sup_{\tilde{d}_\delta(\chi,\chi^0(\delta)) < \varepsilon}  \delta^3 \big\lvert  I_{T,\delta}(\chi) - \E[I_{T,\delta}(\chi)]  \big\rvert \Big] \\
				&\quad= \int_0^\infty \PP\Big(\delta^3 \sup_{\tilde{d}_\delta(\chi,\chi^0(\delta)) < \varepsilon}  \big\lvert  I_{T,\delta}(\chi) - \E[I_{T,\delta}(\chi)]  \big\rvert \geq z \Big) \diff{z}\\
				&\quad\lesssim \delta^{3/2} \varepsilon + \delta^{1/2} \varepsilon^3 + \delta^3.
			\end{split}
		\end{equation}
		We now treat the second summand on the right hand side of \eqref{eq:decomp}.
		We first observe that, similarly to \eqref{eq:err5}, we may write for any $(\theta_-,\theta_+,\theta_\circ,h) \in \{\chi \in \Theta \times (0,1]: \tilde{d}_\delta(\chi,\chi^0(\delta)) < \varepsilon\}$, with $k = \lceil h/\delta \rceil$,
		\[\lvert M_{T,\delta}(\chi) \rvert \lesssim \varepsilon \Big \lvert \sum_{i \in \Lambda_1(k)} M_{\delta,i} \Big \rvert + \varepsilon \Big\lvert \sum_{i \in \Lambda_2(k)} M_{\delta,i} \Big\rvert + ((\upell - \lowell) +2 \varepsilon) \Big\lvert \sum_{i \in \Lambda_3(k)} M_{\delta,i} \Big\rvert + \big(1 \wedge \frac{\varepsilon}{\sqrt{\delta}}\big) \lvert M_{\delta,k} \rvert  ,\]
		where
		\begin{align*}
			\Lambda_1(k) &\coloneqq  \big\{i \in [\delta^{-1}]:  i \neq \koo, i \leq k \wedge \koo -1\big\}, \quad \Lambda_2(k) \coloneqq \big\{i \in [\delta^{-1}]:  i \neq \koo, i \geq (k \vee \koo) +1\big\} \\
			\Lambda_3(k) &\coloneqq \big\{i \in [\delta^{-1}]:  i \neq \koo, (k\wedge \koo) +1 \leq  i \leq k \vee \koo -1\big\}.
		\end{align*}
		Using \eqref{eq:eps_size} and \eqref{eq:eps_size1}, we therefore obtain,
		\begin{equation}\label{eq:decomp_mart}
			\begin{split}
				\E\Big[\delta^3 \sup_{\tilde{d}_\delta(\chi,\chi^0(\delta)) < \varepsilon}  \big\lvert  M_{T,\delta}(\chi)  \big\rvert \Big] &\lesssim \sum_{l=1}^2 \sum_{k \in \Lambda(\varepsilon)} \E\Big[\varepsilon \delta^3 \Big\lvert \sum_{i \in \Lambda_l(k)} M_{\delta,i} \Big\rvert \Big] \\
				& \quad + \sum_{k \in \Lambda(\varepsilon)} \delta^3\E\Big[ \Big\lvert \sum_{i \in \Lambda_3(k)} M_{\delta,i} \Big\rvert \Big] \one_{\{\delta \leq 2\varepsilon^2\}}\\
				&\quad + (\delta^3 \wedge \varepsilon\delta^{5/2})\sum_{k \in \Lambda(\varepsilon)}\E\big[\lvert M_{\delta,k}\rvert\big],
			\end{split}
		\end{equation}
		where $\Lambda(\varepsilon) \coloneqq \{k \in [\delta^{-1}] : (\lvert \koo - k \rvert -1)^+ \leq \varepsilon^2 \delta^{-1}\}$, whose size is bounded by a multiple of $1 + \varepsilon^2 \delta^{-1}$.
		For any $k,l$, $\mathcal{M}_{k,l}(t) \coloneqq \sum_{i \in \Lambda_l(k)} \int_0^t X^\Delta_{\delta,i}(s) \diff{B_{\delta,i}(s)}$ is a martingale in $t$, and, by independence of $(B_{\delta,i})_{i \in [\delta^{-1}]}$, its quadratic variation is given by
		\[\big\langle \mathcal{M}_{k,l} \big\rangle_t = \sum_{i \in \Lambda_l(k)} \int_0^t X^{\Delta}_{\delta,i}(s)^2 \diff{s}, \quad k \in [\delta^{-1}], l \in \{1,2,3\}, t \geq 0.\]
		In particular, $\mathcal{M}_{k,l}(T) = \sum_{i \in \Lambda_l(k)} M_{\delta,i}$ and $\langle \mathcal{M}_{k,l} \rangle_T = \sum_{i \in \Lambda_l(k)} I_{\delta,i}$. Hence, using Cauchy--Schwarz inequality, Lemma \ref{lem:basic_est} and the fact that for $\delta \leq 2\varepsilon^2$ and $k \in \Lambda(\varepsilon)$, $\lvert \Lambda_3(k) \rvert \lesssim \varepsilon^2 \delta^{-1}$, we obtain from \eqref{eq:decomp_mart}
		\begin{equation} \label{eq:mart_max}
			\begin{split}
				&\E\Big[\delta^3 \sup_{\tilde{d}_\delta(\chi,\chi^0(\delta)) < \varepsilon}  \big\lvert  M_{T,\delta}(\chi)  \big\rvert\Big] \\
				&\lesssim \delta^3(1+ \varepsilon^2\delta^{-1}) \varepsilon \Big(\sum_{i \in [\delta^{-1}]}\E[I_{\delta,i}] \Big)^{1/2} + \one_{\{\delta \leq 2\varepsilon^2\}}\delta^3 \sum_{k \in \Lambda(\varepsilon)} \Big(\sum_{i \in \Lambda_3(k)}\E[I_{\delta,i}] \Big)^{1/2}\\
				&\quad + (\delta^3 \wedge \varepsilon\delta^{5/2})(1+ \varepsilon^2 \delta^{-1}) \max_{k\in [\delta^{-1}] \setminus \{\koo\}}\E[ I_{\delta,k}]^{1/2}\\
				&\lesssim \delta^3\varepsilon(1+ \varepsilon^2\delta^{-1}) \delta^{-3/2} + \delta^3 \varepsilon^2 \delta^{-1} (\varepsilon^2 \delta^{-3})^{1/2} + \varepsilon\delta^{3/2}(1+ \varepsilon^2 \delta^{-1})\lesssim \delta^{3/2} \varepsilon + \varepsilon^3 \delta^{1/2}.
			\end{split}
		\end{equation}
		Finally, using Lemma \ref{lem:basic_est}, It\^{o} isometry and Cauchy--Schwarz inequality,
		\begin{equation}\label{eq:cp_bound}
			\begin{split}
				&\delta^2(\delta \wedge \varepsilon^2)\E[\lvert I_{\delta,\koo} - \E[I_{\delta,\koo}] \rvert] + \delta^{5/2}(\delta^{1/2} \wedge \varepsilon)\E[\lvert M_{\delta,\koo} \rvert]\\
				&\,\leq \delta^2(\delta \wedge \varepsilon^2)(\mathrm{Var}(I_{\delta,\koo}))^{1/2} + \delta^{5/2}(\delta^{1/2} \wedge \varepsilon)\E[I_{\delta,\koo}]^{1/2}\\
				&\lesssim \delta \varepsilon^2 + \delta^{3/2} \varepsilon.
			\end{split}
		\end{equation}
		Thus, inserting \eqref{eq:err_qv}, \eqref{eq:mart_max} and \eqref{eq:cp_bound} into \eqref{eq:decomp}, the assertion follows.
	\end{proof}

	\begin{proof}[Proof of Corollary \ref{coro:aux}]
		Using Proposition \ref{lem:remainder} and Lemma \ref{lem:aux}, it holds
		\begin{align*}
			&\E\Big[\sup_{\tilde{d}_\delta(\chi,\chi^0(\delta)) < \varepsilon} \big\lvert (\mathcal{L}_\delta - \tilde{\mathcal{L}}_\delta)(\chi) - ( \mathcal{L}_\delta - \tilde{\mathcal{L}}_{\delta})(\chi^0(\delta)) \big\rvert \Big] \\
			&\, \lesssim \E\Big[\sup_{\tilde{d}_\delta(\chi,\chi^0(\delta)) < \varepsilon} \big\lvert (\mathcal{L}_\delta - \E[\mathcal{L}_\delta])(\chi) - (\mathcal{L}_\delta - \E[\mathcal{L}_\delta])(\chi^0(\delta)) \big\rvert \Big] \\
			&\,\quad  +  \sup_{\tilde{d}_\delta(\chi,\chi^0(\delta)) < \varepsilon} \lvert \theta_{\delta,\koo}(\lceil h/\delta\rceil) - \theta^0_\circ(\delta) \rvert \delta^3\lvert \E[R_{\delta,\koo}(\theta^0_\circ(\delta))]\rvert\\
			&\, \lesssim \E\Big[\sup_{\tilde{d}_\delta(\chi,\chi^0(\delta)) < \varepsilon} \big\lvert (\mathcal{Z}_\delta - \E[\mathcal{Z}_\delta])(\chi) - (\mathcal{Z}_\delta - \E[\mathcal{Z}_\delta])(\chi^0(\delta)) \big\rvert \Big] \\
			&\,\quad  + \sup_{\tilde{d}_\delta(\chi,\chi^0(\delta)) < \varepsilon} \lvert \theta_{\delta,\koo}(\lceil h/\delta\rceil) - \theta^0_\circ(\delta) \rvert \delta^3 (\operatorname{Var}(R_{\delta,\koo}(\theta^0_\circ(\delta))))^{1/2}\\
			&\quad + \sup_{\tilde{d}_\delta(\chi,\chi^0(\delta)) < \varepsilon} \lvert \theta_{\delta,\koo}(\lceil h/\delta\rceil) - \theta^0_\circ(\delta) \rvert \delta^3\lvert \E[R_{\delta,\koo}(\theta^0_\circ(\delta))]\rvert\\
			&\, \lesssim  \tilde{\psi}_\delta(\varepsilon) + \delta^2 \sup_{\tilde{d}_\delta(\chi,\chi^0(\delta)) < \varepsilon} \lvert \theta_{\delta,\koo}(\lceil h/\delta\rceil) - \theta_\circ^0(\delta) \rvert \\
			&\,\lesssim  \delta^3 + \delta^{1/2} \varepsilon^\gamma + \delta \varepsilon^{\varrho} + \delta^{3/2}\varepsilon.
		\end{align*}
		For the last inequality, we used that  $\tilde{d}_\delta(\chi,\chi^0(\delta)) < \varepsilon$ implies
		\[\lvert \theta_{\delta,\koo}(\lceil h/\delta \rceil) - \theta^0_{\delta,\koo} \rvert = \lvert \theta_{\delta,\koo}(\lceil h/\delta \rceil) - \theta^0_\circ(\delta) \rvert \leq \varepsilon/\sqrt{\delta}.\]
	\end{proof}
	
	\section{Remaining proofs of Section \ref{sec:clt}}\label{app:proof_clt}
	We start with verifying the representation of the estimator $\hat k=\hat k(\delta)$ defining the change point estimator via the relation $\hat\tau=\hat k\delta$.
	
	\begin{proof}[Proof of Lemma \ref{lem:cp}]
		Write $\theta_\pm = \theta_\pm(\delta)$.
		First subtracting
		\[\sum_{i=1}^n  \Big(\theta_+ \int_0^T X^{\Delta}_{\delta,i}(t) \diff{X_{\delta,i}(t)} - \frac{\theta_+^2}{2} \int_0^T X^{\Delta}_{\delta,i}(t)^2 \diff{t}\Big)\]
		from the maximum in the definition of $\hat{k}$ in \eqref{eq:cpest} and then adding
		\[
		\sum_{i=1}^{\ko} \Big(\eta \int_0^T X^{\Delta}_{\delta,i}(t) \diff{X_{\delta,i}(t)} + \frac{\theta_-^2 -\theta_+^2}{2} \int_0^T X^{\Delta}_{\delta,i}(t)^2 \diff{t}\Big),
		\]
		it follows that $\hat{k} = \argmaxalt_{k=1,\ldots,\delta^{-1}} \tilde{Z}_k$, where
		\[
		\tilde{Z}_k =
		\begin{cases} 0, &k=\ko,\\
			-\eta \sum_{i=\ko+1}^{k} \int_0^T X^{\Delta}_{\delta,i} \diff{X_{\delta,i}(t)} - \frac{\theta_-^2 - \theta_+^2}{2} \sum_{i=\ko+1}^k\int_0^T X^{\Delta}_{\delta,i}(t)^2 \diff{t}, & k > \ko,\\
			\eta \sum_{i=k + 1}^{\ko} \int_0^T X^{\Delta}_{\delta,i} \diff{X_{\delta,i}(t)} + \frac{\theta_-^2 - \theta_+^2}{2} \sum_{i=k+1}^{\ko}\int_0^T X^{\Delta}_{\delta,i}(t)^2 \diff{t}, & k < \ko.
		\end{cases}
		\]
		Using Proposition \ref{prop:weaksol} and  $\eta\theta_--\frac{\theta_+^2-\theta_-^2}{2}=-\eta^2/2$, one obtains that, for $k < \ko$,
		\begin{align*}
			\tilde{Z}_k &= \one_{\{k\neq \ko-1\}}\Big(\eta\sum_{i=k +1}^{\ko-1} \int_0^T X_{\delta,i}^\Delta(t) \diff{B_{\delta,i}(t)} - \frac{\eta^2}{2}\sum_{i=k+1}^{\ko-1}  \int_0^T X_{\delta,i}^\Delta(t)^2 \diff{t}\Big)\\
			&\quad+ \eta \int_0^T X_{\delta,\ko}^\Delta(t)\diff{X_{\delta,\ko}(t)} + \frac{\theta_-^2 - \theta_+^2}{2} \int_0^T X_{\delta,\ko}^\Delta(t)^2\diff{t}\\
			&\,= \eta\sum_{i=k+1}^{\ko} \int_0^T X_{\delta,i}^\Delta(t)\diff{B_{\delta,i}(t)} - \frac{\eta^2}{2}\sum_{i=k+1}^{\ko}  \int_0^T X_{\delta,i}^\Delta(t)^2 \diff{t}\\
			&\quad + \eta\int_0^T X^{\Delta}_{\delta,\ko}(t) \Big(\int_0^t \langle \Delta_{\vartheta} S_{\vartheta}(t-s)K_{\delta,\ko},\diff{W_s} \rangle - \vartheta_-(\delta) X^\Delta_{\delta,\ko}(t)\Big) \diff{t}\\
			&= Z_k.
		\end{align*}
		Similarly, Proposition \ref{prop:weaksol} together with $-\eta\theta_++\frac{\theta_+^2-\theta_-^2}{2}=-\eta^2/2$ yields $\tilde{Z}_k = Z_k$ for $k>\ko$ as well.
	\end{proof}
	
	Finally, we give the proof for tightness, which involves most of the previous technical considerations in the paper.
	\begin{proof}[Proof of Proposition \ref{prop:tight}]
		We will verify that  $\hat{\tau}- \tau ={\mathcal O}_{\PP}(v_\delta)={\mathcal O}_{\PP}(\delta^3\eta^{-2})$.
		As was shown in Section \ref{sec:clt} (cf.~\eqref{eq:min}), we have
		\begin{align*}
			v_\delta^{-1}(\hat\tau-\tau) &\in \argmin_{h \in \intv}\Big\{\eta M_{T,\delta}^\tau(h)+\frac{\eta^2}{2} I_{T,\delta}^\tau(h) + {\mathcal O}_{\PP}(\eta^2\delta^{-2})\one_{\{h < 0\}}\Big\},
		\end{align*}
		where the ${\mathcal O}_{\PP}$-term is independent of $h$ and comes from the expectation bounds from Lemma \ref{lem:basic_est} and Proposition \ref{lem:remainder} on the quantities $I_{\delta,\ceil{\tau/\delta}}, R_{\ceil{\tau/\delta}}$ associated to the observation block around the change point $\tau$.
		Since $\eta = o(\delta)$, it therefore suffices to show that
		\[
		\forall\eps>0\, \exists R_\eps>0:\; \PP\Big(\inf_{\abs{h}>R_\eps, h \in \intv} \Big\{\eta M_{T,\delta}^\tau(h)+\frac {\eta^2}{2} I_{T,\delta}^\tau(h)\Big\} \leq M_{T,\delta}^\tau(0)+\frac \eta2 I_{T,\delta}^\tau(0) +1\Big)\le\eps,
		\]
		with $R_\eps$ only depending on $\eps$, not on $\delta,\eta$.
		To see this, note that
		\begin{align*}
			&\PP\Big(\inf_{\abs{h}>R_\eps, h \in \intv} \Big\{\eta M_{T,\delta}^\tau(h)+\frac {\eta^2}{2} I_{T,\delta}^\tau(h) + \mathcal{O}_{\PP}(\eta^2 \delta^{-2}) \one_{\{h < 0\}}\Big\}\leq M_{T,\delta}^\tau(0)+\frac \eta2 I_{T,\delta}^\tau(0)\Big)\\
			&\,\leq \PP\Big(\inf_{\abs{h}>R_\eps, h \in \intv} \Big\{\eta M_{T,\delta}^\tau(h)+\frac {\eta^2}{2} I_{T,\delta}^\tau(h)\Big\} \leq M_{T,\delta}^\tau(0)+\frac \eta2 I_{T,\delta}^\tau(0) +1\Big)\\
			&\,\quad+ \PP(o_{\PP}(1) < -1) \one_{\{\tau \delta^{-3} \eta^2 > R_{\varepsilon} \}},
		\end{align*}
		and the second term converges to $0$ as $R_{\varepsilon} \to \infty$.
		Since $M_{T,\delta}^\tau(0) = I_{T,\delta}^\tau(0) = 0$, the required statement will follow from
		\[
		\forall\eps>0\,\exists R_\eps>0:\; \PP\Big(\sup_{a^\delta_\pm(h)\geq \pm h>R_\eps} \Big\{-\eta M_{T,\delta}^\tau(h)-\frac{\eta^2}{2} I_{T,\delta}^\tau(h)\Big\}\geq -1\Big)\le\frac\eps2,
		\]
		where we set $a^\delta_+(h)= (1-\tau)/v_\delta$ and  $a^\delta_-(h)= \tau/v_\delta$ and, moreover, use the convention $\sup \varnothing = -\infty$.
		We only consider the case $h > 0$, the case $h < 0$ is similar.
		
		Let $R_{\varepsilon}$ be large enough to ensure that $R_{\eps}v_\delta/\delta > 1$ for any $\delta \in 1/\N$.
		Inserting the definitions \eqref{eq:Mtau} and \eqref{eq:Itau} yields
		\begin{equation} \label{eq:tight0}
			\sup_{a^\delta_+(h) \geq h>R_\eps} \Big\{-\eta M_{T,\delta}^\tau(h)-\frac{\eta^2}{2} I_{T,\delta}^\tau(h)\Big\} = \max_{k=\ceil{(\tau+R_\eps v_\delta)/\delta},\ldots,\delta^{-1}} \sum_{i=\ceil{\tau/\delta}+1}^k\Big( -\eta M_{\delta,i}-\frac{\eta^2}{2} I_{\delta,i}\Big).
		\end{equation}
		Here and in the following, for a vector $b_\delta = (b_\delta(k))_{k = 1,\ldots, \delta ^{-1}}$,  we set $\max_{k = K, \ldots, \delta^{-1}} b_\delta(k) \coloneqq -\infty$ if $K > \delta^{-1}$.
		
		Let the independent coupled random variables $(\bar{M}_{\delta,i})_{i=1,\ldots,\delta^{-1}}$ be given as in  \eqref{eq:coupling} and note that, by Proposition \ref{prop:coupling}, $(\sum_{i= \ko+1}^{k} \bar{M}_{\delta,i})_{k=1,\ldots,\delta^{-1} - \ko}$ is a martingale.
		For $\alpha > 0$, introduce $\bar{N}_{\delta,\alpha}$ given by
		\[\bar N_{\delta,\alpha}(k) \coloneqq \exp\Big(\sum_{i=\ceil{\tau/\delta}+1}^{\ceil{\tau/\delta}+k} \Big(-\alpha \bar M_{\delta,i}-\frac{\alpha^2}{2}\bar I_{\delta,i}\Big)\Big),\quad k \in \{1,\ldots,\delta^{-1}-\ko\},\]
		which is again a martingale since $\overbar{M}_{\delta,i} \sim N(0, \overbar{I}_{\delta,i})$.
		Note that $\bar{N}_{\delta,\eta/2}$ is positive and has constant expectation $1$.
		Thus, for any $\kappa > 0$, Doob's maximal martingale inequality \cite[Proposition II.1.5]{revyor99} yields, for any $k^{\circ} \in \N$,
		\[
		\PP\Big(\max_{k=k^{\circ},\ldots,\delta^{-1}-\ceil{\tau/\delta}} \bar N_{\delta,\eta/2}(k)>\exp\Big(\frac{\eta^2}{8} \sum_{i=\ceil{\tau/\delta}+1}^{\ceil{\tau/\delta}+k^\circ}\bar I_{\delta,i}-\frac{\kappa}{2}\Big)\Big)
		\le \exp\Big(\frac{\kappa}2-\frac{\eta^2}{8} \sum_{i=\ceil{\tau/\delta}+1}^{\ceil{\tau/\delta}+k^\circ}\bar I_{\delta,i}\Big).\]
		Let  $k^{\circ} \coloneqq \ceil{R_\eps\delta^2/\eta^2} -1 \geq 1$.
		It holds
		\begin{align*}
			&\Big\{\max_{k=\ceil{(\tau+R_\eps v_\delta)/\delta},\ldots,\delta^{-1}} \sum_{i=\ceil{\tau/\delta}+1}^k\Big( -\eta\bar M_{\delta,i}-\frac{\eta^2}{2} \bar I_{\delta,i}\Big)>-\kappa\Big\} \\
			&\quad \subset\Big\{\max_{k=k^{\circ},\ldots,\delta^{-1}-\ceil{\tau/\delta}} \bar N_{\delta,\eta/2}(k)>\exp\Big(\frac{\eta^2}{8} \sum_{i=\ceil{\tau/\delta}+1}^{\ceil{\tau/\delta}+k^\circ} \bar I_{\delta,i}-\frac{\kappa}{2}\Big)\Big\}
		\end{align*}
		and we know from \eqref{eq:exp_quad} that
		\[
		\eta^2\sum_{i=\ceil{\tau/\delta}+1}^{\ceil{\tau/\delta}+k^{\circ}} \bar I_{\delta,i}\ge \frac {T}{2\upell} \eta^2\delta^{-2}k^{\circ}\norm{K'}^2+{\mathcal O} (1).
		\]
		Using $\eta^2\delta^{-2}k^{\circ}\thicksim R_\eps\longrightarrow\infty$, we thus conclude that, for any fixed $\kappa>0$ and $\varepsilon > 0$ choosing $R_{\varepsilon} = R_{\varepsilon}(\kappa)$ large enough,
		\begin{equation}\label{eq:tight1}
			\PP\Big(\max_{k=\ceil{(\tau+R_\eps v_\delta)/\delta},\ldots,\delta^{-1}} \sum_{i=\ceil{\tau/\delta}+1}^k\Big( -\eta \bar M_{\delta,i}-\frac{\eta^2}{2} \bar I_{\delta,i}\Big)>-\kappa -1\Big) \leq \frac{\varepsilon}{4}.
		\end{equation}
		We proceed to studying the difference
		$\sum_{i=\ceil{\tau/\delta}+1}^k(\bar M_{\delta,i} - M_{\delta,i} +\frac\eta2( \bar I_{\delta,i}-I_{\delta,i}))$.
		Proposition \ref{prop:coupling} gives, for any $z,L > 0$,
		\[
		\PP\Big( \sum_{i=\ceil{\tau/\delta}+1}^{\ceil{\tau/\delta}+k}(M_{\delta,i}-\bar M_{\delta,i})\ge z,\, \sum_{i=\ceil{\tau/\delta}+1}^{\ceil{\tau/\delta}+k}\abs{I_{\delta,i}-\bar I_{\delta,i}}\leq L\Big) \le \e^{-z^2/(2L)}.
		\]
		For $1 \leq k_1\le k_2\le \delta^{-1} -\ceil{\tau/\delta}$, we deduce
		\begin{align*}
			&\PP\Big( \max_{k=k_1,\ldots,k_2} \sum_{i=\ceil{\tau/\delta}+1}^{\ceil{\tau/\delta}+k}(M_{\delta,i}-\bar M_{\delta,i})\ge z\Big)\\
			&\quad\le \PP\Big( \sum_{i=\ceil{\tau/\delta}+1}^{\ceil{\tau/\delta}+k_2}\abs{I_{\delta,i}-\bar I_{\delta,i}}> L\Big)\\
			&\quad\quad+ \sum_{k=k_1}^{k_2} \PP\Big( \sum_{i=\ceil{\tau/\delta}+1}^{\ceil{\tau/\delta} + k}(M_{\delta,i}-\bar M_{\delta,i})\ge z,\,
			\sum_{i=\ceil{\tau/\delta}+1}^{\ceil{\tau/\delta}+k}\abs{I_{\delta,i}-\bar I_{\delta,i}}\le L\Big)\\
			&\quad\le \PP\Big( \sum_{i=\ceil{\tau/\delta}+1}^{\ceil{\tau/\delta}+k_2}\abs{I_{\delta,i}-\bar I_{\delta,i}}> L\Big)+ (k_2-k_1+1)\e^{-z^2/(2L)}.
		\end{align*}
		The inequalities $\Var( I_{\delta,i})\lesssim  T \delta^{-2}$, $(\sum_{i=1}^k\alpha_i)^2\le k\sum_{i=1}^k\alpha_i^2$ and Markov's inequality show that
		\[
		\PP\Big( \sum_{i=\ceil{\tau/\delta}+1}^{\ceil{\tau/\delta}+k_2} \abs{I_{\delta,i}-\bar I_{\delta,i}}> L\Big)\le
		\PP\Big( k_2\sum_{i=\ceil{\tau/\delta}+1}^{\ceil{\tau/\delta}+k_2} (I_{\delta,i}-\E[ I_{\delta,i}])^2> L^2\Big)\lesssim \frac{k_2^2T\delta^{-2}}{L^2}.
		\]
		The term in the upper bound tends to zero for $L=R_L k_2\sqrt T \delta^{-1}$ with $R_L\to\infty$.
		Hence, with this choice of $L$, we find for \begin{equation}\label{eq:z_L}
			z_L=R_L\sqrt{k_2\delta^{-1}\log(\delta^{-1})}
		\end{equation}
		that
		\begin{align*}
			& \PP\Big(\max_{k=k_1,\ldots,k_2}  \sum_{i=\ceil{\tau/\delta}+1}^{\ceil{\tau/\delta}+k}(\bar M_{\delta,i}- M_{\delta,i} +\tfrac\eta2(\bar I_{\delta,i}- I_{\delta,i}))\ge 2z_L+\eta L\Big)\\
			&\lesssim \delta^{-1}\e^{-z_L^2\delta/(2R_Lk_2\sqrt T)}+2R_L^{-2} \underset{R_L \to \infty}{\longrightarrow} 0.
		\end{align*}
		Noting $\eta L\leq \eta\sqrt T R_Lk_2^{1/2} \delta^{-3/2}=o(z_L)$ due to $\eta=o(\delta)$, we conclude
		\begin{equation}\label{eq:van}
			\PP\Big(\max_{k=k_1,\ldots,k_2}  \sum_{i=\ceil{\tau/\delta}+1}^{\ceil{\tau/\delta}+k}(\bar M_{\delta,i} - M_{\delta,i} +\tfrac\eta2( \bar I_{\delta,i}- I_{\delta,i}))\ge 3z_L\Big)\underset{R_L \to \infty}{\longrightarrow} 0.
		\end{equation}
		By our assumption $\eta(1/n) = o(1/n)$, there are (possibly empty) subsequences $(n^{(i)}_k)_{k \in \N}$, $i \in \{1,2\}$, such that $\{n^{(1)}_k: k \in \N\} \cup \{n^{(2)}_k: k \in \N\} = \N$ and
		\[\limsup_{k \to \infty} \eta(1/n^{(1)}_k) \sqrt{\log n_{k}^{(1)}}/n_k^{(1)} = 0, \quad \liminf_{k \to \infty} \eta(1/n^{(2)}_k) \sqrt{\log n_{k}^{(2)}}/n_k^{(2)} > 0,\]
		while still $\eta(1/n^{(i)}_k) = o(1/n^{(i)}_k)$ for $i \in \{1,2\}$, Thus, by analysing along these subsequences if necessary, it is sufficient to consider the two following cases: (a) $\eta=o(\delta/\sqrt{\log(\delta^{-1})})$, and (b) $\eta\gtrsim \delta/\sqrt{\log(\delta^{-1})}$, while still $\eta=o(\delta)$.
		
		\paragraph*{Case (a):}
		Choose $k_1=1$ and $k_2=\delta^{-1}-\ceil{\tau/\delta}$ maximally.
		Combining this with the definition of $z_L$ in \eqref{eq:z_L}, we obtain $z_L \leq R_L\sqrt{\log(\delta^{-1})}\delta^{-1} \leq cR_L \eta^{-1}$ for some $c > 0$.
		Hence, \eqref{eq:van} implies
		\[ \PP\Big(\max_{k=1,\ldots,\delta^{-1} - \ceil{\tau/\delta}}  \sum_{i=\ceil{\tau/\delta}+1}^{\ceil{\tau/\delta}+k} \eta(\bar M_{\delta,i} -  M_{\delta,i}+\tfrac\eta2( \bar I_{\delta,i}- I_{\delta,i}))\ge 3cR_L\Big)\underset{R_L \to \infty}{\longrightarrow} 0.
		\]
		Thus, for any $\varepsilon > 0$, there exists $\kappa = \kappa(\eps) > 1$ such that, for any $\delta \in 1/\N$,
		\[ \PP\Big(\max_{k=1,\ldots,\delta^{-1} - \ceil{\tau/\delta}}  \sum_{i=\ceil{\tau/\delta}+1}^{\ceil{\tau/\delta}+k}  \eta(\bar M_{\delta,i}- M_{\delta,i} +\tfrac\eta2( \bar I_{\delta,i}- I_{\delta,i}))\ge \kappa\Big)\leq \frac\eps4.
		\]
		Using \eqref{eq:tight0} and \eqref{eq:tight1}, we therefore obtain for $R_\eps = R_\eps(\kappa)$ large enough
		\begin{align*}
			&\PP\Big(\sup_{ h>R_\eps} \Big\{-\eta M_{T,\delta}^\tau(h)-\frac{\eta^2}{2} I_{T,\delta}^\tau(h)\Big\} \geq - 1\Big)\\
			&\,\leq \PP\Big(\max_{k=\ceil{(\tau+R_\eps v_\delta)/\delta},\ldots,\delta^{-1}} \sum_{i=\ceil{\tau/\delta}+1}^k\Big( -\eta \bar M_{\delta,i}-\frac{\eta^2}{2} \bar I_{\delta,i}\Big) > -\kappa -1 \Big) \\
			&\,\quad + \PP\Big(\max_{k=1,\ldots,\delta^{-1} - \ceil{\tau/\delta}}  \sum_{i=\ceil{\tau/\delta}+1}^{\ceil{\tau/\delta}+k}   \eta(\bar M_{\delta,i} -  M_{\delta,i} +\tfrac\eta2( \bar I_{\delta,i}- I_{\delta,i}))\ge \kappa\Big) \\
			&\,\leq \eps/2.
		\end{align*}
		
		\paragraph*{Case (b):}
		Choose $k_1 = 1$ and $k_2= \min\{\lfloor R_L\delta^{-1}/\log(\delta^{-1}) \rfloor, \delta^{-1}-\ko\}$.
		Then, $z_L \leq R_L^{3/2}\delta^{-1} \leq cR_L^{3/2}\eta^{-1}$ for some $c > 0$ thanks to $\eta = o(\delta)$.
		From \eqref{eq:van}, we thus obtain
		\[ \PP\Big(\max_{k= k_1,\ldots,k_2}  \sum_{i=\ceil{\tau/\delta}+1}^{\ceil{\tau/\delta}+k}\eta ( M_{\delta,i}-\bar M_{\delta,i}-\tfrac\eta2( I_{\delta,i}-\bar I_{\delta,i}))\ge 3cR_L^{3/2}\Big)\underset{R_L\to \infty}{\longrightarrow} 0.
		\]
		Consequently, arguing as before, given $\varepsilon > 0$, choosing $R_L > 0$ and $R_\eps=R_{\varepsilon}(R_L^{3/2})$ large enough yields
		\begin{align*}
			&\PP\Big(\sup_{ h>R_\eps} \Big\{-\eta M_{T,\delta}^\tau(h)-\frac{\eta^2}{2} I_{T,\delta}^\tau(h)\Big\} \geq - 1\Big)\\
			&\,\leq \PP\Big(\max_{k=\ceil{(\tau+R_\eps v_\delta)/\delta},\ldots,\delta^{-1}} \sum_{i=\ko+1}^k\Big( -\eta \bar M_{\delta,i}-\frac{\eta^2}{2} \bar I_{\delta,i}\Big) > -3cR_L^{3/2}-1 \Big) \\
			&\,\quad + \PP\Big(\max_{k=1,\ldots,k_2}  \sum_{i=\ko+1}^{\ko+k}   \eta(\bar M_{\delta,i} -  M_{\delta,i} +\tfrac\eta2( \bar I_{\delta,i}- I_{\delta,i}))\ge 3cR_L^{3/2}\Big) \\
			&\, \quad + \PP\Big(\max_{k=k_2+1,\ldots,\delta^{-1}}  \sum_{i=\ko+1}^{\ko+k}\eta(-M_{\delta,i}-\tfrac\eta2I_{\delta,i})\ge -1\Big) \one_{\{k_2 < \delta^{-1}-\ko\}}\\
			& \leq  \eps/2 + \PP\Big(\max_{k=k_2+1,\ldots,\delta^{-1}}  \sum_{i=\ko+1}^{\ko+k} \eta(-M_{\delta,i}-\tfrac\eta2I_{\delta,i})\ge -1\Big)\one_{\{k_2 < \delta^{-1}-\ko\}}.
		\end{align*}
		It remains to show that the second term becomes small for any $\delta \in 1/\N$  as $R_L \to \infty$. Assume $k_2 < \delta^{-1} - \ko$.
		Using a union bound, we obtain directly via Girsanov's theorem, for any $L^\prime > 0$,
		\begin{align}\nonumber
			&\PP\Big(\max_{k=k_2+1,\ldots,\delta^{-1}}  \sum_{i=\ko+1}^{\ko+k}\eta(-M_{\delta,i}-\frac\eta2I_{\delta,i})\ge -1\Big)\\\nonumber
			&\le \PP\Big(\frac{\eta^2}{8}\sum_{i=\ko+1}^{\ko+k_2+1}I_{\delta,i}\le L' + 1\Big)+\sum_{k=k_2+1}^{\delta^{-1}} \PP\Big(\exp\Big( \sum_{i=\ko+1}^{\ko+k} (-\frac\eta2M_{\delta,i}-\frac{\eta^2}8I_{\delta,i})
			\Big)\ge \e^{L^\prime}\Big)\\\label{eq:tight2}
			&\le \PP\Big(\frac{\eta^2}{8}\sum_{i=\ko+1}^{\ko+ k_2+1}I_{\delta,i}< L' +1\Big)+
			\delta^{-1}\e^{-L'}.
		\end{align}
		Yet, in order to apply Girsanov, we have to check the Novikov condition
		\[
		\E\Big[\exp\Big(\frac{\eta^2}8 \sum_{i=\ko+1}^{\delta^{-1}} I_{\delta,i}\Big)\Big]<\infty.
		\]
		From Proposition \ref{prop:conc}, we know for some $c>0$, independent of $\delta$,
		\[ \PP\Big(\frac{\eta^2}8 \sum_{i=\ko+1}^{\delta^{-1}} (I_{\delta,i}-\E[I_{\delta,i}])\ge z\Big)\le \exp\Big(-c\eta^{-4}z^2/(\eta^{-1}z+\delta^{-3})\Big)=\e^{-cz^2/(\eta^3z+\eta^4\delta^{-3})}.
		\]
		Since $\eta = o(\delta)$, it therefore follows that, for $\delta$ sufficiently small, the right hand side is $o(\mathrm{e}^{-2z})$, whence, $\tfrac{\eta^2}8 \sum_{i=\ceil{\tau/\delta}+1}^{\delta^{-1}} I_{\delta,i}$ has an exponential moment of order $1$, as needed.
		Having verified the validity of \eqref{eq:tight2}, it remains to bound the first term.
		By Lemma \ref{lem:basic_est}, we have
		\[ \E\Big[\eta^2\sum_{i=\ko+1}^{\ko+k_2+1} I_{\delta,i}\Big]\ge \frac T2 \eta^2 \delta^{-2}(k_2 +1)\norm{K'}^2\lowell+{\mathcal O}(1)\thicksim R_L\eta^2\delta^{-3}/\log(\delta^{-1})
		\]
		and
		\[ \Var\Big(\eta^2\sum_{i=\ko+1}^{\ko+k_2+1} I_{\delta,i}\Big) \lesssim  T \eta^4 \delta^{-2}(k_2+1) \thicksim R_L \eta^4 \delta^{-3}/\log(\delta^{-1}).
		\]
		Hence,  choosing $L' \thicksim R_L\eta^2\delta^{-3}/\log(\delta^{-1})$ in case $k_2 < \delta^{-1} - \ko$, Chebyshev's inequality yields
		\[
		\PP\Big(\frac{\eta^2}{8} \sum_{i= \ko +1}^{\ko+k_2+1} I_{\delta,i} < L^\prime +1 \Big) \lesssim \delta^3 R_L^{-1} \log(\delta^{-1}) \lesssim \e^{-R_L} R_L^{-1},\]
		where we used that $k_2 < \delta^{-1} - \ko$ implies $R_L < \log(\delta^{-1}) +1$ and hence $\delta \lesssim \e^{-R_L}$. Hence, by \eqref{eq:tight2}, using also the fact that in our case (b), $L^\prime \gtrsim R_L \delta^{-1}/(\log(\delta^{-1}))^2$, it follows
		\begin{align*}
			&\PP\Big(\max_{k=k_2+1,\ldots,\delta^{-1}}  \sum_{i=\ko+1}^{\ko+k}\eta(-M_{\delta,i}-\frac\eta2I_{\delta,i})\ge -1\Big)\one_{\{k_2 < \delta^{-1}-\ko\}}\\
			 &\quad\leq \mathrm{e}^{-R_L}R_L^{-1} + \delta^{-1} \mathrm{e}^{-R_L \delta^{-1}/(\log(\delta^{-1}))^2},
		\end{align*}
		and the right hand side converges to $0$ uniformly over $\delta \in 1\slash \N$ as $R_L \to \infty$, as desired.
		Putting everything together, we have proved tightness.
	\end{proof}
\end{appendix}

\paragraph*{Acknowledgements}
MR is grateful for financial funding
by Deutsche Forschungsgemeinschaft (DFG) - SFB1294/2 - 318763901. CS and LT gratefully acknowledge financial support of Carlsberg Foundation Young Researcher Fellowship grant CF20-0640 ``Exploring the potential of nonparametric modelling of complex systems via SPDEs''.

\printbibliography

\end{document}